\documentclass
  [ english
  , 11pt
  , a4paper
  , reqno
  ]{amsart}

\author{Fosco Loregian$^\ddag$}
\address{ 
  \noindent $^\ddag$Tallinn University of Technology,\newline %
  Institute of Cybernetics, Akadeemia tee 15/2, \newline %
  12618 Tallinn, Estonia \newline
  \url{fosco.loregian@taltech.ee}
}

\author{Théo de Oliveira Santos$^\S$}
\address{%
    \noindent $^\S$Universidade de São Paulo,\newline
    Instituto de Ciências Matemáticas e de Computação,\newline
    Av.\ Trab.\ São Carlense, 400,\newline
    13566-590 São Carlos, Brasil\newline
    \url{theo.de.oliveira.santos@usp.br}
}

\thanks{This research was supported by the ESF funded Estonian IT Academy research measure (project 2014-2020.4.05.19-0001). The first author would like to thank A. Santamaria and his delightful talk at ItaCa \cite{alessiotalk}, without which this paper would probably not exist, the entire Italian community of category theorists that has made ItaCa possible, and \emph{sensei} Inoue Takehiko.}
\thanks{%
    The second author is greatly indebted to many people for their immense generosity and kindness, including their parents, friends, as well as
    Jonathan Beardsley,
    Lennart  Meier,
    Igor     Mencattini,
    Oziride  Manzoli Neto,
    and
    Eric     Peterson,
    among so many others who I've had the pleasure of knowing. The second author was supported by grant \#2020/02861-7, São Paulo Research Foundation (FAPESP).
}
\title{Coends of Higher Arity}

\usepackage
  [ includehead
  , top=4cm
  , bottom=5cm
  , left=3cm
  , right=3cm
  ]{geometry}
\usepackage[table]{xcolor}

\usepackage{commutative-diagrams}

\usepackage
  { amssymb
  , amsmath
  , adjustbox
  , amsthm
  , booktabs
  , lmodern
  , mathtools
  , tikz-cd
  , todonotes
  , xparse
  , proof
  , wrapfig
  , hyphenat
  , xspace
  , ytableau
  , cancel
  , xfrac 
  , enumitem
  , etoolbox
  , CJKutf8
  , commutative-diagrams
}

\usepackage[english]{babel}
\usepackage[utf8]{inputenc}
\usepackage[T1]{fontenc}

\usetikzlibrary{decorations.pathmorphing}
\ytableausetup{smalltableaux,centertableaux}

\tikzcdset{
arrow style=tikz,
diagrams={>={Stealth[round,length=4pt,width=4.95pt,inset=2.75pt]}}
}
\tikzset{double line with arrow/.style args={#1,#2}{decorate,decoration={markings,%
					mark=at position 0 with {\coordinate (ta-base-1) at (0,1pt);
							\coordinate (ta-base-2) at (0,-1pt);},
					mark=at position 1 with {\draw[#1] (ta-base-1) -- (0,1pt);
							\draw[#2] (ta-base-2) -- (0,-1pt);
						}}}}
\tikzset{ Equals/.style={-,double line with arrow={-,-}}
        , hypercube/.style={inner sep=.25pt,font=\scriptsize}}

\newcommand{\tk}[1]{\begin{tikzcd}[ampersand replacement=\&]
		#1
	\end{tikzcd}}

\newcommand{\xlongdashrightarrow}[1]{%
	\mathrel{%
		\mkern-15.25mu
		\begin{tikzcd}[row sep=2.7em, column sep=1.8em, ampersand replacement=\&]
			{}
			\arrow[r, dashed, "#1"] \&
			{}
		\end{tikzcd}
		\mkern-16.25mu
	}
}
\usetikzlibrary{decorations.markings}
\tikzset{mid vert/.style={/utils/exec=\tikzset{every node/.append style={outer sep=0.8ex}},
			postaction=decorate,decoration={markings,
					mark=at position 0.5 with {\draw[-] (0,#1) -- (0,-#1);}}},
	mid vert/.default=0.75ex}

\let\pto\proarrow

\let\longrightrightarrows\relax
\newcommand{\longrightrightarrows}{%
	\begin{tikzcd}[row sep=2.7em, column sep=1.35em, ampersand replacement=\&,cramped,{>={Stealth[round,length=4pt,width=4.95pt,inset=2.75pt]}}]
		{}
		\arrow[r, shift left=0.625]
		\arrow[r, shift right=0.625] \&
		{}
	\end{tikzcd}
}

\let\xto\xlongrightarrow
\newcommand{\xloongtwoheadsrightarrow}[1]{%
	\mathrel{%
		\mkern-15.25mu
		\begin{tikzcd}[row sep=3.6em, column sep=2.7em, ampersand replacement=\&]
			{}
			\arrow[r, "#1",two heads]
			\&
			{}
		\end{tikzcd}
		\mkern-16.25mu
	}
}
\newcommand{\xloongrightarrow}[1]{%
	\mathrel{%
		\mkern-15.25mu
		\begin{tikzcd}[row sep=3.6em, column sep=2.7em, ampersand replacement=\&]
			{}
			\arrow[r, "#1"] \&
			{}
		\end{tikzcd}
		\mkern-16.25mu
	}
}

\newcommand{\longtwoheadsrightarrow}{%
	\mathrel{%
		\mkern-15.25mu
		\begin{tikzcd}[row sep=2.7em, column sep=1.35em, ampersand replacement=\&]
			{}
			\arrow[r, two heads] \&
			{}
		\end{tikzcd}
		\mkern-16.25mu
	}
}
\newcommand{\longhookrightarrow}{%
	\mathrel{%
		\mkern-15.25mu
		\begin{tikzcd}[row sep=2.7em, column sep=1.35em, ampersand replacement=\&]
			{}
			\arrow[r, hook] \&
			{}
		\end{tikzcd}
		\mkern-16.25mu
	}
}

\newcommand{\din}{%
	\mathrel{%
		\begin{tikzpicture}[baseline=-0.58ex]
			\draw[line width=.425pt] (0,.035) -- ++(.45,0);
			\draw[line width=.425pt] (0,-.035) -- ++(.45,0);
			\draw[fill=white] (.45,.0) circle (2.25pt);
		\end{tikzpicture}
	}
}

\tikzset{
	din/.style={
			double distance=0.2em,
			decoration={
					markings,
					mark={
							at position #1
							with {
									\draw[fill=backgroundColor] circle [radius=2pt];
								}
						}
				},
			postaction=decorate
		},
	double,
	din/.default=0em
}

\newcommand{\xlongertwoheadrightarrow}[1]{%
	\mathrel{%
		\mkern-15.25mu
		\begin{tikzcd}[row sep=3.6em, column sep=1.8em, ampersand replacement=\&]
			{}
			\arrow[r, "#1", two heads] \&
			{}
		\end{tikzcd}
		\mkern-16.25mu
	}
}
\let\xlongertwoheadsrightarrow\xlongertwoheadrightarrow

\DeclareRobustCommand{\longrightleftrightarrows}{%
	\mathrel{%
		\mkern-15.25mu
		\begin{tikzcd}[row sep=2.7em, column sep=1.35em, ampersand replacement=\&]
			{}
			\arrow[r, shift left=1.5]
			\arrow[r, leftarrow]
			\arrow[r, shift right=1.5]
			\&
			{}
		\end{tikzcd}
		\mkern-16.25mu
	}
}

\definecolor{FancyRed}{RGB}{199,10,41}
\definecolor{FancyGreen}{RGB}{36,135,48}
\definecolor{FancyBlue}{RGB}{0,22,166}
\definecolor{FancyDarkBlue}{RGB}{95,30,196}
\definecolor{veryLightGray}{RGB}{240,240,240}

\definecolor{LightRed}{rgb}{0.980,0.407,0.400}
\definecolor{backgroundColor}{rgb}{1,1,1}
\definecolor{linkBrownRed}{rgb}{0.6,0.12156862745,0}
\definecolor{lightGray}{rgb}{0.8,0.8,0.8}
\definecolor{grayY}{rgb}{0.95,0.95,0.95}
\definecolor{grayN}{rgb}{1,1,1}
\definecolor{OIvermillion}{RGB}{213,94,0}
\definecolor{OIblue}{RGB}{0,114,178}

\input{preamble/widebar.tex}
\usepackage{inconsolata}
\let\mathrm\relax
\newcommand{\mathrm}[1]{\text{#1}}
\let\lim\relax
\DeclareMathOperator*{\lim}{\mathrm{lim}}
\DeclareMathOperator*{\colim}{\mathrm{colim}}

\usepackage
[ maxbibnames=20
	, backend=biber
	, backref=true
	, style=alphabetic
	, citestyle=alphabetic
]{biblatex}

\usepackage{hyperref}
\hypersetup{
	colorlinks,
	citecolor=linkBrownRed,
	filecolor=linkBrownRed,
	linkcolor=linkBrownRed,
	urlcolor=linkBrownRed
}

\usepackage[capitalise,nameinlink,noabbrev]{cleveref}
\renewcommand{\cref}[1]{\Cref{#1}}
\makeatletter
\def\defthm#1#2{%
	\newtheorem{#1}{#2}[section]%
	\expandafter\def\csname #1autorefname\endcsname{#2}%
	\expandafter\let\csname c@#1\endcsname\c@theorem}
\makeatother

\theoremstyle{definition}
\numberwithin{equation}{section}
\newtheorem{theorem}{Theorem}[section]
\defthm{corollary}{Corollary}
\defthm{proposition}{Proposition}
\defthm{lemma}{Lemma}
\defthm{claim}{Claim}
\defthm{definition}{Definition}
\defthm{notation}{Notation}
\defthm{remark}{Remark}
\defthm{conjecture}{Conjecture}
\defthm{axiom}{Axiom}
\defthm{example}{Example}
\defthm{non-example}{Non-Example}
\defthm{examples}{Examples}
\defthm{exercise}{Exercise}
\defthm{counterexample}{Counterexample}
\defthm{construction}{Construction}
\defthm{warning}{Warning}
\defthm{digression}{Digression}
\defthm{perspective}{Perspective}
\defthm{discussion}{Discussion}
\defthm{terminology}{Terminology}
\defthm{heuristics}{Heuristics}

\newtheorem*{definition*}{Definition}
\newtheorem*{question*}{Question}

\ExplSyntaxOn
\NewDocumentCommand{\makeabbrev}{mmm}
{
	\yoruk_makeabbrev:nnn { #1 } { #2 } { #3 }
}

\cs_new_protected:Npn \yoruk_makeabbrev:nnn #1 #2 #3
{
	\clist_map_inline:nn { #3 }
	{
		\cs_new_protected:cpn { #2 } { #1 { ##1 } }
	}
}
\ExplSyntaxOff

\makeabbrev{\textbf}{bf#1}{
	a,b,c,d,e,f,g,h,i,j,k,l,m,n,o,p,q,r,t,u,v,w,x,y,z,%
	A,B,C,D,E,F,G,H,I,J,K,L,M,N,O,P,Q,R,T,U,V,W,X,Y,Z }
\makeabbrev{\boldsymbol}{bs#1}{%
	a,b,c,d,e,f,g,h,i,j,k,l,m,n,o,p,q,r,s,t,u,v,w,x,y,z,%
	A,B,C,D,E,F,G,H,I,J,K,L,M,N,O,P,Q,R,S,T,U,V,W,X,Y,Z }
\makeabbrev{\mathsf}{sf#1}{
	a,b,c,d,e,f,g,h,i,j,k,l,m,n,o,p,q,r,s,t,u,v,w,x,y,z,%
	A,B,C,D,E,F,G,H,I,J,K,L,M,N,O,P,Q,R,S,T,U,V,W,X,Y,Z }
\makeabbrev{\mathfrak}{fk#1}{
	a,b,c,d,e,f,g,h,j,k,i,l,m,n,o,p,q,r,s,t,u,v,w,x,y,z,%
	A,B,C,D,E,F,G,H,I,J,K,L,M,N,O,P,Q,R,S,T,U,V,W,X,Y,Z }
\makeabbrev{\mathcal}{cl#1}{
	A,B,C,D,E,F,G,H,I,J,K,L,M,N,O,P,Q,R,S,T,U,V,W,X,Y,Z }
\makeabbrev{\mathbb}{bb#1}{
	A,B,C,D,E,F,G,H,I,J,K,L,M,N,O,P,Q,R,S,T,U,V,W,X,Y,Z }
\makeabbrev{\underline}{u#1}{
	a,b,c,d,e,f,g,h,j,k,i,l,m,n,o,p,q,r,s,t,u,v,w,x,y,z,%
	A,B,C,D,E,F,G,H,I,J,K,L,M,N,O,P,Q,R,S,T,U,V,W,X,Y,Z }
\input{preamble/diagram_environment.tex}
\tikzset{
	partial ellipse/.style args={#1:#2:#3}{
			insert path={+ (#1:#3) arc (#1:#2:#3)}
		}
}
\definecolor{kusarigamaBrown}{RGB}{140,117,59}

\let\kg\kusarigama
\newcommand{\kgb}[1]{\Gamma\big(#1\big)}

\let\ckg\coKusarigama
\newcommand{\ckgb}[1]{
	\mathchoice%
	{\rotatebox[origin=c]{180}{$\Gamma$}\big(#1\big)}
	{\rotatebox[origin=c]{180}{$\Gamma$}\big(#1\big)}
	{\rotatebox[origin=c]{180}{$\scriptstyle\Gamma$}\big(#1\big)}
	{\rotatebox[origin=c]{180}{$\scriptscriptstyle\Gamma$}\big(#1\big)}
}

\newcommand{\kgpq}[3]{\Gamma^{#1,#2}(#3)}

\newcommand{\ckgpq}[3]{
	\mathchoice%
	{\rotatebox[origin=c]{180}{$\Gamma$}^{#1,#2}(#3)}
	{\rotatebox[origin=c]{180}{$\Gamma$}^{#1,#2}(#3)}
	{\rotatebox[origin=c]{180}{$\scriptstyle\Gamma$}^{#1,#2}(#3)}
	{\rotatebox[origin=c]{180}{$\scriptscriptstyle\Gamma$}^{#1,#2}(#3)}
}

\newcommand{\kgfpq}[2]{\Gamma^{#1,#2}}
\newcommand{\ckgfpq}[2]{\rotatebox[origin=c]{180}{$\Gamma$}^{#1,#2}}
\newcommand{\kgf}{\Gamma}
\newcommand{\ckgf}{\rotatebox[origin=c]{180}{$\Gamma$}}

\newcommand{\wlim}[1]{\mathrm{lim}^{#1}}
\newcommand{\wcolim}[1]{\mathrm{colim}^{#1}}

\newcommand{\ph}{\mathsf{h}}
\newcommand{\pqdiag}{\Delta_{p,q}}

\def\jap#1{\begin{CJK*}{UTF8}{bsmi}#1\end{CJK*}}

\newcommand{\CatFont}[1]{\mathcal{#1}}
\newcommand{\Hom}{\mathrm{hom}}%
\newcommand{\SloganFont}[1]{{\textit{#1. }}}

\newcommand{\Wedges}[2]{\mathsf{Wd}_{#1}\big(#2\big)}
\newcommand{\Cowedges}[2]{\mathsf{CWd}_{#1}\big(#2\big)}
\newcommand{\Mor}{\mathrm{Mor}}%

\newcommand{\otimesDay}{\mathbin{\circledast}}
\DeclareRobustCommand{\otimesV}{\mathbin{\boxtimes_{\CatFont{V}}}}

\newcommand{\nin}{\not{\in}}
\newcommand{\Open}{\mathsf{Op}}

\newcommand{\Tot}{\mathrm{Tot}}

\newcommand{\Cech}{\textsf{Č}_{\bullet}}
\newcommand{\pr}{\bst}
\newcommand{\projection}{\mathrm{pr}}
\newcommand{\inj}{[\boldsymbol{0}]}

\let\ceil\ceiling

\newcommand{\ConstantPNo}[2]{\boldsymbol{#1}}

\definecolor{darkGreen}{RGB}{2,127,27}

\newcommand{\mrp}[1]{\mathrlap{#1}}

\DeclareMathOperator*{\prodbi}{\sfW}

\newcommand{\negphantom}[1]{\settowidth{\dimen0}{#1}\hspace*{-\dimen0}}

\newcommand{\pt}{\mathrm{pt}}

\newcommand{\Tw}[1]{\mathsf{Tw}(#1)}
\newcommand{\Twckg}[3]{\mathsf{Tw}^{#1,#2}_{\rotatebox[origin=c]{180}{$\scriptstyle\Gamma$}}(#3)}
\newcommand{\Forgetful}{\Sigma}

\newcommand{\otimesDayN}[1]{\mathbin{\circledast}_{#1}}
\newcommand{\SimplexCategory}{\boldsymbol{\Delta}}

\newcommand{\ev}{\mathsf{ev}}
\newcommand{\Lan}{\mathsf{Lan}}
\newcommand{\Ran}{\mathsf{Ran}}
\newcommand{\M}{\mathsf{M}}
\newcommand{\W}{\mathsf{W}}

\newcommand{\DiLan}{\mathsf{DiLan}}
\newcommand{\DiRan}{\mathsf{DiRan}}

\newcommand{\pqdashv}[2]{\dashv}

\newcommand{\Unit}{\mathbf{1}}

\DeclareMathAlphabet{\dutchcal}{U}{dutchcal}{m}{n}
\SetMathAlphabet{\dutchcal}{bold}{U}{dutchcal}{b}{n}
\DeclareMathAlphabet{\dutchbcal}{U}{dutchcal}{b}{n}
\newcommand{\SheafFont}[1]{\dutchcal{#1}}

\newcommand{\say}[1]{``#1''}
\newcommand{\DayOperad}{\mathsf{Day}}

\newcommand{\id}{\mathrm{id}}%
\newcommand{\defeq}{\overset{\scriptscriptstyle\mathrm{def}}=}
\newcommand{\eHom}{\textbf{hom}}
\newcommand{\op}{\mathsf{op}}
\newcommand{\N}{\mathbb{N}}
\newcommand{\VNat}[1]{\mathbf{Nat}_{#1}}
\newcommand{\catpt}{\mathsf{pt}}
\newcommand{\typepq}[2]{\left[\pqMat{#1\\#2}\right]}

\NewDocumentCommand{\dummy}{O{r} O{s}}{
	\text{ð}^{#1}_{\kern-.1em #2}
}
\newcommand{\dummyF}{\text{ð}}

\newcommand{\Fun}{\Cats}
\newcommand{\PSh}[1]{\textsf{\upshape PSh}(#1)}
\newcommand{\PShf}{\textsf{\upshape PSh}}
\newcommand{\Shv}[1]{\textsf{\upshape Shv}(#1)}

\makeatletter
\DeclareRobustCommand{\evdots}{
	\vbox{\baselineskip4\p@\lineskiplimit\z@\kern0\p@\hbox{.}\hbox{.}\hbox{.}}}
\makeatother
\newcommand{\cate}[1]{\mathsf{#1}}
\newcommand{\Sets}{\cate{Set}}

\newcommand{\Cat}{\cate{Cat}}
\let\Cats\Cat
\newcommand{\Prof}{\cate{Prof}}
\let\Set\Sets

\setuptodonotes{fancyline, color=blue!10}

\DeclareDocumentEnvironment{enumtag}{m}{%
	\begin{enumerate}%
		[ label = \textsc{#1}\oldstylenums{\arabic*}),%
			ref   = \textsc{#1}\oldstylenums{\arabic*}%
		] }{ \end{enumerate} }

\newlength{\OneCm}
\newlength{\OneCmAndAHalf}
\newcommand{\WeightedEnd}[2]{\int_{#1}^{#2}}
\newcommand{\WeightedCoend}[2]{\int^{#1}_{#2}}
\newcommand{\DiNat}{\mathrm{DiNat}}
\newcommand{\one}{\mathrm{pt}}

\newcommand{\VDiNat}[1]{{\bf \DiNat}_{#1}}

\newcommand{\pqDiNat}[2]{\DiNat^{(#1,#2)}}

\newcommand{\Nat}{\textup{Nat}}

\def\Wd{\mathsf{Wd}}
\def\CWd{\mathsf{CWd}}
\def\catWd{\mathsf{Wd}}
\newcommand{\pqWedges}[4]{\Wd^{(#1,#2)}_{#3}(#4)}
\newcommand{\pqCoWedges}[4]{\CWd^{(#1,#2)}_{#3}(#4)}
\newcommand{\pqWedgesFunctor}[3]{\Wd^{(#1,#2)}_{#3}}
\newcommand{\pqCoWedgesFunctor}[3]{\CWd^{(#1,#2)}_{#3}}
\newcommand{\pqEnd}[3]{\mathop{\prescript{}{\Scale[.75]{\raisebox{.25em}{$(#1,#2)$}}}{\int_{#3}}}}
\newcommand{\pqTw}[3]{\mathsf{Tw}^{(#1,#2)}(#3)}

\newcommand{\pqCoend}[3]{%
	\mathchoice
	{\mathop{\prescript{\Scale[.75]{\raisebox{-.25em}{$(#1,#2)$}}\kern-.625em}{}{\int^{#3}}}}
	{\mathop{\prescript{\Scale[.75]{\raisebox{-.25em}{$(#1,#2)$}}\kern-.25em}{}{\int^{#3}}}}
	{\mathop{\prescript{\Scale[.75]{\raisebox{-.25em}{$(#1,#2)$}}\kern-.25em}{}{\int^{#3}}}}
	{\mathop{\prescript{\Scale[.75]{\raisebox{-.25em}{$(#1,#2)$}}\kern-.25em}{}{\int^{#3}}}}
}

\NewDocumentCommand{\tpl}{m O{1} O{n}}{
	#1_{#2},\dots,#1_{#3}
}

\NewDocumentCommand{\pq}{m O{p} O{q}}{
	{#1}^{(#2,#3)}
}

\newcommand*{\Scale}[2][4]{\scalebox{#1}{\ensuremath{#2}}}%

\newcommand{\pqMat}[1]{
	{\Scale[.925]{\begin{smallmatrix}
					#1 
				\end{smallmatrix}}}
}

\newcommand{\oo}[1]{{#1}^\op\times {#1}}

\newenvironment{xsmallmatrix}[1]
{\renewcommand\thickspace{\kern#1}\smallmatrix}
{\endsmallmatrix}
\NewDocumentCommand{\var}{o m m}{
	\IfNoValueTF{#1}{
		\left[
			\begin{smallmatrix}
				#2 \\
				\downarrow \\
				#3
			\end{smallmatrix}\right]}
	{
		\left[
			\begin{xsmallmatrix}{0em}
				& #2 \\
				#1 & \downarrow \\
				& #3
			\end{xsmallmatrix}\right]}}

\newcommand{\subst}[3]{{{\boldsymbol{#1}}[#2/#3}]} 
\newcommand{\substMV}[4]{{{\boldsymbol{#1}}[#2,#3 / #4]}} %

\renewcommand{\ell}[1]{\lvert #1 \rvert}
\makeatletter
\displaywidth=\textwidth
\displayindent=-\leftskip
\patchcmd\start@gather{$$}{%
	$$%
	\displaywidth=\textwidth
	\displayindent=-\leftskip
}{}{\errmessage{Cannot patch \string\start@gather}}
\patchcmd\start@align{$$}{%
	$$%
	\displaywidth=\textwidth
	\displayindent=-\leftskip
}{}{\errmessage{Cannot patch \string\start@align}}
\patchcmd\start@multline{$$}{%
	$$%
	\displaywidth=\textwidth
	\displayindent=-\leftskip
}{}{\errmessage{Cannot patch \string\start@multline}}
\patchcmd\mathdisplay{$$}{%
	$$%
	\displaywidth=\textwidth
	\displayindent=-\leftskip
}{}{\errmessage{Cannot patch \string\mathdisplay}}
\makeatother

\NewDocumentCommand{\wEnd}{o m}{
	\IfNoValueTF{#1}
	{\int_{#2}}
	{\int^{[#1]}_{#2}}
}

\NewDocumentCommand{\wCoend}{o m}{
	\IfNoValueTF{#1}
	{\int^{#2}}
	{\int_{[#1]}^{#2}}
}
\newcommand{\wLan}[1]{\mathsf{Lan}^{\left[#1\right]}}
\newcommand{\wRan}[1]{\mathsf{Ran}^{\left[#1\right]}}
\newcommand{\wDiLan}[1]{\mathsf{DiLan}^{\left[#1\right]}}
\newcommand{\wDiRan}[1]{\mathsf{DiRan}^{\left[#1\right]}}
\newcommand{\wNat}[1]{\mathsf{Nat}^{[#1]}}

\newcommand{\wVDiNat}[2]{\mathsf{DiNat}^{[#1]}_{#2}}
\newcommand{\wVNat}[2]{\mathsf{Nat}^{[#1]}_{#2}}

\allowdisplaybreaks

\begin{document}
\maketitle
\begin{abstract}
    We specialise a recently introduced notion of generalised dinaturality for functors $T : (\clC^\op)^p \times \clC^q \to \clD$ to the case where the domain (resp., codomain) is constant, obtaining notions of ends (resp., coends) of higher arity, dubbed herein $(p,q)$\hyp{}\emph{ends} (resp., $(p,q)$\hyp{}\emph{coends}). While higher arity co/ends are particular instances of ``totally symmetrised'' (ordinary) co/ends, they serve an important technical role in the study of a number of new categorical phenomena, which may be broadly classified as two new variants of category theory.

    The first of these, \emph{weighted category theory}, consists of the study of weighted variants of the classical notions and construction found in ordinary category theory, \emph{besides that of a limit}. This leads to a host of varied and rich notions, such as \emph{weighted Kan extensions}, \emph{weighted adjunctions}, and \emph{weighted ends}.

    The second, \emph{diagonal category theory}, proceeds in a different (albeit related) direction, in which one replaces universality with respect to natural transformations with universality with respect to \emph{dinatural} transformations, mimicking the passage from limits to ends. In doing so, one again encounters a number of new interesting notions, among which one similarly finds \emph{diagonal Kan extensions}, \emph{diagonal adjunctions}, and \emph{diagonal ends}.
\end{abstract}

\tableofcontents

\section{Introduction}
A functor $T : \oo{\clC} \to \clD$ can be thought as a generalised form of pairing defined on objects of $\clC$. Such a $T$ depends on two variables $C,C'$ running over a ``generalised space'' $\clC$, and given its action on morphisms, once covariant and once contravariant, the `generalised quantity' $T(C,C')$ can be `integrated', yielding two objects with dual universal properties:
\begin{enumtag}{c}
	\item A \emph{coend}, resulting by the symmetrisation along the diagonal of $T$, i.e. modding out the coproduct $\coprod_{C\in \clC}T(C,C)$ by the equivalence relation generated by the arrow functions $T(-,C') : \clC^\op(X,Y) \to \clD(TX, TY)$ and $T(C,-) : \clC(X,Y) \to \clD(TX, TY)$;
    \item An \emph{end}, i.e.\ an object $\int_C T(C,C)$ of $\clC$ arising as an `object of invariants' of `fixed points' for the same action of $T$ on arrows. By dualisation, if a coend is a quotient of $\coprod_{C\in \clC}T(C,C)$, an end is a \emph{subobject} of the product $\prod_{C\in \clC} T(C,C)$.
\end{enumtag}
It is the fact that, given $T : \oo{\clB}\times\oo{\clC} \to\clD$, the coend operation satisfies the commutativity rule
\[\int^A\!\int^B T(B,B;C,C)\cong \int^B\!\int^A T(B,B;C,C),\]
that motivated N.\ Yoneda \cite{Yoneda} to adopt for them an integral-like notation and terminology. For obvious reasons, this is called the \emph{Fubini rule} for coends. Equally obviously, an analogous result holds for ends.

Since, given a functor $T$ as above, the (co)end of $T$ can be computed as a certain (co)equaliser, (co)ends can be regarded as just particular (co)limits, associated to functors of a particular type of variance. Central to this reduction rule of (co)ends to (co)limits is the \emph{twisted arrow category} of $\clC$, i.e.\ the category of elements of the hom functor $\hom_\clC : \oo{\clC} \to \Set$: it is the case that 
\[ \int_A T(A,A)\cong \lim \Big(\Tw{\clC} \xto{\Sigma} \oo{\clC} \xto{T} \clD\Big) \]
and similarly, coends are colimits over $\Tw{\clC^\op}^\op$.




Throughout the years, \emph{coend calculus}, i.e.\ the set of rules allowing one to formally manipulate integrals of the above kind in order to prove statements in category theory, has found applications in many different fields of mathematics and even of computer science; see \cite{cofriend} for a comprehensive account of the topic.
\begin{question*}
	The particular variance $T$ is forced to have begs the question of whether there is an analogue of the above picture (a universal property, a Fubini rule, a mass of examples and applications), for more general functors
	\[ T : \underbrace{\clC^\op \times\dots\times \clC^\op}_{p \text{ times}} \times \underbrace{\clC\times\dots\times\clC}_{q \text{ times}} \to \clD \]
	taking $p\geq1$ contravariant arguments, and $q\geq1$ covariant arguments, admitting the possibility that $p\neq q$.\footnote{In \cite{kelly1972many}, the author says that a functor $T : \clC^n \to \clD$ has ``type $n$''; mimicking this nomenclature, we shortly refer to our functors as having ``type $\typepq{p}{q}$''.} The present paper aims to answer this question in the positive.
\end{question*}
Such generalised (or ``$(p,q)$''-)co/ends exist and can be characterised as suitable particular cases of classical co/ends, called $(1,1)$-co/ends. Moreover, they yield a similarly rich calculus.

Knowing that all $(p,q)$\hyp{}coends are suitable $(1,1)$-coends, one might expect that they do not provide new examples. This is not the case, but there is a price to pay, and that is inventing new category theory: the resulting theory is rich, having a great variety of examples (see \cref{sec:examples-of-higher-arity-ends}). Among these, we focus on two natural frameworks harbouring ``higher arity'' co/ends:
\begin{enumtag}{or}
	\item The possibility to define a higher arity analogue of the Day convolution monoidal structure of \cite{day:thesis,day:report,imkelly} on the presheaf category $\Cat(\clC^\op,\Set)$ over a monoidal category $(\clC,\otimes)$. Classically, the Day convolution of two presheaves $\SheafFont{F},\SheafFont{G}:\clC^\op\to\Set$ is defined by
	\[\SheafFont{F}\otimesDay \SheafFont{G} : \int^{X,Y\in\clC} \SheafFont{F}(X)\times \SheafFont{G}(Y) \times \clC(-,X\otimes Y).\]
    We generalise this notion in \cref{nn_day_convolution}: given $n$ presheaves $\tpl{\SheafFont{F}}$, their \emph{$n$-ary Day convolution} is the $(n,n)$-coend
	\[(\SheafFont{F}_{1}\otimesDayN{n}\cdots\otimesDayN{n}\SheafFont{F}_{n})(A)\defeq\pqCoend{n}{n}{A\in\clC}\SheafFont{F}_{1}(A)\times\cdots\times\SheafFont{F}_{n}(A)\times\clC\left(-,A^{\otimes n}\right).\]
	The sets of various $n$-ary Day convolutions, together with the convolution of order $k\le n$, organise into an operad that we dub the \emph{Day operad} in \cref{the-day-higher-arity-convolution-operad}.
    \item The object of higher arity dinatural transformations from a functor $F : (\clC^\op)^p \times \clC^q \to \clD$ to a functor $G\colon (\clC^{\op})^{q}\times\clC^{p}\longrightarrow\CatFont{D}$ turns out to be a natural example of a $(p,q)$\hyp{}end. Firstly noticed by Street and Dubuc in \cite[Theorem 1]{dinatural-transformations} for (ordinary, i.e. $(1,1)$-) dinatural transformations, this result holds in the higher arity context and provides one with an important tool for the study of higher arity co/ends. 
    
    Extending Street and Dubuc's work, we provide a canonical way to appropriately resolve a functor $F$ of type $\typepq{p}{q}$ into a functor $\ckgpq{p}{q}{F}$ of type $\typepq{q}{p}$ satisfying
    \[\Nat\left(\ckgpq{p}{q}{F},G\right) \cong \pqDiNat{p}{q}(F,G).\]
    That is, $\ckgpq{p}{q}{F}$ is characterised by the universal property that every \emph{dinatural} transformation from $F$ to $G$ amounts exactly to a \emph{natural} transformation from $\ckgpq{p}{q}{F}$ to $G$. Dually, one can ``resolve'' $G$ into a functor $\kgpq{q}{p}{G}$ satisfying
	\[\Nat\left(F,\kgpq{p}{q}{G}\right)  \cong \pqDiNat{p}{q}(F,G)\]
    We study these functors in detail in \cref{sec:higher-arity-yoneda}, where, as an application, we construct an analogue of the twisted arrow category $\Tw{\clC}$ of $\clC$ in the higher arity setting, dubbed the \emph{higher arity twisted arrow category} of $\clC$ in \cref{higher-arity-twisted-arrow-categories}.
\end{enumtag}
Higher arity co/ends also pave the way to a number of future research directions:
\begin{enumerate}
	\item\emph{Weighted category theory (besides co/limits)}.
	      The notion of weighted co/limit arises as the solution to a representability problem. Computing the `conical' colimit of a diagram $D : \clC \to \clD$, we aim to find an object $\colim(D)$ that represents the functor sending $X\in \clD$ to the set of cocones for $D$, i.e.\ natural transformations between the `constant at the point' functor $*$ and the functor $\clD(D-,X)$. In a similar fashion, when we want to compute the \emph{weighted} colimit of $D$ we try to represent the functor
	      \[X\mapsto \Cat(\clC^\op,\Set)(W,\clD(D-,X)).\]
	      Now, what if we try to do the same for the other usual categorical notions apart from that of co/limits, such as adjunctions, Kan extensions, monads, or co/ends? For example, what if, instead of trying to represent the functor that sends $T : \oo{\clC} \to \clD$ to the set of its co/wedges, we try to represent the functor that sends $X$ to $\DiNat(W,\clD(T,X))$?

	      We dub \emph{weighted category theory} the piece of technology that addresses this problem.
	\item\emph{Diagonal category theory}.
          In diagonal category theory, a different (albeit related) path is taken. In analogy to the passage from limits (universal cones) to ends (universal wedges), we seek a general framework, for other categorical constructions, in which the passage from cones to wedges is meaningful. Thus, we aim to replace naturality requests with dinaturality in categorical concepts besides that of a limit, obtaining, as a result, a very rich theory with notions such as diagonal Kan extensions, adjunctions and monads, all standing to their classical counterparts as co/ends stand to co/limits.

	      Despite the intrinsic obstruction to compose dinatural transformations, the theory so obtained is non-trivial and sheds light on a number of aspects of classical category theory, such as the disparity between limits and colimits assorting themselves into a triple adjunction
	      \[
		      (\colim\dashv\Delta_{(-)}\dashv\lim)\colon\Fun(\clC,\CatFont{D})\longrightleftrightarrows\CatFont{D}
	      \]
          with no such result existing for ends and coends: rather, we have \emph{diagonally adjoint functors}, of which ends and coends are a fundamental example, assembling into a \say{triple diagonal adjunction} $\int^{A}\dashv_{\delta}\Delta_{(-)}\dashv_{\delta}\int_{A}$.
\end{enumerate}
\subsection{Structure of the paper}
In \cref{sec:dinatural-transformations-of-higher-arity-from-the-geometry-of-hypercubes}, we motivate higher arity dinaturality, showing how such a notion arises geometrically as a \say{diagonal} version of natural transformations between functors with domain a product category. All the material there contained is very well-known, and this is only meant to fix notation at the outset. We borrow from \cite[pp.~48--50]{baez2010physics} an intuitive explanation of how dinaturality arises from elementary considerations\footnote{%
    See also \cite{dinaturality-from-naturality} for a similar presentation.
}, %
first recalling it in \cref{rem:from-naturality-cubes-to-dinaturality} and then generalising that classical argument to the higher arity case in \cref{rem:from-naturality-hypercubes-to-p-q-dinaturality}.

In \cref{sec:higher-arity-wedges}, we review and specialise the notion of dinaturality introduced in \cite{AlessioThesis,2007.07576} to the appropriate setting for considering ``universal higher arity dinaturality''. In detail, we first recall the notion of a \emph{$(p,q)$\hyp{}dinatural transformation} and study its properties, generalising results of Street--Dubuc (\cite{dinatural-transformations}) to functors of arbitrary arity. We then proceed to discuss $(p,q)$\hyp{}dinatural transformations from constant functors, which we dub \emph{$(p,q)$\hyp{}wedges}, in analogy with the classical case.

In \cref{sec:higher-arity-ends}, we formulate the notion of a \emph{higher arity co/end}. Just as ends are universal wedges, higher arity ends are universal $(p,q)$\hyp{}wedges. After introducing them in \cref{def:p-q-ends}, we discuss some of their basic properties (\cref{prop:properties-of-p-q-ends}). The intuition behind the definition is that a $(p,q)$-end of a functor $T$ is an end of a version of $T$ that has been ``completely symmetrised''. Then, in \cref{adjoints-and-the-fubini-rule}, we state and prove a Fubini rule for higher arity co/ends, generalising the classical Fubini rule for co/ends.

In \cref{sec:examples-of-higher-arity-ends}, we illustrate the theory developed so far by working out a large number of examples. We study naturally-appearing instances of higher arity co/ends in category theory as well as in related areas. The machinery employed here is elementary, but provides insightful examples when applied to laying down the rules of a `calculus of weighted ends', and to `diagonal category theory' (see \cref{glance_at_weightends,glance_at_extradiag} for reference).

In \cref{sec:higher-arity-yoneda}, we introduce the notion of \emph{co/kusarigama}. These are fundamental constructions in higher arity co/end calculus allowing us to reduce the study of $(p,q)$\hyp{}dinaturality to that of (ordinary) naturality. Co/kusarigama also provide us with a way to express higher arity co/ends as weighted co/limits, as well as with a higher arity version of the twisted arrow category (\cref{higher-arity-twisted-arrow-categories}).
\subsection{Geometric motivation for higher arity dinaturality}\label{sec:dinatural-transformations-of-higher-arity-from-the-geometry-of-hypercubes}
\begin{notation}[$(p,q)$\hyp{}products, tensor calculus notation]\label{fst_notations}
    The entire paper deals with categories that are the product of $q$ copies of a (small) category $\clC$, and $p$ copies of the opposite category $\clC^\op$, for $p$, $q$ two non-negative integers: throughout the entire discussion we will adopt the notation
	\[\pq{\clC} \defeq \underbrace{\clC^\op \times\dots\times \clC^\op}_{p \text{ times}} \times \underbrace{\clC\times\dots\times\clC}_{q \text{ times}}.\]
	An alternative, short notation for the category $(\clC^\op)^p=(\clC^p)^\op$ is $\clC^{-p}\times \clC^q$, but this has to be used \emph{cum grano salis}, as some of the usual sign conventions do not apply (for example, exponents of discordant signs do not add: $\clC^{-1}\times\clC^1 = \oo{\clC} = \pq{\clC}[1][1]$ is `irreducible').

	\medskip A functor having domain $\pq{\clC}$ will be called a functor of \emph{type} $\typepq{p}{q}$.

	\medskip To denote the action of a functor of type $\typepq{p}{q}$ on objects we write
	\[F_{\uA'}^{\uA''} \defeq F_{\tpl{A'}[1][q]}^{\tpl{A''}[1][p]} \defeq F\left(A''_{1},\ldots,A''_{p},A'_{1},\ldots,A'_{q}\right)\]
    for (tuples of) objects $\uA'\defeq (\tpl{A'}[1][q])\in\clC^q$ and $\uA''\defeq(\tpl{A''}[1][p])\in\clC^{-p}$. (See \cref{sec:notation} below for variations and specialisations of this notation.)

	This is somewhat reminding of the way in which one writes the coordinates of a tensor of type $\typepq{p}{q}$: contravariant indices go up.
\end{notation}
With notation now fixed, we start by recalling the notion of dinatural transformation. We borrow an argument from \cite[pp.~48--50]{baez2010physics} that shows how dinaturality is, if not unavoidable, at least motivated by elementary geometric considerations.

Informally, we may summarise the discussion below as follows: to arrive at the notion of dinaturality, one performs a symmetrisation process, first considering a ``naturality cube'' and then removing its ``non-diagonal'' pieces.

Baez--Stay's argument can be adapted to motivate our notion of \emph{$(p,q)$\hyp{}dinaturality} (\cref{def:p-q-dinatural-transformation}) as similarly unavoidable (see \cref{rem:from-naturality-hypercubes-to-p-q-dinaturality} below).
\begin{remark}[From naturality cubes to dinaturality hexagons]\label{rem:from-naturality-cubes-to-dinaturality}
	Let $\clC$ be a category, and consider the product category $\clC^{\op}\times\clC$. This is the category whose
	\begin{enumtag}{cc}
		\item\label{product-category-1}Objects are pairs $(A,B)$ of objects of $\clC$ (usually, the set/class of objects of a category will be denoted as $\clC_o$);
		\item\label{product-category-2}Morphisms $(A,B) \to (A',B')$ are pairs $\Big(\var[g]{A'}{A}, \var[f]{B}{B'}\Big)$ of morphisms of $\clC$.
	\end{enumtag}
	Now, each morphism $f\colon A\to  B$ of $\clC$, gives rise to a commutative square in $\clC^{\op}\times\clC$ of the form
	\begin{center}
		\begin{tikzcd}
			{(B,A)}
			\arrow[d, "{(f,A)}"']
			\arrow[r, "{(B,f)}"]
			&
			{(B,B)}
			\arrow[d, "{(f,B)}"]
			\\
			{(A,A)}
			\arrow[r, "{(A,f)}"']
			&
			{(A,B)\mathrlap{,}}
		\end{tikzcd}
	\end{center}
	to which we can apply functors $F,G\colon\oo{\clC} \longrightrightarrows \clD$, obtaining two naturality squares
	\[
		\begin{tikzcd}
			F^{B}_{A}
			\arrow[r, "F^{B}_{f}"]
			\arrow[d, "F^{f}_{A}"']
			&
			F^{B}_{B}
			\arrow[d, "F^{f}_{B}"]
			\\
			F^{A}_{A}
			\arrow[r, "F^{A}_{f}"']
			&
			F^{A}_{B}
		\end{tikzcd}
		\quad
		\quad
		\begin{tikzcd}
			G^{B}_{A}
			\arrow[r, "G^{B}_{f}"]
			\arrow[d, "G^{f}_{A}"']
			&
			G^{B}_{B}
			\arrow[d, "G^{f}_{B}"]
			\\
			G^{A}_{A}
			\arrow[r, "G^{A}_{f}"']
			&
			G^{A}_{B}\mathrlap{,}
		\end{tikzcd}
	\]
	from which we derive that $F^f_B\circ F^B_f = F_f^A\circ F^f_A$, and similarly for $G$.

	A natural transformation from $F$ to $G$ is then a collection
	\[
		\{\alpha^{A}_{B}\colon F^{A}_{B}\to  G^{A}_{B} \mid (A,B)\in(\clC^{\op}\times\clC)_o\}
	\]
	of morphisms of $\clD$ making the diagram
	\begin{center}
		\begin{tikzcd}
			F^{A}_{B}
			\arrow[r, "F^{g}_{f}"]
			\arrow[d, "\alpha^{A}_{B}"']
			&
			F^{A'}_{B'}
			\arrow[d, "\alpha^{A'}_{B'}"]
			\\
			G^{A}_{B}
			\arrow[r, "G^{g}_{f}"']
			&
			G^{A'}_{B'}
		\end{tikzcd}
	\end{center}
	commute for every $\var[g]{A'}{A}$ and $\var[f]{B}{B'}$. When $g=f$, this reduces to
	\begin{center}
		\begin{tikzcd}
			F^{B}_{A}
			\arrow[r, "F^{f}_{f}"]
			\arrow[d, "\alpha^{B}_{A}"']
			&
			F^{A}_{B}
			\arrow[d, "\alpha^{A}_{B}"]
			\\
			G^{B}_{A}
			\arrow[r, "G^{f}_{f}"']
			&
			G^{A}_{B},
		\end{tikzcd}
	\end{center}
	which we may rewrite as the following commutative cube, using that $(f,f)=(\id_{B},f)\circ(f,\id_{A})$ in $\oo{\clC}$:
	\[\footnotesize
		\begin{tikzcd}[row sep=.25cm, column sep=.25cm]
			F^{B}_{A}
			\arrow[rr]
			\arrow[rd]
			\ar[dd]
			&&
			F^{B}_{B}
			\ar[dr]
			&
			\\
			&
			F^{A}_{A}
			\ar[rr,crossing over]
			&&
			F^{A}_{B}
			\ar[dd]
			\\
			G^{B}_{A}
			\ar[dr]
			&&
			{}
			&
			\\
			&
			G^{A}_{A}
			\ar[from=uu]
			\arrow[rr]
			&&
			G^{A}_{B}
		\end{tikzcd}
		\hspace{0.5cm}{\fontsize{20}{20}\selectfont =}\hspace{0.5cm}
		\begin{tikzcd}[row sep=.25cm, column sep=.25cm]
			F^{B}_{A}
			\arrow[rr]
			\ar[dd]
			&&
			F^{B}_{B}
			\arrow[dd]
			\ar[dr]
			&
			\\
			&
			{}
			&&
			F^{A}_{B}
			\ar[dd]
			\\
			G^{B}_{A}
			\ar[dr]
			\ar[rr]
			&&
			G^{B}_{B}
			\arrow[rd]
			&
			\\
			&
			G^{A}_{A}
			\arrow[rr]
			&&
			G^{A}_{B}
		\end{tikzcd}
	\]
	(in all arrows the action of $F$ on its covariant or contravariant component is taken into account).

	A notion of \say{diagonal transformation} between $F$ and $G$ is then a collection of morphisms of $\clD$ from $F^{A}_{A}$ to $G^{A}_{A}$. To obtain it, we should remove the \say{non-diagonal pieces} $\alpha^{B}_{A}$ and $\alpha^{A}_{B}$ from this cube, arriving at the diagram
	\[\scriptscriptstyle
		\begin{tikzcd}[ampersand replacement=\&, row sep=.5cm, column sep=.5cm]
			F^{B}_{A}
			\arrow[rr, "F^{B}_{f}"]
			\arrow[rd, "F^{f}_{A}"']
			\ar[dd, gray!40]
			\&\&
			F^{B}_{B}
			\arrow[dd, "\alpha^{B}_{B}",pos=0.675]
			\ar[dr, gray!40]
			\&
			\\
			\&
			F^{A}_{A}
			\ar[rr, gray!40,crossing over]
			\&\&
			\textcolor{gray!40}{F^{A}_{B}}
			\ar[dd, gray!40]
			\\
			\textcolor{gray!40}{G^{B}_{A}}
			\ar[dr, gray!40]
			\ar[rr, gray!40]
			\&\&
			G^{B}_{B}
			\arrow[rd, "G^{f}_{B}"]
			\&
			\\
			\&
			G^{A}_{A}
			\ar[from=uu, "\alpha_A^A"', crossing over,pos=0.325]
			\arrow[rr, "G^{A}_{f}"']
			\&\&
			G^{A}_{B}\mathrlap{.}
		\end{tikzcd}
	\]
	Writing $\alpha_{A}$ (resp.\ $\alpha_{B}$) for $\alpha^{A}_{A}$ (resp.\ $\alpha^{B}_{B}$) and `flattening' the resulting diagram, we get the dinaturality hexagon
	\[\small
		\begin{tikzcd}[row sep={0.0cm,between origins}, column sep={0.0cm,between origins}, ampersand replacement=\&]
			\&[0.43301270189\OneCmAndAHalf]
			F^{A}_{A}
			\arrow[r,"\alpha_{A}"]
			\&[1.0\OneCmAndAHalf]
			G^{A}_{A}
			\arrow[rd,"G^{A}_{f}"]
			\&[0.43301270189\OneCmAndAHalf]
			\\[0.86602540378\OneCmAndAHalf]
			F^{B}_{A}
			\arrow[ru,"F^{f}_{A}"]
			\arrow[rd,"F^{B}_{f}"']
			\&[0.43301270189\OneCmAndAHalf]
			\&[1.0\OneCmAndAHalf]
			\&[0.43301270189\OneCmAndAHalf]
			G^{A}_{B}
			\\[0.86602540378\OneCmAndAHalf]
			\&[0.43301270189\OneCmAndAHalf]
			F^{B}_{B}
			\arrow[r,"\alpha_{B}"']
			\&[1.0\OneCmAndAHalf]
			G^{B}_{B}
			\arrow[ru,"G^{f}_{B}"']
			\&[0.43301270189\OneCmAndAHalf]
		\end{tikzcd}
	\]
	for $\alpha : F\din G$.
\end{remark}
Given functors $F : \pq{\clC} \to \clD$ and $G\colon\clC^{(q,p)}\longrightarrow\CatFont{D}$, we seek a similar intuitive explanation for an analogous notion of $(p,q)$\hyp{}dinaturality. Of course, as the sum $p+q$ grows bigger, it is harder and harder to visualise the underlying geometry, since we have to work in dimension $p+q+1\ge 4$. For this reason, we content ourselves with the case $(p,q)=(2,1)$.

Recall that we write $\pq{\clC}[2][1]$ for the category $\clC^{\op}\times\clC^{\op}\times\clC$.
\begin{remark}[From naturality hypercubes to $(2,1)$-Dinaturality]\label{rem:from-naturality-hypercubes-to-p-q-dinaturality}
    In a similar fashion to \cref{rem:from-naturality-cubes-to-dinaturality}, a morphism $f: A\to B$ induces a commutative cube in $\pq{\clC}[2][1]$ which, under the action of $F$ and $G$, yields two commutative cubes
	\[\footnotesize
		\begin{tikzcd}[row sep={1.0cm,between origins}, column sep={1.0cm,between origins}, baseline=(current bounding box.center)]
			F^{BB}_{A}
			\arrow[rr]
			\arrow[dd]
			\arrow[rd]
			&&
			F^{BB}_{B}
			\arrow[rd]
			\arrow[dd,dashed]
			&
			\\
			&
			F^{BA}_{A}
			\arrow[rr,crossing over]
			&&
			F^{BA}_{B}
			\arrow[dd]
			\\
			F^{AB}_{A}
			\arrow[rd]
			\arrow[rr,dashed]
			&&
			F^{AB}_{B}
			\arrow[rd,dashed]
			&
			\\
			&
			F^{AA}_{A}
			\arrow[rr]
			\arrow[from=uu,crossing over]
			&&
			F^{AA}_{B}
		\end{tikzcd}
		\hspace{1cm}
		\begin{tikzcd}[row sep={1.0cm,between origins}, column sep={1.0cm,between origins}, baseline=(current bounding box.center)]
			G^{BB}_{A}
			\arrow[rr]
			\arrow[dd]
			\arrow[rd]
			&&
			G^{BB}_{B}
			\arrow[rd]
			\arrow[dd,dashed]
			&
			\\
			&
			G^{BA}_{A}
			\arrow[rr,crossing over]
			&&
			G^{BA}_{B}
			\arrow[dd]
			\\
			G^{AB}_{A}
			\arrow[rd]
			\arrow[rr,dashed]
			&&
			G^{AB}_{B}
			\arrow[rd,dashed]
			&
			\\
			&
			G^{AA}_{A}
			\arrow[rr]
			\arrow[from=uu,crossing over]
			&&
			G^{AA}_{B}
		\end{tikzcd}
	\]
	in $\clD$. Now, a natural transformation $\alpha$ from $F$ to $G$ is a collection
	\[
        \Big\{\alpha^{AB}_{C}\colon F^{AB}_{C}\to  G^{AB}_{C}\mid (A,B,C)\in \pq{\clC}[2][1]_o \Big\}
	\]
	of morphisms of $\clD$ making the hypercube diagram below-left commute:
	\begin{center}
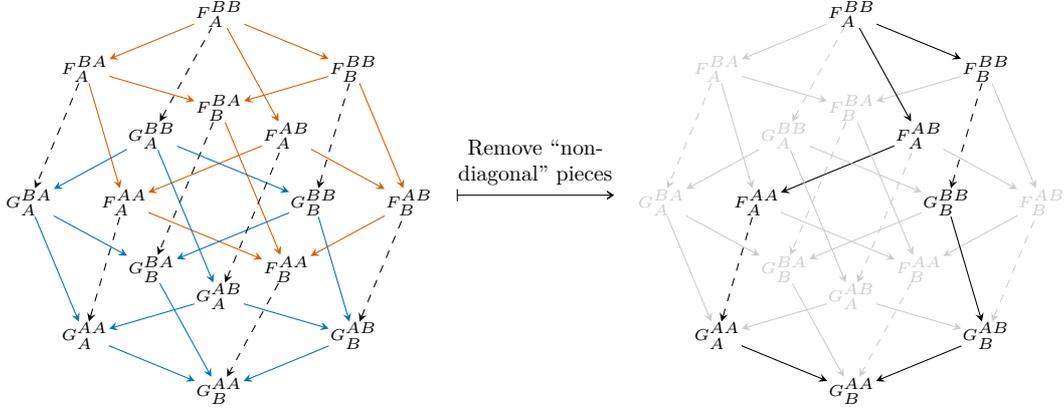
        \begin{codi}[baseline=(current bounding box.center),scale=.625]
            \node[inner sep=.25pt] (0-inner) at (0*45:2) {$\scriptscriptstyle G^{BB}_{B}$};
            \node[inner sep=.25pt] (0-outer) at (0*45:4) {$\scriptscriptstyle F^{AB}_{B}$};
            \node[inner sep=.25pt] (1-inner) at (1*45:2) {$\scriptscriptstyle F^{AB}_{A}$};
            \node[inner sep=.25pt] (1-outer) at (1*45:4) {$\scriptscriptstyle F^{BB}_{B}$};
            \node[inner sep=.3pt] (2-inner) at (2*45:2) {$\scriptscriptstyle F^{BA}_{B}$};
            \node[inner sep=.25pt] (2-outer) at (2*45:4) {$\scriptscriptstyle F^{BB}_{A}$};
            \node[inner sep=.25pt] (3-inner) at (3*45:2) {$\scriptscriptstyle G^{BB}_{A}$};
            \node[inner sep=.25pt] (3-outer) at (3*45:4) {$\scriptscriptstyle F^{BA}_{A}$};
            \node[inner sep=.25pt] (4-inner) at (4*45:2) {$\scriptscriptstyle F^{AA}_{A}$};
            \node[inner sep=.25pt] (4-outer) at (4*45:4) {$\scriptscriptstyle G^{BA}_{A}$};
            \node[inner sep=.25pt] (5-inner) at (5*45:2) {$\scriptscriptstyle G^{BA}_{B}$};
            \node[inner sep=.25pt] (5-outer) at (5*45:4) {$\scriptscriptstyle G^{AA}_{A}$};
            \node[inner sep=.25pt] (6-inner) at (6*45:2) {$\scriptscriptstyle G^{AB}_{A}$};
            \node[inner sep=.25pt] (6-outer) at (6*45:4) {$\scriptscriptstyle G^{AA}_{B}$};
            \node[inner sep=.25pt] (7-inner) at (7*45:2) {$\scriptscriptstyle F^{AA}_{B}$};
            \node[inner sep=.25pt] (7-outer) at (7*45:4) {$\scriptscriptstyle G^{AB}_{B}$};
            \mor (1-inner) [OIvermillion,->] (4-inner) [OIvermillion,->] (7-inner) [OIvermillion,<-] (2-inner) [dashed, ->] (5-inner) [OIblue,<-] (0-inner) [OIblue,<-] (3-inner) [OIblue,->] (6-inner) [dashed, <-] (1-inner);
            \mor (0-outer) [OIvermillion,<-] (1-outer) [OIvermillion,<-] (2-outer) [OIvermillion,->] (3-outer) [dashed, ->] (4-outer) [OIblue,->] (5-outer) [OIblue,->] (6-outer) [OIblue,<-] (7-outer) [dashed, <-] (0-outer);
            \mor (1-inner) [OIvermillion,->] (0-outer) [OIvermillion,->] (7-inner);
            \mor (2-inner) [OIvermillion,<-] (1-outer) [dashed,->] (0-inner);
            \mor (3-inner) [dashed,<-] (2-outer) [OIvermillion,->] (1-inner);
            \mor (4-inner) [OIvermillion,<-] (3-outer) [OIvermillion,->] (2-inner);
            \mor (5-inner) [OIblue,<-] (4-outer) [OIblue,<-] (3-inner);
            \mor (6-inner) [OIblue,->] (5-outer) [dashed,<-] (4-inner);
            \mor (7-inner) [dashed,->] (6-outer) [OIblue,<-] (5-inner);
            \mor (0-inner) [OIblue,->] (7-outer) [OIblue,<-] (6-inner);
        \end{codi}
        \hspace{-0.375em}
        \begin{tikzcd}[row sep={6.0em,between origins}, column sep={6.0em,between origins}, ampersand replacement=\&,cramped]
            {}
            \arrow[r, "\scalebox{0.75}{%
                    \begin{tabular}{c}%
                        Remove ``non- \\%
                        diagonal'' pieces%
                    \end{tabular}%
                }",mapsto]
            \&
            {}
        \end{tikzcd}
        \hspace{-0.375em}
        \begin{codi}[baseline=(current bounding box.center),scale=.625]
            \node[inner sep=.25pt] (0-inner) at (0*45:2) {$\scriptscriptstyle G^{BB}_{B}$};
            \node[inner sep=.25pt] (0-outer) at (0*45:4) {\textcolor{gray!40}{$\scriptscriptstyle F^{AB}_{B}$}};
            \node[inner sep=.25pt] (1-inner) at (1*45:2) {$\scriptscriptstyle F^{AB}_{A}$};
            \node[inner sep=.25pt] (1-outer) at (1*45:4) {$\scriptscriptstyle F^{BB}_{B}$};
            \node[inner sep=.25pt] (2-inner) at (2*45:2) {\textcolor{gray!40}{$\scriptscriptstyle F^{BA}_{B}$}};
            \node[inner sep=.25pt] (2-outer) at (2*45:4) {$\scriptscriptstyle F^{BB}_{A}$};
            \node[inner sep=.25pt] (3-inner) at (3*45:2) {\textcolor{gray!40}{$\scriptscriptstyle G^{BB}_{A}$}};
            \node[inner sep=.25pt] (3-outer) at (3*45:4) {\textcolor{gray!40}{$\scriptscriptstyle F^{BA}_{A}$}};
            \node[inner sep=.25pt] (4-inner) at (4*45:2) {$\scriptscriptstyle F^{AA}_{A}$};
            \node[inner sep=.25pt] (4-outer) at (4*45:4) {\textcolor{gray!40}{$\scriptscriptstyle G^{BA}_{A}$}};
            \node[inner sep=.25pt] (5-inner) at (5*45:2) {\textcolor{gray!40}{$\scriptscriptstyle G^{BA}_{B}$}};
            \node[inner sep=.25pt] (5-outer) at (5*45:4) {$\scriptscriptstyle G^{AA}_{A}$};
            \node[inner sep=.25pt] (6-inner) at (6*45:2) {\textcolor{gray!40}{$\scriptscriptstyle G^{AB}_{A}$}};
            \node[inner sep=.25pt] (6-outer) at (6*45:4) {$\scriptscriptstyle G^{AA}_{B}$};
            \node[inner sep=.25pt] (7-inner) at (7*45:2) {\textcolor{gray!40}{$\scriptscriptstyle F^{AA}_{B}$}};
            \node[inner sep=.25pt] (7-outer) at (7*45:4) {$\scriptscriptstyle G^{AB}_{B}$};
            \mor (1-inner) [->] (4-inner) [gray!40,->] (7-inner) [gray!40,<-] (2-inner) [dashed, gray!40, ->] (5-inner) [gray!40,<-] (0-inner) [gray!40,<-] (3-inner) [gray!40,->] (6-inner) [dashed, gray!40, <-] (1-inner);
            \mor (0-outer) [gray!40,<-] (1-outer) [<-] (2-outer) [gray!40,->] (3-outer) [dashed, gray!40, ->] (4-outer) [gray!40,->] (5-outer) [->] (6-outer) [<-] (7-outer) [dashed, gray!40, <-] (0-outer);
            \mor (1-inner) [gray!40,->] (0-outer) [gray!40,->] (7-inner);
            \mor (2-inner) [gray!40,<-] (1-outer) [dashed,->] (0-inner);
            \mor (3-inner) [dashed, gray!40,<-] (2-outer) [->] (1-inner);
            \mor (4-inner) [gray!40,<-] (3-outer) [gray!40,->] (2-inner);
            \mor (5-inner) [gray!40,<-] (4-outer) [gray!40,<-] (3-inner);
            \mor (6-inner) [gray!40,->] (5-outer) [dashed,<-] (4-inner);
            \mor (7-inner) [dashed, gray!40,->] (6-outer) [gray!40,<-] (5-inner);
            \mor (0-inner) [->] (7-outer) [gray!40,<-] (6-inner);
        \end{codi}
	\end{center}
    (\textcolor{OIvermillion}{vermillion}: the $F$ cube; \textcolor{OIblue}{blue}: the $G$ cube). Again deleting the nodes $F_X^{Y,Z}$, $G_X^{Y,Z}$ for which $X,Y,Z$ are not all equal, we get the above-right diagram. Flattening the result, this gives us the octagonal diagram below-left, which becomes the \say{$(2,1)$-dinaturality hexagon} below-right upon using that $F^{f,f}_{A} = F^{A,f}_{A}\circ F^{f,B}_{A}$ and $G^{f,f}_{B} = G^{A,f}_{B}\circ G^{f,B}_{B}$:
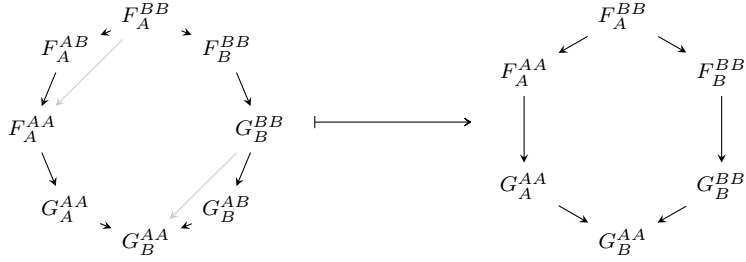
	\begin{center}
		\begin{codi}[baseline=(current bounding box.center)]
			\foreach [count=\i from 0] \l in
			{ G_B^{BB} 
			, F_B^{BB} 
			, F_A^{BB} 
			, F_A^{AB} 
			, F_A^{AA} 
			, G_A^{AA} 
			, G_B^{AA} 
			, G_B^{AB} 
			}
			\node (\i) at (\i*45:1.5) {$\scriptstyle\l$};
			\mor (2) -> (1) -> (0) -> (7) -> (6);
			\mor * -> (3) -> (4) -> (5) -> *;
			\mor (2) [gray!40,->] (4);
			\mor (0) [gray!40,->] (6);
		\end{codi}
		\begin{tikzcd}[row sep={6.0em,between origins}, column sep={6.0em,between origins}, ampersand replacement=\&,cramped]
			{}
			\arrow[r, mapsto]
			\&
			{}
		\end{tikzcd}
		\begin{codi}[baseline=(current bounding box.center)]
			\foreach [count=\i from 0] \l in
			{ F_A^{BB} 
			, F_A^{AA} 
			, G_A^{AA} 
			, G_B^{AA} 
			, G_B^{BB} 
			, F_B^{BB} 
			}
			\node (\i) at (\i*60+90:1.5) {$\scriptstyle\l$};
			\mor (0) -> (1) -> (2) -> (3);
			\mor * -> (5) -> (4) -> *;
		\end{codi}
	\end{center}
\end{remark}
The reader might have noticed, at this point, that in order not to clutter the page with too many unwanted apices and pedices, we have to establish appropriate notation to represent how functors act on tuples.
\subsection{Notation and preliminaries}\label{sec:notation}
All the basic notation for categories and functors used in this paper follows standard practice. Apart from this, and apart from what we already introduced in \cref{fst_notations}, we need notation for:
\begin{enumtag}{n}
	\item\label{n_1} A generic tuple of objects,
	\[\uA \defeq (\tpl{A})\]
	often split as the juxtaposition $\uA';\uA''$ of two subtuples of length $p,q$,
	\[ \uA'\defeq (\tpl{A}[1][q]), \qquad \uA'' \defeq(\tpl{A}[p+1][p+q])\]
	\item \label{n_2}  As already said, the image of a split tuple $\uA';\uA''$ under a functor of type $\typepq{p}{q}$, $F : \pq{\clC} \to \clD$ is denoted $F^{\uA'}_{\uA''}$: the contravariant components come first/top, and the covariant component come second/bottom. So: contravariant components are always \emph{left} in the typing \[F : \pq{\clC} \to \clD\]
	of a functor, and \emph{up} in its action on objects.
    \item \label{n_3}  Denoting a functor $F$ of type $\typepq{p}{q}$ evaluated at a diagonal tuple $\bsA\defeq(A,\ldots,A)$ with $A\in\CatFont{C}_{o}$: we write
	\[F_{\bsA}^{\bsA} \defeq F_{A,\ldots,A}^{A,\ldots,A},\]
	where the superscript has $p$ elements, and the subscript has $q$ elements.
\end{enumtag}
\begin{definition}\label{def:p-q-diagonal-functor}%
	The \emph{$(p,q)$\hyp{}diagonal functor} is the functor $\pqdiag\colon\clC^{\op}\times\clC\to \pq{\clC}$ defined by
    \[\pqdiag\defeq\underbrace{\Delta^{\op}\times\cdots\times\Delta^{\op}}_{\text{$p$ times}}\times\underbrace{\Delta\times\cdots\times\Delta}_{\text{$q$ times}}.\]
	More explicitly, $\pqdiag\colon\clC^{\op}\times\clC\to \pq{\clC}$ is the functor sending
		\begin{enumtag}{ud}
			\item An object $(A,B)$ of $\clC^{\op}\times\clC$ to the object $(\bsA,\bsB)\defeq(A,\ldots,A,B,\ldots,B)$ of $\pq{\clC}$, and
			\item A morphism $(f,g)\colon(A,B)\to (A',B')$ of $\clC^{\op}\times\clC$ to the morphism $(\bsf,\bsg)\defeq(f,\ldots,f,g,\ldots,g)$ of $\pq{\clC}$,
		\end{enumtag}
		where in the expression $(\bsA,\bsB)$ we have $p$ repeated copies of $A$ and $q$ repeated copies of $B$, and similarly for $(\bsf,\bsg)$.
\end{definition}
The notational choices and lemmas below will be used only from \cref{adjoints-and-the-fubini-rule} on, so we advise the reader to skip them at this point, referring to them later as needed.
\begin{notation}[Mixed Products and Coproduts]\label{biprod-p-q}
    Let $\clC$ be a category with finite products and coproducts. Given tuples $(\tpl{A}[1][p])$ and $(\tpl{B}[1][q])$, we write
	\begin{align*}
		\W_{p,q}(\uA,\uB) & \defeq \left(\prod_{i=1}^{p}A_{i},\coprod_{j=1}^{q}B_{j}\right), \\
		\M_{p,q}(\uA,\uB) & \defeq \left(\coprod_{i=1}^{p}A_{i},\prod_{j=1}^{q}B_{j}\right).
	\end{align*}
\end{notation}
\begin{lemma}[Adjoints to the $(p,q)$\hyp{}Diagonal Functor]\label{adjoints-to-the-p-q-diagonal-functor}
	If $\clC$ has products and coproducts, then we have a triple adjunction
	\begin{center}
		$\left(\W_{p,q}\dashv\pqdiag\dashv\M_{p,q}\right)\colon$
		\begin{tikzcd}[row sep=4.5em, column sep=4.5em, background color=backgroundColor]
			\pq{\clC}
			\arrow[r, "\W_{p,q}"{name=Lan}, shift left=1.5em]
			\arrow[r, "\M_{p,q}"'{name=Ran}, shift right=1.5em]
			&
			\phantom{\pq{\clC}}\mathrlap{\negphantom{$\pq{\clC}$}\clC^{\op}\times\clC.}
			\arrow[l, "\pqdiag"{name=K,description}]
			\arrow[phantom,from=K, to=Lan,"\dashv" rotate=-90]
			\arrow[phantom,from=K, to=Ran,"\dashv" rotate=-90]
		\end{tikzcd}
		$\phantom{\times\clC.}$
	\end{center}%
\end{lemma}
\begin{proof}%
	The lemma follows from the universal property of the co/product, as we have a string of bijections
	\begin{align*}
		\hom_{\pq{\clC}}\left((\bsA,\bsB),\pqdiag(C,D)\right) & \defeq \left(\prod_{i=1}^{p}\hom_{\clC^{\op}}(A_{i},C)\right)\times\left(\prod_{j=1}^{q}\hom_{\clC}(B_{j},D)\right) \\
		                                                      & \defeq \left(\prod_{i=1}^{p}\hom_{\clC}(C,A_{i})\right)\times\left(\prod_{j=1}^{q}\hom_{\clC}(B_{j},D)\right)       \\
		                                                      & \cong  \hom_{\clC}\left(C,\prod_{i=1}^{p}A_{i}\right)\times\hom_{\clC}\left(\coprod_{j=1}^{q}B_{j},D\right)         \\
		                                                      & \defeq  \hom_{\clC^{\op}}\left(\prod_{i=1}^{p}A_{i},C\right)\times\hom_{\clC}\left(\coprod_{j=1}^{q}B_{j},D\right)  \\
		                                                      & \defeq  \hom_{\clC^{\op}\times\clC}\left(\left(\prod_{i=1}^{p}A_{i},\coprod_{j=1}^{q}B_{j}\right),(C,D)\right),     \\
		                                                      & \defeq  \hom_{\clC^{\op}\times\clC}\left(\W_{p,q}(A_{i},B_{j}),(C,D)\right),
	\end{align*}
	natural in $A_{1},\ldots,A_{p},B_{1},\ldots,B_{q},C,D\in\clC_o$. The proof that $\pqdiag$ admits a right adjoint is dual to this one.%
\end{proof}
We also collect a couple of standard results on generating strings of adjunctions by left/right Kan extending a given adjunction:
\begin{lemma}[Applying Kan Extensions to an Adjunction]\label{applying-kan-extensions-to-an-adjunction}
	Every adjunction
	\[ L : \clC \leftrightarrows \clD : R
	\]
	induces a quadruple adjunction $\Lan_{K}\dashv K^{*}\dashv L^{*}\dashv\Ran_{L}$ such that
	\[
		\Lan_{K} \cong L^{*}, \qquad \Ran_{L} \cong K^{*}.
	\]
\end{lemma}
\begin{proof}%
	Both $L$ and $K$ induce triple adjunctions between $\Fun(\clC,\clE)$ and $\Fun(\clD,\clE)$. Proving that $\Lan_{K}\cong L^{*}$ and $\Ran_{L}\cong K^{*}$ would show that these are actually parts of a single quadruple adjunction, which is the stated one. That this is indeed so follows from the string of isomorphisms
	\begin{align*}
		L^{*}(F) & \defeq F\circ L                                  \\
		         & \cong  \int_{X}\clC\left(L(-),X\right)\odot F(X) \\
		         & \cong  \int_{X}\clD(-,K(X))\odot F(X)            \\
		         & \cong  \Ran_{K}F
	\end{align*}
	natural in $F$. Hence $K^{*}\cong\Lan_{L}$. By a similar argument, $K^{*}\cong\Ran_{L}$, finishing the proof.
\end{proof}
Combining two applications of \cref{adjoints-to-the-p-q-diagonal-functor} as well as uniqueness of adjoint functors with \cref{applying-kan-extensions-to-an-adjunction}, we get the following corollary:
\begin{corollary}\label{quintuple-biprod-delta-pq-adjunction}%
	We have a quintuple adjunction
	\[
		\left(\Lan_{\M_{p,q}}\dashv\M_{p,q}^{*}\dashv\pqdiag^{*}\dashv\W_{p,q}^{*}\dashv\Ran_{\W_{p,q}}\right)\colon
		\begin{tikzcd}[row sep={10.8em,between origins}, column sep={10.8em,between origins},  ampersand replacement=\&]
			\Fun(\pq{\clC},\clD)
			\arrow[r,shift left=4]
			\arrow[r,shift left=2,leftarrow]
			\arrow[r]
			\arrow[r,shift right=2,leftarrow]
			\arrow[r,shift right=4]
			\&
			\Fun(\clC^{\op}\times\clC,\clD),
		\end{tikzcd}
	\]
	with natural isomorphisms
	\begin{align*}
		\Lan_{\pqdiag} & \cong \M_{p,q}^{*}, \\
		\Ran_{\pqdiag} & \cong \W_{p,q}^{*}.
	\end{align*}
\end{corollary}

\section{Higher Arity Wedges}\label{sec:higher-arity-wedges}
This section formalises completely the notion that we dubbed ``(2,1)-dinaturality'' above and presents it for general $p, q\ge 0$.

The definition of dinaturality given below is not new: it was recently introduced in Santamaria's PhD thesis \cite{2007.07576,AlessioThesis}, building on previous work by M.\ Kelly \cite{kelly1972abstract,kelly1972many} in fair more generality than needed for our purposes.

In \cite{2007.07576,AlessioThesis}, however, an ``unbiased'' arrangement of the factors in $\pq{\clC}$ is considered, in the sense that \cite[Definition 2.4]{2007.07576} takes into account functors $\clC^\alpha \to \clB$, where $\alpha$ is a ``binary multi-index'', i.e.\ an element in the free monoid over the set $\{\oplus,\ominus\}$, and the convention is that $\clC^\varnothing \defeq \catpt$, the terminal category, $\clC^\oplus \defeq \clC$, $\clC^\ominus \defeq \clC^\op$, and $\clC^{\alpha \uplus \alpha'} \defeq \clC^\alpha \times \clC^{\alpha'}$.

Here instead, we adopt a different convention: a generic power $\clC^\alpha$ is always ``reshuffled'' in order for all its minus and plus signs to appear on the same side, respectively on the left and on the right. The categories $\clC^\alpha$ and $\pq{\clC}$ so obtained are, of course, canonically isomorphic, and the tuple $\alpha$ is equivalent to the reshuffled tuple $(\ominus_1,\dots,\ominus_p,\oplus_1,\dots,\oplus_q)$.
\subsection{Higher arity dinaturality}
Let $p,q\in\N$ and $\clC$ be a category. The definition of a \emph{$(p,q)$\hyp{}dinatural transformation} from a functor $F\colon\pq{\clC} \to \clD$ of type $\typepq{p}{q}$ to a functor $G\colon\pq{\clC}[q][p] \to \clD$ of type $\typepq{q}{p}$ can be shortly stated as the condition that a dinaturality hexagon commutes, when filled with the conjoint action of $F$ (resp.\ $G$) in all its contravariant and covariant components separately.

The choice of joining with a transformation just functors of opposite types deserves a bit of explanation.

As stated in \cref{def:p-q-dinatural-transformation}, the definition could be dubbed ``$(p,q)$\hyp{}to-$(q,p)$'' dinaturality. It sits between a rigid one (a ``$(p,q)$\hyp{}to-$(p,q)$'' dinaturality with $F$ and $G$ of the same type) and a loose one (a ``$(p,q)$\hyp{}to-$(r,s)$'' dinaturality, where $F$ and $G$ can have completely different types: this is the path chosen by \cite{AlessioThesis}, recalled in \cref{def:dinaturality-i-th-variable} below).

We could have stuck with the tighter notion, but some of the characterisations we give would not hold: for example, the set of $(p,q)$\hyp{}dinatural transformations between $F$ and $G$ is a $(p,q)$\hyp{}end only with our convention (see \cref{p-q-as-dinaturals}).

We could have stuck with the looser one; but the definition of co/wedge given in \cref{def:p-q-wedges} wouldn't have changed (a constant functor can be ``made mute'', in the sense of \cref{not:dummy}, to have whatever type is needed).

Both the loose and rigid notions of generalised dinaturality mentioned above allow for component-wise composition, but, as is well-known in the classical case for ordinary dinaturality, the resulting family may fail to be dinatural. This does not present a problem for developing the theory of higher arity co/ends, as one just needs generalised dinaturality to be pre- or post-composable with natural transformations, and this is indeed the case (\cref{theta-circ-alpha-is-indeed-a-p-q-dinatural-transformation}) for the notion of higher arity dinaturality introduced below.%
%
%
%
\begin{definition}\label{def:p-q-dinatural-transformation}%
	A \emph{$(p,q)$\hyp{}dinatural transformation} $\alpha: F\din G$ is a collection
	\[
		\big\{\alpha_A: F^{\overset{p\text{ times}}{A,\dots,A}}_{\underset{q\text{ times}}{A,\dots,A}}
		\to
		G^{\overset{q\text{ times}}{A,\dots,A}}_{\underset{p\text{ times}}{A,\dots,A}} \mid A \in \clC_o \big\}
	\]
	of morphisms of $\clD$ indexed by the objects of $\clC$ such that, for each morphism $f: A\to  B$ of $\clC$, the diagram
	\begin{center}
		\begin{codi}
			\obj[hexagonal=horizontal side 4em angle 60] {%
			|[overwrite=false] (3)| F^{\ConstantPNo{A}{p}}_{\ConstantPNo{A}{q}}
			&
			|[overwrite=false] (2)| G^{\ConstantPNo{A}{q}}_{\ConstantPNo{A}{p}}
			&
			\\
			|[overwrite=false] (4)| F^{\ConstantPNo{B}{p}}_{\ConstantPNo{A}{q}}
			&
			&
			|[overwrite=false] (1)| G^{\ConstantPNo{A}{q}}_{\ConstantPNo{B}{p}}
			\\
			|[overwrite=false] (5)| F^{\ConstantPNo{B}{p}}_{\ConstantPNo{B}{q}}
			&
			|[overwrite=false] (6)| G^{\ConstantPNo{B}{q}}_{\ConstantPNo{B}{p}}
			&
			\\};
			\mor (3) ["\alpha_{A}",->] (2);
			\mor (5) ["\alpha_{B}"',->] (6);
			\mor (4) ["F^{\ConstantPNo{f}{p}}_{\ConstantPNo{A}{q}}",->] (3);
			\mor (4) ["F^{\ConstantPNo{B}{p}}_{\ConstantPNo{f}{q}}"',->] (5);
			\mor (6) ["G^{\ConstantPNo{f}{q}}_{\ConstantPNo{B}{p}}"',->] (1);
			\mor (2) ["G^{\ConstantPNo{A}{q}}_{\ConstantPNo{f}{p}}",->] (1);
		\end{codi}
	\end{center}
	commutes.
\end{definition}
\begin{example}%
	For $(p,q)$=$(2,1)$, a $(2,1)$-dinatural transformation is a collection
	\[
		\Big\{\alpha_{A}: F^{A,A}_{A}\to  G^{A}_{A,A}\ \Big| \ A \in \clC_o\Big\}
	\]
	of morphisms of $\clD$ such that, for each morphism $f: A\to  B$ of $\clC$, the following hexagonal diagram commutes:
	\begin{center}
		\begin{codi}
			\obj[hexagonal=horizontal side 4em angle 60] {%
			|[overwrite=false] (3)| F^{A,A}_{A}
			&
			|[overwrite=false] (2)| G^{A}_{A,A}
			&
			\\
			|[overwrite=false] (4)| F^{B,B}_{A}
			&
			&
			|[overwrite=false] (1)| G^{A}_{B,B}
			\\
			|[overwrite=false] (5)| F^{B,B}_{B}
			&
			|[overwrite=false] (6)| G^{B}_{B,B}\mrp{.}
			&
			\\};
			\mor (3) ["\alpha_{A}",->] (2);
			\mor (5) ["\alpha_{B}"',->] (6);
			\mor (4) ["F^{f,f}_{A}",->] (3);
			\mor (4) ["F^{B,B}_{f}"',->] (5);
			\mor (6) ["G^{f}_{B,B}"',->] (1);
			\mor (2) ["G^{A}_{f,f}",->] (1);
		\end{codi}
	\end{center}
\end{example}
\begin{notation}
	We write $\pqDiNat{p}{q}(F,G)$ for the set of $(p,q)$\hyp{}dinatural transformations from $F$ to $G$.
\end{notation}
\begin{remark}
	The same convention of \cref{sec:notation} applies to morphisms as well as to objects: $F_{\bsf}^{\bsB} \defeq F_{f,\dots,f}^{B,\dots,B}$ is the morphism $F_{A,\dots,A}^{B,\dots,B} \to F_{A,\dots,A}^{B,\dots,B}$ induced by the conjoint action of $f$ in all the covariant components of $F$, and similarly for $F^{\uf}_{\bsA}, G^{\bsA}_{\bsf}$, etc.
\end{remark}
Dinatural transformations can always be composed with \emph{natural} ones of the appropriate arity, on the left and on the right.
\begin{definition}[Composing dinaturals with naturals]
	Let $F$ and $G$ be a functors of type $\typepq{p}{q}$, let $H$ and $K$ be functors of type $\typepq{q}{p}$, let $\alpha: F \Longrightarrow  G$ and $\beta: H \Longrightarrow  K$ be natural transformations, and let $\theta: G\din H$ be a $(p,q)$\hyp{}dinatural transformation.
	\begin{enumtag}{dc}
		\item \label{dc_1} The \emph{vertical composition of $\theta$ with $\alpha$} is the $(p,q)$\hyp{}dinatural transformation
		\[\theta\circ\alpha: F\din H\]
		defined as the collection
		\[\Big\{(\theta\circ\alpha)_{A}: F^{\bsA}_{\bsA}\to  H^{\bsA}_{\bsA}\ \mid \ A \in \clC_o\Big\},\]
		where $(\theta\circ\alpha)_{A}=\theta_{A}\circ\alpha^{\bsA}_{\bsA}$;
		\item \label{dc_2} The \emph{vertical composition of $\beta$ with $\theta$} is the $(p,q)$\hyp{}dinatural transformation
		\[\beta\circ\theta: G\din K\]
		defined as the collection
		\[\Big\{(\beta\circ\theta)_{A}: G^{\bsA}_{\bsA}\to  K^{\bsA}_{\bsA}\ \mid \ A \in \clC_o\Big\},\]
		where $(\beta\circ\theta)_{A}=\beta^{\bsA}_{\bsA}\circ\theta_{A}$.
	\end{enumtag}
\end{definition}
\begin{remark}
    Note that the definition of $(p,q)$-dinaturality in \cref{def:p-q-dinatural-transformation} does not allow one to construct a dinatural transformation from a natural transformation $\alpha^{\uA}_{\uB}$ by ``complete symmetrisation'' $\alpha^{\uA}_{\uB}\mapsto \alpha^{\bsA}_{\bsA}$; in fact, this request does not even make sense, as natural transformations are defined only between functors of the same type. For instance, given $F$ and $G$ of variance $(1,2)$, a natural transformation $\alpha\colon F\Longrightarrow G$ is a collection of morphisms of the form
	\[
        \alpha^{A}_{A,A}\colon F^{A}_{A,A}  \to  G^{A}_{A,A},
	\]
	rather than of the form
	\[
        \alpha^{A}_{A,A}\colon F^{A}_{A,A}  \to  G^{A,A}_{A},
	\]
    as required in \cref{def:p-q-dinatural-transformation}. More formally, the non-existence of \say{associated $(p,q)$\hyp{}dinatural transformations} for natural transformations between functors of different variance boils down to the absence of identity dinatural transformations for $p\neq q$. We note, however, that these two aspects of higher arity dinaturality are completely irrelevant for developing the theory of higher arity co/ends.
\end{remark}
\begin{proposition}\label{theta-circ-alpha-is-indeed-a-p-q-dinatural-transformation}
	$\theta\circ\alpha$ and $\beta \circ\theta$ are $(p,q)$\hyp{}dinatural transformations.
\end{proposition}
\begin{proof}
	The $(p,q)$\hyp{}dinaturality condition for $\theta\circ\alpha$ is the requirement that the boundary of the diagram
		{\small
			\begin{center}%
				\begin{tikzcd}[row sep={0.0cm,between origins}, column sep={0.0cm,between origins}, ampersand replacement=\&]
					\&[0.43301270189\OneCmAndAHalf]
					F^{\bsA}_{\bsA}
					\arrow[rr, "\alpha^{\bsA}_{\bsA}"]
					\&[1.0\OneCmAndAHalf]
					\&[0.43301270189\OneCmAndAHalf]
					G^{\bsA}_{\bsA}
					\arrow[r, "\theta_{A}"]
					\&[1.0\OneCmAndAHalf]
					H^{\bsA}_{\bsA}
					\arrow[rd]
					\&[0.43301270189\OneCmAndAHalf]
					\\[0.86602540378\OneCmAndAHalf]
					F^{\bsB}_{\bsA}
					\arrow[ru]
					\arrow[rd]
					\arrow[rr,"\alpha^{\bsB}_{\bsA}"description]
					\&[0.43301270189\OneCmAndAHalf]
					\&[1.0\OneCmAndAHalf]
					G^{\bsB}_{\bsA}
					\arrow[ru]
					\arrow[rd]
					\&[0.43301270189\OneCmAndAHalf]
					\&[1.0\OneCmAndAHalf]
					\&[0.43301270189\OneCmAndAHalf]
					H^{\bsA}_{\bsB}
					\\[0.86602540378\OneCmAndAHalf]
					\&[0.43301270189\OneCmAndAHalf]
					F^{\bsB}_{\bsB}
					\arrow[rr,"\alpha^{\bsB}_{\bsB}"']
					\&[1.0\OneCmAndAHalf]
					\&[0.43301270189\OneCmAndAHalf]
					G^{\bsB}_{\bsB}
					\arrow[r,"\theta_{B}"']
					\&[1.0\OneCmAndAHalf]
					H^{\bsB}_{\bsB}
					\arrow[ru]
					\&[0.43301270189\OneCmAndAHalf]
					\arrow[from=1-2,to=2-3,phantom,"\scriptstyle(1)"]
					\arrow[from=3-2,to=2-3,phantom,"\scriptstyle(2)"]
					\arrow[from=2-3,to=2-6,phantom,"\scriptstyle(3)"]
				\end{tikzcd}
			\end{center}}%
	commutes. Since
	\begin{enumerate}
		\item Sub-diagrams $(1)$ and $(2)$ commute by the naturality of $\alpha$, and
		\item Sub-diagram $(3)$ commutes by the $(p,q)$\hyp{}dinaturality of $\theta$,
	\end{enumerate}
	so does the boundary diagram: $\theta\circ\alpha$ is indeed a $(p,q)$\hyp{}dinatural transformation.

	Similarly for $\beta\circ\theta$: one considers instead the corresponding diagram having the form
	\[
		\scriptscriptstyle
		\begin{tikzcd}[row sep={0.0cm,between origins}, column sep={0.0cm,between origins}, ampersand replacement=\&]
			\&[0.43301270189\OneCm]
			\bullet
			\arrow[r]
			\&[1.0\OneCm]
			\bullet
			\arrow[rd]
			\arrow[rr]
			\&[0.43301270189\OneCm]
			\&[1.0\OneCm]
			\bullet
			\arrow[rd]
			\&[0.43301270189\OneCm]
			\\[0.86602540378\OneCm]
			\bullet
			\arrow[ru]
			\arrow[rd]
			\&[0.43301270189\OneCm]
			\&[1.0\OneCm]
			\&[0.43301270189\OneCm]
			\bullet
			\arrow[rr]
			\&[0.43301270189\OneCm]
			\&[1.0\OneCm]
			\bullet\mathrlap{,}
			\\[0.86602540378\OneCm]
			\&[0.43301270189\OneCm]
			\bullet
			\arrow[r]
			\&[1.0\OneCm]
			\bullet
			\arrow[ru]
			\arrow[rr]
			\&[0.43301270189\OneCm]
			\&[1.0\OneCm]
			\bullet
			\arrow[ru]
			\&[0.43301270189\OneCm]
			\arrow[from=2-1,to=2-4,phantom,"\scriptstyle(1)"]
			\arrow[from=2-4,to=1-5,phantom,"\scriptstyle(2)"]
			\arrow[from=2-4,to=3-5,phantom,"\scriptstyle(3)"]
		\end{tikzcd}
	\]
	where again each sub-diagram commutes by either the dinaturality of $\theta$ or the naturality of $\beta$.
\end{proof}
For the next proposition, recall the definition of the \emph{$(p,q)$\hyp{}diagonal functor} $\Delta_{p,q}\colon\oo{\clC} \to \pq{\clC}$ of $\clC$ introduced in \cref{def:p-q-diagonal-functor}.
\begin{proposition}[Higher arity dinaturality via ordinary dinaturality]\label{higher-arity-dinaturality-via-ordinary-dinaturality}%
	Let $F\colon\pq{\clC} \to \clD$ and $G\colon\pq{\clC}[q][p] \to \clD$ be functors. We have a natural bijection
	\begin{equation}\label{eq:pq-dinat-via-nat}
		\pqDiNat{p}{q}(F,G)\cong\pqDiNat{1}{1}\big(\pqdiag^{*}(F),\Delta_{q,p}^{*}(G)\big).
	\end{equation}
\end{proposition}
\begin{proof}%
	This is simply a matter of unwinding the definitions: since $\big(F\circ\pqdiag\big){}^{A}_{B}\defeq F^{\bsA}_{\bsB}$ (and similarly for morphisms and for $G$), it follows that a $(p,q)$\hyp{}dinatural transformation $F\din G$ is precisely a dinatural transformation $\pqdiag^{*}(F)\din\Delta_{q,p}^{*}(G)$.
\end{proof}
\subsection{Higher arity wedges}
The notion of wedge (resp., cowedge) for a diagram $D:\oo{\clC} \to \clD$ arises when assuming that the domain (resp., codomain) of a dinatural transformation to/from $D$ is constant; similarly, a $(p,q)$\hyp{}wedge (resp., $(p,q)$\hyp{}cowedge) for a diagram $D : \pq{\clC} \to \clD$ consists of a $(p,q)$\hyp{}dinatural transformation whose domain (resp., codomain) is a constant functor $\Delta_{X} : \pq{\clC}[q][p] \to \clD$ of type $\typepq{q}{p}$ with value $X\in\CatFont{D}_{o}$.
\begin{definition}\label{def:p-q-wedges}
	Let $D: \pq{\clC} \to \clD$ be a functor and let $X\in\clD_o$.
	\begin{enumtag}{cw}
		\item A \emph{$(p,q)$\hyp{}wedge for $D$ under $X$} is a $(p,q)$\hyp{}dinatural transformation $\theta:\Delta_{X}\din D$ from the constant functor of type $\typepq{q}{p}$ with value $X$ to $D$;
		\item A \emph{$(p,q)$\hyp{}cowedge for $D$ over $X$} is a $(p,q)$\hyp{}dinatural transformation $\zeta:D\din\Delta_{X}$ from $D$ to the constant functor of type $\typepq{q}{p}$ with value $X$.
	\end{enumtag}
\end{definition}
\begin{remark}[Unwinding \cref{def:p-q-wedges}]\label{unwinding-p-q-co-wedges}\leavevmode
	\begin{enumtag}{cwu}
		\item\label{unwinding-p-q-wedges}A $(p,q)$\hyp{}wedge $\theta:\Delta_{X}\din D$ is a collection
		\[\big\{\theta_{A}: X\to D^{\bsA}_{\!\bsA}\ : \ A\in\clC_{o}\big\}\]
		of morphisms of $\clD$ such that, for each morphism $f: A\to  B$ of $\clC$, the diagram
		\begin{center}
			\begin{tikzcd}[row sep={4.5em,between origins}, column sep={4.5em,between origins}]
				X
				\arrow[r, "\theta_{B}"]
				\arrow[d, "\theta_{A}"']
				&
				D^{\bsB}_{\bsB}
				\arrow[d, "D^{\bsf}_{\bsB}"]
				\\
				D^{\bsA}_{\!\bsA}
				\arrow[r, "D^{\bsA}_{\bsf}"']
				&
				D^{\bsA}_{\bsB}
			\end{tikzcd}
		\end{center}%
		commutes.
		\item\label{unwinding-p-q-cowedges}A $(p,q)$\hyp{}cowedge $\zeta:D\din\Delta_{X}$  is a collection
		\[\big\{\zeta_{A}: D^{\bsA}_{\!\bsA}\to X\ : \ A\in\clC_{o}\big\}\]
		of morphisms of $\clC$ such that, for each morphism $f: A\to  B$ of $\clC$, the diagram
		\begin{center}
			\begin{tikzcd}[row sep={4.5em,between origins}, column sep={4.5em,between origins}]
				X
				\arrow[r, "\zeta_{B}",leftarrow]
				\arrow[d, "\zeta_{A}"',leftarrow]
				&
				D^{\bsB}_{\bsB}
				\arrow[d, "D^{\bsB}_{\bsf}",leftarrow]
				\\
				D^{\bsA}_{\!\bsA}
				\arrow[r, "D^{\bsf}_{\bsA}"',leftarrow]
				&
				D^{\bsB}_{\bsA}
			\end{tikzcd}
		\end{center}%
		commutes.
	\end{enumtag}
\end{remark}
\begin{remark}\label{sake-of-clarity}
	For the sake of clarity, we remind the reader that in our notation, the commutativity of the diagram in \cref{unwinding-p-q-wedges} of \cref{unwinding-p-q-co-wedges} above means that, for every $f \in \Mor(\clC)$,
	\[ D^{\tpl{f}[1][p]}_{\tpl{B}[1][q]} \circ \theta_B = \theta_A \circ D^{\tpl{A}[1][p]}_{\tpl{f}[1][q]} \]
	where $f_i \equiv f$, $A_i \equiv \text{src } f, B_i\equiv \text{trg } f$ are the domain and codomain of $f$, for every index in the relevant range, and 
	$\theta_A : X \to D_{A,\dots,A}^{A,\dots,A}$ is a morphism in $\clD$.
\end{remark}
\begin{notation}%
	We write $\pqWedges{p}{q}{X}{D}$ for the set of $(p,q)$\hyp{}wedges of $X$ under $D$, and similarly, $\pqCoWedges{p}{q}{X}{D}$ for $(p,q)$\hyp{}cowedges.
\end{notation}
\begin{proposition}%
	Let $D:\pq{\clC}\to\clD$ be a functor.%
	\begin{enumtag}{wdf}
		\item\label{wdf-1}The assignment $X\mapsto\pqWedges{p}{q}{X}{D}$ defines a presheaf
		\[\pqWedges{p}{q}{(-)}{D}:\clC^{\op}\to \Sets.\]
		\item\label{wdf-2}The assignment $X\mapsto\pqCoWedges{p}{q}{X}{D}$ defines a functor
		\[\pqCoWedges{p}{q}{(-)}{D}:\clC\to \Sets.\]
	\end{enumtag}
\end{proposition}
\begin{proof}%
	\cref{wdf-1}: Let $f: X\to Y$ be a morphism of $\clC$. We have a map
	\begin{center}
		\newsavebox{\BoxTwo}
		\savebox{\BoxTwo}{
			$\mspace{-10.0mu}$
			\begin{tikzcd}[row sep=1.8em, column sep=1.8em, ampersand replacement=\&,cramped]
				X
				\arrow[r, "f"]
				\&
				Y
			\end{tikzcd}
			$\mspace{-10.0mu}$
		}
		\begin{tikzcd}[row sep=0.0em, column sep=2.7em, ampersand replacement=\&]
			\mathllap{\pqWedges{p}{q}{f}{D}}\colon \pqWedges{p}{q}{Y}{D}
			\arrow[r]
			\&
			\pqWedges{p}{q}{X}{D}
			\\
			\left(Y\din D\right)
			\arrow[r, mapsto]
			\&
			\left(\usebox{\BoxTwo}\din D\right)\mathrlap{,}
		\end{tikzcd}
	\end{center}%
	where we have used \cref{theta-circ-alpha-is-indeed-a-p-q-dinatural-transformation}. As it is clear that this construction preserves composition and identities, we get our desired presheaf.

	\cref{wdf-2}: This is dual to \cref{wdf-1}.
\end{proof}
\begin{proposition}\label{functoriality-of-wedges-ii}
	Let $D:\pq{\clC}\to\clD$ be a functor.%
	\footnote{%
		More generally, the assignments $F,G\mapsto\pqDiNat{p}{q}(F,G)$ define functors
		\begin{align*}
			 & \pqDiNat{p}{q}(-,-) \colon \Cat(\pq{\clC},\clD)^{\op}\times\Cat(\pq{\clC},\clD) \to \Sets, \\
			 & \pqDiNat{p}{q}(F,-) \colon \Cat(\pq{\clC},\clD) \to \Sets,                                 \\
			 & \pqDiNat{p}{q}(-,G) \colon \Cat(\pq{\clC},\clD)^{\op} \to \Sets.
		\end{align*}
	}%
	\begin{enumtag}{wdf'}
		\item The assignment $D\mapsto\pqWedges{p}{q}{X}{D}$ defines a functor
		\[\pqWedgesFunctor{p}{q}{X}:\Cat(\pq{\clC},\clD)\to \clD.\]
		\item The assignment $D\mapsto\pqCoWedges{p}{q}{X}{D}$ defines a functor
		\[\pqCoWedgesFunctor{p}{q}{X}:\Cat(\pq{\clC},\clD)\to \clD.\]
	\end{enumtag}
\end{proposition}
\begin{proof}%
	Let $\alpha: D \to  D'$ be a natural transformation. We have a map
	\begin{center}
		\newsavebox{\BoxThree}
		\savebox{\BoxThree}{
			$\mspace{-10.0mu}$
			\begin{tikzcd}[row sep=1.8em, column sep=1.8em, ampersand replacement=\&,cramped]
				D
				\arrow[r, "\alpha",Rightarrow]
				\&
				D'
			\end{tikzcd}
			$\mspace{-10.0mu}$
		}
		\begin{tikzcd}[row sep=0.0em, column sep=2.7em, ampersand replacement=\&]
			\mathllap{\pqWedges{p}{q}{X}{\alpha} : }\pqWedges{p}{q}{X}{D}
			\arrow[r]
			\&
			\pqWedges{p}{q}{X}{D'}
			\\
			\left(X\din D\right)
			\arrow[r, mapsto]
			\&
			\left(X\din\usebox{\BoxThree}\right)\mspace{-3.5mu}\mathrlap{,}
		\end{tikzcd}
	\end{center}%
	where we have used \cref{theta-circ-alpha-is-indeed-a-p-q-dinatural-transformation}. As it is clear that this construction preserves composition and identities, we get our desired functor.
\end{proof}
\begin{definition}
	Let $\theta : X\din D$ be a $(p,q)$\hyp{}wedge, and $\zeta : D \din Y$ be a $(p,q)$\hyp{}cowedge;
	\begin{enumtag}{pc}
		\item \label{pc_1} The \emph{$(p,q)$\hyp{}wedge post-composition natural transformation} associated to a $(p,q)$\hyp{}wedge $\theta: X\din D$ is the natural transformation
		\[\theta_{*}: h_{X}\Longrightarrow \pqWedges{p}{q}{(-)}{D}\]
		consisting of the collection
		\[
			\big\{\theta_{*,A}: h_{X}(A)\to \pqWedges{p}{q}{A}{D} : A\in \clC_o\big\},
		\]
		where $\theta_{*,A}$ is the map
		\begin{center}
			\begin{tikzcd}[row sep=0.0em, column sep=2.7em, ampersand replacement=\&]
                \mathllap{\theta_{*,A}\colon}h_{X}(A)
				\arrow[r]
				\&
				\pqWedges{p}{q}{A}{D}
				\\
				\var[f]{A}{X}
				\arrow[r, mapsto]
				\&
                \left(\Delta_{A}\xrightarrow{\Delta_{f}}\Delta_{X}\din D\right)\mspace{-4.5mu}\mathrlap{.}
			\end{tikzcd}
		\end{center}
		\item \label{pc_2} The \emph{$(p,q)$\hyp{}cowedge pre-composition natural transformation} associated to a $(p,q)$\hyp{}cowedge $\zeta : D\din Y$ is the natural transformation
		\[
			\zeta^{*}: h^{Y} \to \pqCoWedges{p}{q}{(-)}{D}
		\]
		consisting of the collection
		\[
			\big\{\zeta^{*}_{A}: h^{Y}(A)\to \pqCoWedges{p}{q}{A}{D} : A\in \clC_o \big\},
		\]
		where $\zeta^{*}_{A}$ is the map
		\begin{center}
			\begin{tikzcd}[row sep=0.0em, column sep=2.7em, ampersand replacement=\&]
                \mathllap{\zeta^{*}_{A}\colon}h^{Y}(A)
				\arrow[r]
				\&
				\pqCoWedges{p}{q}{A}{D}
				\\
				\var[f]{Y}{A}
				\arrow[r, mapsto]
				\&
                \left( D \din \Delta_Y \xrightarrow{\Delta_{f}} \Delta_A \right)\mspace{-4.5mu}\mathrlap{.}
			\end{tikzcd}
		\end{center}
	\end{enumtag}
\end{definition}
The notion of dinaturality introduced in \cite{2007.07576} is in fact more general, as \cite[Definition 2.4]{2007.07576} introduces what would be called here a \emph{$(p,q)$\hyp{}to-$(r,s)$-dinatural transformation}. Recall that in the setting of \cite{2007.07576}, the tuple of powers of $\clC$ is unbiased. Their definition is as follows:
\begin{definition}\label{def:transformation}
	Let $\alpha,\beta$ be two multi-indices, and let $F : \clC^\alpha \to \clD$, $G : \clC^\beta \to \clD$ be functors. A \emph{transformation} $\phi : F \to G$ \emph{of type}
	$
		\begin{tikzcd}[cramped,sep=small]
			\ell\alpha \ar[r,"\sigma"] & n & \ell\beta \ar[l,"\tau"']
		\end{tikzcd}
	$
	(with $n = \ell \bsA$ a positive integer) is a family of morphisms in $\clD$
	\[
		\bigl( \phi_{\bsA} : F(\bsA\sigma) \to G(\bsA\tau) \bigr)_{\bsA \in \clC^n}.
	\]
\end{definition}
This translates into a family $\phi_{A_1,\dots,A_n} : F(A_{\sigma 1}, \dots, A_{\sigma\ell\alpha}) \to G(A_{\tau1},\dots,A_{\tau\ell\beta})$).

Notice that $\alpha$ and $\beta$ are \emph{different} multi-indices in this definition, and $\sigma$, $\tau$ need not be injective or surjective, so we may have repeated or unused variables.
\begin{notation}%
    Before proceeding, we need two pieces of notation:
    \begin{enumerate}
        \item\label{n_4}  Substitution of an object at a prescribed index
        \[\uA[X/i] \defeq (A_1,\dots A_{i-1},X,A_{i+1},\dots, A_n).\]
        \item\label{n_5} Substitution of a tuple at a prescribed tuple of indices
        \[\uA[X_1,\dots,X_r/i_1,\dots,i_r] \defeq ((\uA[X_1/i_1])[X_2/i_2]\cdots)[X_r/i_r].\]
    \end{enumerate}
\end{notation}
\begin{definition}\label{def:dinaturality-i-th-variable}
	Let $\phi = (\phi_{A_1,\dots,A_n}) : F \to G$ be a transformation. For $i \in \{1,\dots,n\}$, we say that $\phi$ is \emph{dinatural in $A_i$} (or, more precisely, \emph{dinatural in its $i$-th variable}) if and only if for all $A_1,\dots,A_{i-1}, A_{i+1},\dots,A_n$ objects of $\clC$ and for all $f : A \to B$ in $\clC$ the following hexagon commutes:

	{\footnotesize
	\begin{center}
		\begin{codi}
			\obj [hexagonal=horizontal side 2cm angle 60]
			{
				|(1)| F(\subst \bsA A i \sigma) && |(4)| G(\subst \bsA A i \tau) & \\
				|(2)| F(\substMV \bsA B A i \sigma) &&&  |(5)| G(\substMV \bsA A B i \tau) \\
				|(3)| F(\subst \bsA B i \sigma) && |(6)| G(\subst \bsA B i \tau) & \\
			};
			\mor (2) ["F(\substMV \bsA f A i \sigma)",->] (1) ["\phi_{\subst \bsA A i}",->]  (4) ["G(\substMV \bsA A f i \tau)",->] (5);
			\mor (2) ["F(\substMV \bsA B f i \sigma)"',->] (3) ["\phi_{\subst \bsA B i}"',->] (6) ["G(\substMV \bsA f B i \tau)"',->] (5);
		\end{codi}
	\end{center}
	}
	where $\bsA$ is the $n$-tuple $(A_1,\dots,A_n)$ of the objects above with an additional (unused in this definition) object $A_i$ of $\clC$.
\end{definition}
As far as higher arity co/wedges (i.e.\ higher arity dinatural transformations from/to a constant functor) are concerned, however, the notions of $(p,q)$\hyp{}dinaturality and $(p,q)$\hyp{}to-$(r,s)$-dinaturality agree and yield the same theory of higher arity co/ends.

\section{Higher Arity Ends}\label{sec:higher-arity-ends}
\subsection{Basic definitions}
\begin{definition}\label{def:p-q-ends}
	Let $D:\pq{\clC}\to \clD$ be a functor.
	\begin{enumtag}{pq}
		\item\label{pqe}The \emph{$(p,q)$\hyp{}end} of $D$ is, if it exists, the pair $\smash{\left(\pqEnd{p}{q}{A\in\clC}D^{\bsA}_{\!\bsA},\omega\right)}$ formed by an object
		\[\pqEnd{p}{q}{A\in\clC}D^{\bsA}_{\!\bsA}\]
		of $\clD$, and a $(p,q)$\hyp{}wedge
		\[\omega:{\pqEnd{p}{q}{A\in\clC}D^{\bsA}_{\!\bsA}}\din D\]
		for ${\pqEnd{p}{q}{A\in\clC}D^{\bsA}_{\!\bsA}}$ over $D$, such that the $(p,q)$\hyp{}wedge post-composition natural transformation
		\[\omega_{*} : \ph\left(-,\pqEnd{p}{q}{A\in\clC}D^{\bsA}_{\!\bsA}\right)\Longrightarrow\pqWedges{p}{q}{(-)}{D}\]
		is a natural isomorphism.
		\par\vspace*{0.5\baselineskip}
		\item\label{pqc}The \emph{$(p,q)$\hyp{}coend} of $D$ is, if it exists, the pair $\smash{\left(\pqCoend{p}{q}{A\in\clC}D^{\bsA}_{\!\bsA},\xi\right)}$ formed by an object
		\[ \pqCoend{p}{q}{A\in\clC}D^{\bsA}_{\!\bsA} \]
		of $\clD$, and a $(p,q)$\hyp{}cowedge
		\[ \xi:D\din {\pqCoend{p}{q}{A\in\clC}D^{\bsA}_{\!\bsA}} \]
		for ${\pqCoend{p}{q}{A}D^{\bsA}_{\!\bsA}}$ under $D$, such that the $(p,q)$\hyp{}cowedge post-composition natural transformation
		\[\xi^{*} : \ph\left(\pqCoend{p}{q}{A\in\clC}D^{\bsA}_{\!\bsA},-\right)\Longrightarrow\pqCoWedges{p}{q}{(-)}{D}\]
		is a natural isomorphism.
	\end{enumtag}
\end{definition}
We follow the customary abuse of notation of denoting the $(p,q)$\hyp{}end of $D$ as just the tip $\pqEnd{p}{q}{A\in\clC}D^{\bsA}_{\!\bsA}$ of the terminal $(p,q)$\hyp{}wedge $\omega$. The object $\pqEnd{p}{q}{A\in\clC}D^{\bsA}_{\!\bsA}$ can also be shortly denoted as $\pqEnd{p}{q}{A}D$, or $\pqEnd{p}{q}{}D$.
\begin{remark}
	The co/representability conditions of \cref{def:p-q-ends} unwind as the following universal properties:
	\begin{enumtag}{upq}
		\item The $(p,q)$\hyp{}end of $D$ consists of a pair $\smash{\left(\pqEnd{p}{q}{A\in\clC}D^{\bsA}_{\!\bsA},\omega\right)}$ with
		\begin{enumerate}[label=\arabic*)]
			\item $\pqEnd{p}{q}{A\in\clC}D^{\bsA}_{\!\bsA}$ an object of $\clD$, and
			\item $\omega$ a natural isomorphism with components
			      \[\omega_E : \clD\left(E, \pqEnd{p}{q}{A\in\clC}D^{\bsA}_{\!\bsA}\right) \cong \pqWedges{p}{q}{E}{D}.\]
		\end{enumerate}
		The family of such morphisms of $\clD$ is such that evaluating the isomorphism $\omega_E$ at the identity of $E=\pqEnd{p}{q}{A\in\clC}D^{\bsA}_{\!\bsA}$ gives a $(p,q)$\hyp{}wedge
		\[\Big\{\omega_{A} :\pqEnd{p}{q}{A\in\clC}D^{\bsA}_{\!\bsA}\to  D^{\bsA}_{\!\bsA}\ :\ A\in \clC_o\Big\}\]
		indexed by the objects of $\clC$. This $(p,q)$\hyp{}wedge has the following universal property:

		\begin{itemize}
			\item[$(\star)$]Given another such pair $(E,\theta)$, there exists a unique morphism $E\xlongdashrightarrow{\exists!}\pqEnd{p}{q}{A}D^{\bsA}_{\!\bsA}$ filling the diagram
			      \begin{center}
				      \begin{codi}
					      \obj [hexagonal=horizontal side 2.25cm angle 60] {
					      E & {} & \\
					      {} & |(end)| \pqEnd{p}{q}{A}D^{\bsA}_{\!\bsA} & |(DBB)| D_{\bsB}^{\bsB} \\
					      |(DAA)| D_{\bsA}^{\bsA} & |(DAB)| D^{\bsA}_{\bsB} & \\
					      };
					      \mor E "\theta_B":[bend left,->] DBB "D^{\bsf}_{\bsB}":-> DAB;
					      \mor[swap] * "\theta_A":[bend right, ->] DAA "D_{\bsf}^{\bsA}":-> DAB;
					      \mor * "\exists!":[dashed,->] end "\omega_B":-> DBB;
					      \mor[swap] end "\omega_A":-> DAA;
				      \end{codi}
			      \end{center}
		\end{itemize}
		\item The $(p,q)$\hyp{}coend of $D$ consists of a pair $\smash{\left(\pqCoend{p}{q}{A\in\clC}D^{\bsA}_{\!\bsA},\xi\right)}$ with
		\begin{enumerate}[label=\arabic*)]
			\item $\pqCoend{p}{q}{A\in\clC}D^{\bsA}_{\!\bsA}$ an object of $\clD$, and
			\item $\xi$ a natural isomorphism with components
			      \[\xi_E : \clD\left(\pqCoend{p}{q}{A\in\clC}D^{\bsA}_{\!\bsA},E\right) \cong \pqCoWedges{p}{q}{E}{D}.\]
		\end{enumerate}
		The family of such morphisms of $\clD$ is such that evaluating the isomorphism $\xi_D$ at the identity of $C=\pqCoend{p}{q}{A\in\clC}D^{\bsA}_{\!\bsA}$ gives a $(p,q)$\hyp{}cowedge
		\[\Big\{\xi_{A} :D^{\bsA}_{\!\bsA}\to  \pqCoend{p}{q}{A\in\clC}D^{\bsA}_{\!\bsA}\ :\ A\in \clC_o\Big\}\]
		indexed by the objects of $\clC$. This $(p,q)$\hyp{}cowedge has the following universal property:
		\begin{itemize}
			\item[$(\star)$]Given another such pair $(C,\zeta)$, there exists a unique morphism $\pqCoend{p}{q}{A}D^{\bsA}_{\!\bsA}\xlongdashrightarrow{\exists!}C$ filling the diagram
			      \begin{center}
				      \begin{codi}
					      \obj [hexagonal=horizontal side 2.25cm angle 60] {
					      C & {} & \\
					      {} & |(end)| \pqCoend{p}{q}{A}D^{\bsA}_{\!\bsA} & |(DBB)| D^{\bsB}_{\!\bsB} \\
					      |(DAA)| D^{\bsA}_{\!\bsA} & |(DAB)| D^{\bsB}_{\!\bsA} & \\
					      };
					      \mor E "\zeta_B":[bend left,<-] DBB "D_{\bsf}^{\bsB}":<- DAB;
					      \mor[swap] * "\zeta_A":[bend right, <-] DAA "D^{\bsf}_{\bsA}":<- DAB;
					      \mor * "\exists!":[dashed,<-] end "\xi_B":<- DBB;
					      \mor[swap] end "\xi_A":<- DAA;
				      \end{codi}
			      \end{center}
		\end{itemize}
	\end{enumtag}
\end{remark}
\begin{remark}\label{itsa_terminal}
	This means that the $(p,q)$\hyp{}end of $D$ is the terminal object of the category of wedges of $D$, whose morphisms $h : (\alpha :\Delta_X \din D) \to (\beta :\Delta_Y \din D)$ are defined as the morphisms $h : X \to Y$ of $\clD$ such that for every $A\in \clC_o$ one has $\beta_A \circ h = \alpha_A$:
	\[
		\begin{tikzcd}[row sep={3.6em,between origins}, column sep={3.6em,between origins}]
			X \ar[rr,"h"]\ar[dr, "\alpha_A"']&& Y \ar[dl, "\beta_A"]\\
			& D^{\bsA}_{\!\bsA}\mathrlap{.}
		\end{tikzcd}
	\]
\end{remark}
\cref{itsa_terminal} can be dualised to define $(p,q)$\hyp{}\emph{co}ends as initial $(p,q)$\hyp{}cowedges. This is straightforward, and we leave it to the reader to spell out.

In the following proposition, we will make use of the $(p,q)$\hyp{}diagonal functor $\pqdiag$ introduced in \cref{def:p-q-diagonal-functor}, and duplicated in the following
\begin{notation}\label{not:dummy}
	We say that
	\begin{itemize}
		\item A functor $F : \pq{\clC}[p+r][q+s] \to \clD$ is $(r,s)$-\emph{mute} if it factors through the canonical projection $\pi_{r,s} : \pq{\clC}[p+r][q+s] \to \pq{\clC}$ that cancels the last $r$ contravariant components, and the last $s$ covariant components.
        \item Given a functor $F : \pq{\clC} \to \clD$, we write $\dummy F : \pq{\clC}[p+r][q+s]\longrightarrow\CatFont{D}$ for the composition $\pq{\clC}[p+r][q+s]\xrightarrow{\pi_{r,s}} \pq{\clC} \xrightarrow{F} \clD$; this promotes every functor of type $\left[\pqMat{p\\q}\right]$ to an $(r,s)$-mute one.
	\end{itemize}
	It's immediate that every functor that is mute in \emph{some} of its variables can be made into an $(r,s)$-mute one by suitably reshuffling its arguments.
\end{notation}
\begin{proposition}[Properties of $(p,q)$\hyp{}ends and $(p,q)$\hyp{}coends]\label{prop:properties-of-p-q-ends}
	Let $D :\pq{\clC}\to \clD$ be a functor.
	\begin{enumtag}{pe}
		\item\label{functoriality-of-p-q-ends}\SloganFont{Functoriality}Let $D :\pq{\clC}\to \clD$ be a functor. The assignments $D\mapsto\pqEnd{p}{q}{A}D^{\bsA}_{\!\bsA},\pqCoend{p}{q}{A}D^{\bsA}_{\!\bsA}$ define functors
		\begin{align*}
			\pqEnd{p}{q}{A\in\clC}   & : \Cat\big(\pq{\clC},\clD\big) \to \clD, \\
			\pqCoend{p}{q}{A\in\clC} & : \Cat\big(\pq{\clC},\clD\big) \to \clD
		\end{align*}
		with domain the category of functors from $\clC$ of type $\left[\pqMat{p\\q}\right]$ to $\clD$ and natural transformations between them.
		\item\label{p-q-wedges-and-p-q-diagonals}\SloganFont{$(p,q)$\hyp{}Wedges and $(p,q)$\hyp{}diagonals}For each $X\in\clC_o$ we have natural isomorphisms
		\begin{align*}
			\pqWedges{p}{q}{(-)}{D}   & \cong   \Wedges{(-)}{\Delta_{p,q}^{*}(D)}, \\
			\pqCoWedges{p}{q}{(-)}{D} & \cong \Cowedges{(-)}{\Delta_{p,q}^{*}(D)}.
		\end{align*}
		where $\pqdiag$ is the functor introduced in \cref{def:p-q-diagonal-functor}.
		\item\label{p-q-ends-as-ordinary-ends}\SloganFont{$(p,q)$\hyp{}Ends as ordinary ends}We have natural isomorphisms
		\begin{align*}
			\pqEnd{p}{q}{A\in\clC}D^{\bsA}_{\!\bsA}   & \cong \int_{A\in\clC}\Delta_{p,q}^{*}(D)^{A}_{A}, \\
			\pqCoend{p}{q}{A\in\clC}D^{\bsA}_{\!\bsA} & \cong \int^{A\in\clC}\Delta_{p,q}^{*}(D)^{A}_{A}.
		\end{align*}
		where $\pqdiag$ is the functor introduced in \cref{def:p-q-diagonal-functor}. In other words, the $(p,q)$\hyp{}end functor factors as a composition
		\begin{diagram*}
			\begin{tikzcd}[row sep={11.7em,between origins}, column sep={11.7em,between origins},  ampersand replacement=\&]
				\Fun(\pq{\clC},\clD)
				\arrow[r, "\Delta^{(p,q)}_{*}"]
				\&
				\Fun(\clC^{\op}\times\clC,\clD)
				\arrow[r, "\int_{A}"]
				\&
				\clD\mathrlap{,}
			\end{tikzcd}
		\end{diagram*}%
		and similarly so do $(p,q)$\hyp{}coends.
		\item\label{p-q-ends-as-limits}\SloganFont{$(p,q)$\hyp{}Ends as limits} The $(p,q)$\hyp{}end and $(p,q)$\hyp{}coend of $D$ fit respectively into an equaliser and into a coequaliser diagram
		\begin{gather*}
			\begin{tikzcd}[row sep=0cm,ampersand replacement=\&]
				\pqEnd{p}{q}{A\in\clC}D^{\bsA}_{\!\bsA}
				\arrow[r,outer sep=-0.1em]
				\&
				\displaystyle\prod_{\mathclap{A\in\clC_o}} D^{\bsA}_{\!\bsA}
				\arrow[r,shift left =1.0, "\lambda"]
				\arrow[r,shift right=1.0, "\rho"']
				\&
				\displaystyle\prod_{\mathclap{u\colon A\to B}} D^{\bsA}_{\bsB}\\
				\& \displaystyle\coprod_{\mathclap{u\colon A\to B}}D^{\bsB}_{\bsA}
				\arrow[r,shift left =1.0,"\lambda'"]
				\arrow[r,shift right=1.0,"\rho'"']
				\&
				\displaystyle\coprod_{\mathclap{A\in\clC_o}}D^{\bsA}_{\!\bsA}
				\arrow[r,outer sep=-0.1em]
				\&
				\pqCoend{p}{q}{A\in\clC}D^{\bsA}_{\!\bsA}
			\end{tikzcd}
		\end{gather*}
		for suitable maps $\lambda,\rho,\lambda',\rho'$, induced by the morphisms $D^{\bsA}_{\bsu}, D^{\bsu}_{\bsB}$.
		\item\label{p-q-ends-as-limits-again}\SloganFont{$(p,q)$\hyp{}Ends as limits, again}We have natural isomorphisms
		\begin{align*}
			\pqEnd{p}{q}{A\in\clC}D^{\bsA}_{\!\bsA}   & \cong \lim\Big(\Tw{\clC}\xloongtwoheadsrightarrow{\Forgetful_{p,q}}\pq{\clC}\xloongrightarrow{D}\clD\Big),   \\
			\pqCoend{p}{q}{A\in\clC}D^{\bsA}_{\!\bsA} & \cong \colim\Big(\Tw{\clC}\xloongtwoheadsrightarrow{\Forgetful_{p,q}}\pq{\clC}\xloongrightarrow{D}\clD\Big),
		\end{align*}
		where $\Forgetful_{p,q}\colon\Tw{\clC}\to \pq{\clC}$ is the composition $\Delta^{(p,q)}\circ\Forgetful$, with $\Forgetful$ the usual forgetful functor from $\Tw{\clC}$ to $\clC^{\op}\times\clC$. Explicitly, $\Forgetful_{p,q}$ is the functor
		\[
			\begin{tikzcd}[row sep=0cm]
                \mathllap{\Forgetful_{p,q}\colon}
				\Tw{\clC} \ar[r] & \pq{\clC}\mrp{,}\\
                \var[f]{A}{B} \ar[r,mapsto] & (\boldsymbol{A},\boldsymbol{B})\mrp{,}\\
				\left[
					\begin{smallmatrix}
						& A & \xrightarrow{f}   & B & \\
						\phi\kern-.5em & \uparrow && \downarrow & \kern-.5em\psi \\
						& C & \xrightarrow[g]{} & D
					\end{smallmatrix}
					\right]
				\ar[r,mapsto] & (\boldsymbol{\phi},\boldsymbol{\psi})\mrp{.}
			\end{tikzcd}
		\]
		\item\label{p-q-ends-as-limits-yet-again}\SloganFont{$(p,q)$\hyp{}Ends as limits, yet again}There exists a category $\pqTw{p}{q}{\clC}$ together with a universal fibration
		\[\Forgetful\colon\pqTw{p}{q}{\clC}\longtwoheadsrightarrow\pq{\clC}\]
		inducing natural isomorphisms
		\begin{align*}
			\pqEnd{p}{q}{A\in\clC}D^{\bsA}_{\!\bsA}   & \cong \lim\Big(\pqTw{p}{q}{\clC}\xloongtwoheadsrightarrow{\Forgetful}\pq{\clC}\xloongrightarrow{D}\clD\Big),   \\
			\pqCoend{p}{q}{A\in\clC}D^{\bsA}_{\!\bsA} & \cong \colim\Big(\pqTw{p}{q}{\clC}\xloongtwoheadsrightarrow{\Forgetful}\pq{\clC}\xloongrightarrow{D}\clD\Big).
		\end{align*}
		\item\label{p-q-ends-as-p-plus-r-q-plus-s-ends}\SloganFont{$(p,q)$\hyp{}Ends as $(p+r,q+s)$-ends}We have
		\begin{align*}
			\pqEnd{p}{q}{A\in\clC}D^{\bsA}_{\!\bsA}   & \cong \pqEnd{p+r}{q+s}{A\in\clC}\dummy(D)^{\bsA}_{\bsA},   \\
			\pqCoend{p}{q}{A\in\clC}D^{\bsA}_{\!\bsA} & \cong \pqCoend{p+r}{q+s}{A\in\clC}\dummy(D)^{\bsA}_{\bsA},
		\end{align*}
		where $\dummyF$ is the functor introduced in \cref{not:dummy}.
		\item\label{p-q-ends-commute-with-homs}\SloganFont{Commutativity of $(p,q)$\hyp{}ends with homs}We have natural isomorphisms
		\begin{align*}
			\clD\left(-,\pqEnd{p}{q}{A\in\clC}D^{\bsA}_{\!\bsA}\right)   & \cong \pqEnd{p}{q}{A\in\clC}\clD\Big(-,D^{\bsA}_{\!\bsA}\Big)  \\
			\clD\left(\pqCoend{p}{q}{A\in\clC}D^{\bsA}_{\!\bsA},-\right) & \cong \pqEnd{q}{p}{A\in\clC}\clD\Big(D^{\bsA}_{\!\bsA},-\Big).
		\end{align*}
	\end{enumtag}
\end{proposition}
\begin{proof}%
	We provide proofs of the above statements for $(p,q)$\hyp{}ends only; the case of $(p,q)$\hyp{}coends is dual.

	\cref{functoriality-of-p-q-ends}: Let $\alpha : D\Longrightarrow D'$ be a natural transformation and consider the composition
	\begin{center}
		\begin{tikzcd}[row sep=3.6em, column sep=3.6em, ampersand replacement=\&]
			\ph\Big(-,\pqEnd{p}{q}{A}D^{\bsA}_{\!\bsA}\Big)
			\arrow[r,"(\omega_{D})_{*}"]
			\&
			\pqWedges{p}{q}{X}{D}
			\arrow[r,"\pqWedges{p}{q}{X}{\alpha}"]
			\&[3.0em]
			\pqWedges{p}{q}{X}{D'}
			\arrow[r,"(\omega_{D'})_{*}"]
			\&
			\ph\Big(-,\pqEnd{p}{q}{A}(D')^{\bsA}_{\bsA}\Big)\mathrlap{,}
		\end{tikzcd}
	\end{center}
	where we have used \cref{functoriality-of-wedges-ii}. This gives us a morphism between the representable functors associated to the $(p,q)$\hyp{}ends of $D$ and $D'$. The Yoneda lemma now yields a morphism
	\[\pqEnd{p}{q}{A}D^{\bsA}_{\!\bsA}\to \pqEnd{p}{q}{A}(D')^{\bsA}_{\bsA}.\]
	between the $(p,q)$\hyp{}ends. Since all constructions involved are functorial, it follows that $(p,q)$\hyp{}ends preserve composition and identities, and hence define a functor.

	\cref{p-q-wedges-and-p-q-diagonals}: This is the special case of \cref{higher-arity-dinaturality-via-ordinary-dinaturality} where $F=\Delta_{(-)}$.

	\cref{p-q-ends-as-ordinary-ends}: We have
	\[
		\ph\left(-,\pqEnd{p}{q}{A\in\clC}D^{\bsA}_{\!\bsA}\right)%
		\cong%
		\pqWedges{p}{q}{(-)}{D}%
		\cong%
		\Wedges{(-)}{\Delta^{(p,q)}_{*}(D)}
		\cong%
		\ph\left(-,\int_{A\in\clC}\Delta_{p,q}^{*}(D)^{\bsA}_{\bsA}\right)\mspace{-3.0mu},%
	\]
	from which the result follows from the Yoneda lemma.

	\cref{p-q-ends-as-limits}: This is again a combination of \cref{p-q-ends-as-ordinary-ends} with the ``products-and-equalisers'' formula for ends.

	\cref{p-q-ends-as-limits-again}: This is just a combination of \cref{p-q-ends-as-ordinary-ends} with the usual formula computing ends as limits of diagrams from the twisted arrow category.

	\cref{p-q-ends-as-limits-yet-again}: This problem is studied in \cref{higher-arity-twisted-arrow-categories}, with the statement of \cref{p-q-ends-as-limits-yet-again} being proved in \cref{p-q-ends-from-p-q-tw-c}.

	\cref{p-q-ends-as-p-plus-r-q-plus-s-ends}: We have
	\[
		\ph\left(-,\pqEnd{p}{q}{A\in\clC}D^{\bsA}_{\!\bsA}\right)%
		\cong%
		\pqWedges{p}{q}{(-)}{D}%
		\cong%
		\pqWedges{p+r}{q+s}{(-)}{\dummy(D)}%
		\cong%
		\ph\left(-,\pqEnd{p+r}{q+s}{A\in\clC}\dummy(D)^{\bsA}_{\bsA}\right)\mspace{-5.0mu},%
	\]
	from which the result follows from the Yoneda lemma.

	\cref{p-q-ends-commute-with-homs}: This follows from \cref{p-q-ends-as-ordinary-ends} and the fact that co/ends commute with $\Hom$s.\qedhere
\end{proof}
\subsection{Adjoints and the Fubini rule}\label{adjoints-and-the-fubini-rule}
The scope of this section is to prove that $(p,q)$\hyp{}ends are right adjoints (and, dually, that $(p,q)$\hyp{}coends are left adjoints), and from this to derive a Fubini rule. Although the proofs are elementary, these properties of $(p,q)$\hyp{}ends are more subtle to assess and require a certain amount of new terminology. Thus we separate them from the above list.
Let's start with a simple definition:
\begin{definition}[The $\hom_\Pi$ functor]\label{prod_hom}
	Let $p,q\geq1$ be natural numbers; let's define a functor
	\[\hom_{\Pi,p,q} : \pq{\clC} \to \Set\]
	by sending a pair of tuples $(\uA,\uB)$ to the product $\prod_{i,j=1}^{p,q}\hom_{\clC}(A_i,B_j)$, namely to the iterated product $\prod_{i=1}^p\prod_{j=1}^q(A_i,B_j)$.
\end{definition}
\begin{remark}
	If $\CatFont{C}$ has finite products and finite coproducts, then we have a canonical factorisation
	\[
		\begin{tikzcd}[row sep={7.2em,between origins}, column sep={7.2em,between origins},  ampersand replacement=\&]
			\pq{\clC}
			\arrow[r, "\M_{p,q}"]
			\arrow[rr, to path={%
						|- ([yshift=-2.0ex]\tikztotarget.south)node[near end,below]{$\scriptstyle\Hom_{\Pi,p,q}$}
						-- (\tikztotarget)}]
			\&
			\clC^{\op}\times\clC
			\arrow[r, "\Hom"]
			\&
			\Sets.
		\end{tikzcd}
	\]
	where $\M_{p,q}$ is the functor of \cref{biprod-p-q}.
\end{remark}
\begin{proposition}\label{dis}
	Let $G : \pq{\clC} \to \Set$ be a functor. There is an isomorphism, natural in $G$,
	\[ \pqDiNat{p}{q}(\pt,G)\cong \Nat(\hom_{\Pi,p,q},G) \]
\end{proposition}
The proof of \cref{dis} requires several lemmas (\cref{pink,pank,ponk}), which we now discuss.
\begin{lemma}\label{pink}
	Let $F,G : \clC^{\op}\times\clC \to \Set$ be functors. We have a natural isomorphism
	\[
		\DiNat(F,G)\cong \Nat(\ph,\left[F,G\right]).
	\]
\end{lemma}
\begin{proof}%
	The proof is a formal derivation and mimics \cref{dinatural-transformations-as-a-2-2-end}:
	\begin{align*}
		\DiNat(F,G) & \cong \int_{A\in\clC}\hom_{\clD}\big(F^{A}_{A},G^{A}_{A}\big)                                                     \\
		            & \cong \int_{A,B\in\clC}\left[\hom_{\clC}(A,B),\hom_{\clD}\big(F^{B}_{A},G^{A}_{B}\big)\right]                     \\
		            & \cong \Nat\left(\hom_{\clC}(-_{1},-_{2}),\hom_{\clD}\big(F^{-_{2}}_{-_{1}},G^{-_{1}}_{-_{2}}\big)\right).\qedhere
	\end{align*}
\end{proof}%
\begin{lemma}\label{pank}%
	Let $G\colon\pq{\clC}[q][p]\to \clD$ be a functor. If $\clD$ is cocomplete, then
	\[\pqDiNat{p}{q}\left(\Delta_{\pt},G\right)\cong\Nat\left(\Lan_{\Delta_{q,p}}\ph,G\right).\]
\end{lemma}
\begin{proof}%
	We have
	\begin{align*}
		\pqDiNat{p}{q}\left(\Delta_{\pt},G\right) & \cong \DiNat\left(\Delta_{p,q}^{*}\Delta_{\pt},\Delta_{q,p}^{*}G\right),                \\
		                                          & \cong \Nat\left(\ph,\left[\Delta_{p,q}^{*}\Delta_{\pt},\Delta_{q,p}^{*}G\right]\right). \\
		                                          & \cong \Nat\left(\ph,\left[\Delta_{\pt}',\Delta_{q,p}^{*}G\right]\right)                 \\
		                                          & \cong \Nat\left(\ph,\Delta_{q,p}^{*}G\right)                                            \\
		                                          & \cong \Nat\left(\Lan_{\Delta_{q,p}}\ph,G\right).\qedhere
	\end{align*}
\end{proof}
\begin{remark}[Computing $\Lan_{\Delta_{q,p}}\ph$]\label{computing-lan-delta-q-p-ph}%
	We have
	\begin{align}
		\Lan_{\Delta_{q,p}}\ph & \cong  \int^{A,B\in\clC}\Hom_{\pq{\clC}[q][p]}\left(\Delta_{q,p}((A,B));(-,-)\right)\odot \ph^{A}_{B}                            \notag        \\
		                       & \cong  \int^{A,B\in\clC}\Hom_{\pq{\clC}[q][p]}\left((\bsA,\bsB);(-,-)\right)\odot \ph^{A}_{B}                                    \notag        \\
		                       & \cong  \int^{A\in\clC}\Hom_{\pq{\clC}[q][p]}\left((\bsA,\bsA);(-,-)\right)                                                       \notag        \\
		                       & \defeq \int^{A\in\clC}\ph^{-_{1}}_{A}\times\cdots\times\ph^{-_{q}}_{A}\times\ph^{A}_{-_{1}}\times\cdots\times\ph^{A}_{-_{p}},\label{eq:ponk-2}
	\end{align}
	meaning the end of
	\begin{equation}\label{eq:ckgpt-defining-end}
		(A,B) \mapsto \ph^{-_{1}}_{B}\times\cdots\times\ph^{-_{q}}_{B}\times\ph^{A}_{-_{1}}\times\cdots\times\ph^{A}_{-_{p}}.
	\end{equation}
\end{remark}
\begin{lemma}\label{ponk}
	There is an isomorphism of functors
	\begin{equation}\label{eq:ponk}
		\Lan_{\Delta_{q,p}}\ph
		\cong
		\hom_{\Pi,p,q}.
	\end{equation}
\end{lemma}
\begin{proof}
	We shall prove the case $(p,q)=(2,1)$, as the general case is analogous. Namely, we claim that
	\begin{align*}
        \int^{X\in\CatFont{C}}\ph^{-_{1}}_{X}\times\ph^{X}_{-_{2}}\times\ph^{X}_{-_{3}} & \cong \Lan_{\Delta_{1,2}}\ph                        \tag{\cref{computing-lan-delta-q-p-ph}}               \\
		                                                                                & \cong                                                       \hom_{\Pi,2,1}                                \\
		                                                                                & \defeq                                                      \ph^{-_{1}}_{-_{2}}\times\ph^{-_{1}}_{-_{3}}.
	\end{align*}
    Fix $A,B,C\in\CatFont{C}_{o}$. We will show that the diagram
	\[
		\begin{tikzcd}[ampersand replacement=\&]
			\displaystyle\coprod_{u\colon X\to Y}\ph^{A}_{X}\times\ph^{Y}_{B}\times\ph^{Y}_{C}
			\arrow[r,"\lambda",shift left=0.8]
			\arrow[r,"\rho"',shift right=0.8]
			\&
			\displaystyle\coprod_{X\in\CatFont{C}}\ph^{A}_{X}\times\ph^{X}_{B}\times\ph^{X}_{C}
			\arrow[r, "\sigma"]
			\&
			\ph^{A}_{B}\times\ph^{A}_{C},
		\end{tikzcd}
	\]
	where $\lambda$ and $\rho$ are induced by the maps given by
	\begin{align}
		\lambda\Big(\var[u] {X} Y;\var[f] {A} X,\var[g] Y {B},\var[h] Y C\Big) & = (f,g\circ u,h\circ u),\notag \\
		\rho   \Big(\var[u] {X} Y;\var[f] {A} X,\var[g] Y {B},\var[h] Y C\Big) & = (u\circ f,g,h),\notag
	\end{align}
	and $\sigma$ is induced by the maps
	\[
		\sigma_{X,X,X}\left(\var[f] {A} X,\var[g] X {B},\var[h] X C\right)
		\defeq
		(g\circ f,h\circ f),
	\]%
	is a coequaliser diagram. Firstly, note that $\sigma$ indeed coequalises $\lambda$ and $\rho$, as
	\begin{align*}
		\sigma(\lambda(u,f,g,h)) & = \sigma(f,g\circ u,h\circ u)           \\
		                         & = ((g\circ u)\circ f,(h\circ u)\circ f) \\
		                         & = (g\circ (u\circ f),h\circ (u\circ f)) \\
		                         & = \sigma(u\circ f,g,h)                  \\
		                         & = \sigma(\rho(u,f,g,h)).
	\end{align*}
	Now, given another morphism
	\[
		\begin{tikzcd}[ ampersand replacement=\&]
			\displaystyle\coprod_{X\in\CatFont{C}}\ph^{A}_{X}\times\ph^{X}_{B}\times\ph^{X}_{C}
			\arrow[r, "\zeta"]
			\&
			E,
		\end{tikzcd}
	\]
	coequalising $(\lambda,\rho)$, i.e.\ such that
	\begin{equation}\label{zeta-coequalises-lambda-and-rho-dim-1}
		\zeta(f,g\circ u,h\circ u)=\zeta(u\circ f,g,h),
	\end{equation}
	we claim that there exists a unique morphism $\widebar{\zeta}\colon\ph^{A}_{B}\times\ph^{A}_{C}\xlongdashrightarrow{\exists!}E$ making the diagram
	\begin{diagram}\label{di:magic-diagram-dim-1}
		\begin{tikzcd}[ampersand replacement=\&]
			\displaystyle\coprod_{u\colon X\to Y}\ph^{A}_{X}\times\ph^{Y}_{B}\times\ph^{Y}_{C}
			\arrow[r,"\lambda",shift left=1]
			\arrow[r,"\rho"',shift right=1]
			\&
			\displaystyle\coprod_{X\in\CatFont{C}}\ph^{A}_{X}\times\ph^{X}_{B}\times\ph^{X}_{C}
			\arrow[r, "\sigma"]
			\arrow[rd, "\zeta"']
			\&
			\ph^{A}_{B}\times\ph^{A}_{C},
			\arrow[d,"\widebar{\zeta}", dashed]
			\\
			\&
			\&
			E\mrp{,}
		\end{tikzcd}
	\end{diagram}%
	commute. Indeed:
	\begin{enumtag}{trp}
		\item\SloganFont{Existence}For each pair $\left(\var[f]{A}B,\var[g]{A}C\right)$ in $\ph^{A}_{B}\times\ph^{A}_{C}$, we define%
		\footnote{%
			Note that in writing $\zeta(\id_{A},f,g)$ we are applying $\zeta$ to the triple $(\id_{A},f,g)$ in the component $\ph^{A}_{A}\times\ph^{A}_{B}\times\ph^{A}_{C}$ of $\displaystyle\coprod_{X\in\CatFont{C}}\ph^{A}_{X}\times\ph^{X}_{B}\times\ph^{X}_{C}$.
		}%
		\[\widebar{\zeta}(f,g)\defeq\zeta(\id_{A},f,g).\]
		\item\SloganFont{Uniqueness}Let $\widetilde{\zeta}\colon\ph^{A}_{B}\times\ph^{A}_{C}\to  E$ be another morphism making \cref{di:magic-diagram-dim-1} commute. Then, for each pair $\left(\var[f]{A}B,\var[g]{A}C\right)$ in $\ph^{A}_{B}\times\ph^{A}_{C}$, we have
		\begin{align*}
			\widetilde{\zeta}(f,g) & =      \widetilde{\zeta}(f\circ\id_{A},g\circ\id_{A})                                                            \\
			                       & \defeq \widetilde{\zeta}(\sigma(\id_{A},f,g))                                                                    \\
                                   & =      \zeta(\id_{A},f,g)                             \tag{by the commutativity of \cref{di:magic-diagram-dim-1}}\\
			                       & \defeq \widebar{\zeta}(f,g).
		\end{align*}
		\item\SloganFont{\cref{di:magic-diagram-dim-1} commutes}For each triple $\left(\var[f]{A}X,\var[g]{X}B,\var[h]{X}C\right)$ in $\smash{\displaystyle\coprod_{X\in\CatFont{C}}\ph^{A}_{X}\times\ph^{X}_{B}\times\ph^{X}_{C}}$, we have
		\begin{align*}
			\widebar{\zeta}(\sigma(f,g,h))                        & =      \widebar{\zeta}(g\circ f,h\circ f)                                                   \\
			                                                      & \defeq \zeta(\id_{A},g\circ f,h\circ f)                                                     \\
                                                                  & =      \zeta(\id_{A}\circ f,g,h)          \tag{\cref{zeta-coequalises-lambda-and-rho-dim-1}}\\
			                                                      & =      \zeta(f,g,h).\qedhere
		\end{align*}
	\end{enumtag}
\end{proof}
Taken all together, \cref{pink}, \cref{pank}, and \cref{ponk} yield \cref{dis}.
\begin{corollary}[$(p,q)$\hyp{}co/ends as weighted co/limits]\label{pq_are_weighted}
	Let $G : \pq{\clC} \to \Set$ be a functor. We have functorial isomorphisms%
	\footnote{%
		For $(p,q)=(1,1)$, this amounts to the well-known statement that the co/end of $T : \oo{\clC}$ is the weighted co/limit of $T$ by the hom functor $\Hom_{\CatFont{C}}(-,-):\oo{\clC} \to \Set$; see \cite[Section 3.10]{kelly}.%
	}%
	\[
		\pqEnd{p}{q}{\bsA\in\clC}G^{\bsA}_{\bsA}   \cong \wlim{\hom_{\Pi,p,q}} G,   \qquad
		\pqCoend{p}{q}{\bsA\in\clC}G^{\bsA}_{\bsA} \cong \wcolim{\hom_{\Pi,p,q}} G,
	\]
\end{corollary}
\begin{proof}
	We have isomorphisms, natural in $A\in\clD$
	\begin{align*}
		\clD\left(A,\pqEnd{p}{q}{A}G^{\bsA}_{\bsA}\right) &\defeq \pqDiNat{p}{q}\left({\one},\clD(A,G)\right)                 \\
                                                          &\cong  \Nat\left(\hom_{\Pi,p,q},\clD(A,G)\right)   \tag{\cref{dis}}\\
		                                                  &\defeq \clD(A,\wlim{\hom_{\Pi,p,q}} G).
	\end{align*}
	The result then follows from the Yoneda lemma. A dual argument yields the second identity.
\end{proof}
From \cref{pq_are_weighted}, a general fact about weighted limits (\cite[Lemma 4.3.1]{cofriend}) yields
\begin{corollary}\label{adjadjadj}
	There is an adjunction
	\[
		\begin{tikzcd}[column sep=2cm]
			\Cat(\pq{\clC},\clD) \ar[r,shift right=.5em, "\pqEnd{p}{q}{\clC}"'] \ar[r,phantom, "\perp"]&\ar[l,shift right=.5em, "\hom_{\Pi,p,q}\odot -"'] \clD\mrp{,}
		\end{tikzcd}
	\]
	where the left adjoint $\hom_{\Pi,p,q}\odot -$ is defined by $D\mapsto \big((\uA,\uB)\mapsto \hom_{\Pi,p,q}(\uA,\uB)\odot D\big)$.

	Dually, there is an adjunction
	\[
		\begin{tikzcd}[column sep=2cm]
			\Cat(\pq{\clC},\clD) \ar[r,shift left=.5em, "\pqCoend{p}{q}{\clC}"] \ar[r,phantom, "\perp"]&\ar[l,shift left=.5em, "\hom_{\Pi,p,q}\pitchfork -"] \clD\mrp{,}
		\end{tikzcd}
	\]
	where the left adjoint $\hom_{\Pi,p,q}\odot -$ is defined by $D\mapsto \big((\uA,\uB)\mapsto \hom_{\Pi,p,q}(\uA,\uB)\odot D\big)$, and the right adjoint $\hom_{\Pi,p,q}\pitchfork -$ is defined by $D\mapsto \big((\uA,\uB)\mapsto \hom_{\Pi,p,q}(\uA,\uB)\pitchfork D\big)$.
\end{corollary}
\begin{lemma}[Shishido identity, first form]\label{shishido-identity-1}
    The product-hom functor of \cref{prod_hom} satisfies the \emph{Shishido identity}:\footnote{Shishido Baiken is the name of a Japanese swordsman (his existence is attested in the \emph{Nitenki} written in 1776, but the reliability of the text is currently object of debate). Baiken was a skilled master of \emph{kusarigama-jutsu} and, according to the legend, lost a duel (and his life) with Miyamoto Musashi.}
	\[
		\hom_{\Pi,p,q}\times\hom_{\Pi,r,s}
		\cong
		\hom_{\Pi,r,s}\times\hom_{\Pi,p,q}
		\cong
		\prod_{i=1}^{p+r}\prod_{j=1}^{q+s}\ph^{(-_{i},-_{i})}_{(-_{j},-_{j})}.
	\]
\end{lemma}
\begin{proof}
	The first isomorphism is clear. For the second one, let $(\uA,\uB)\in\pq{\clC}$ and $\uA',\uB'\in\pq{\clC}[r][s]$; we then have natural isomorphisms
	\begin{align*}
		\hom_{\Pi,p,q}\times \hom_{\Pi,r,s} ((\uA;\uA'), (\uB;\uB')) & =     \prod_{i,j=1}^{p,q}\Hom_{\CatFont{C}}\left(A_i,B_j\right) \times \prod_{h,k=1}^{r,s}\Hom_{\CatFont{C}}\left(A'_h, B'_k\right) \\
		                                                             & \cong \prod_{i,j=1}^{p,q}\prod_{h,s}^{r,s}\Hom_{\CatFont{C}}\left(A_i,B_j\right) \times \Hom_{\CatFont{C}}\left(A'_h, B'_k\right)   \\
		                                                             & \cong \prod_{i,j=1}^{p+r,q+s}\Hom_{\CatFont{C}}\left((\uA; \uA')_i,(\uB; \uB')_j\right).
	\end{align*}
	where the tuple $\uA;\uA'$ is the juxtaposition of $\uA$ and $\uA'$.
\end{proof}
\begin{theorem}[The Fubini Rule]\label{fubini-for-p-q-co-ends}
	Let $D :\clA^{(p,q)}\times\clB^{(r,s)}\to \clD$ be a functor. Then
	\begin{gather}
		\pqEnd{p+r}{q+s}{(A,B)}D^{\bsA,\bsB}_{\bsA,\bsB}
		\cong
		\pqEnd{p}{q}{A}\pqEnd{r}{s}{B}D^{\bsA,\bsB}_{\bsA,\bsB}
		\cong
		\pqEnd{r}{s}{B}\pqEnd{p}{q}{A}D^{\bsA,\bsB}_{\bsA,\bsB},\label{fubini-expressions}\\
		\pqCoend{p+r}{q+s}{(A,B)}D^{\bsA,\bsB}_{\bsA,\bsB}
		\cong
		\pqCoend{p}{q}{A}\pqCoend{r}{s}{B}D^{\bsA,\bsB}_{\bsA,\bsB}
		\cong
		\pqCoend{r}{s}{B}\pqCoend{p}{q}{A}D^{\bsA,\bsB}_{\bsA,\bsB}
	\end{gather}
	as objects of $\clD$, meaning that any of these expressions exist if and only if the others do, and, if so, they are are all canonically isomorphic.%
\end{theorem}
\begin{proof}
	To prove that three expressions in \cref{fubini-expressions} are isomorphic, it suffices to show that their adjoints (\cref{adjadjadj})
	\begin{align*}
		\left(\hom_{\Pi,p,q}\pitchfork\hom_{\Pi,r,s}\right)                                & \pitchfork (-) \\
		\left(\hom_{\Pi,r,s}\pitchfork\hom_{\Pi,p,q}\right)                                & \pitchfork (-) \\
		\left(\prod_{i=1}^{p+r}\prod_{j=1}^{q+s}\ph^{(-_{i},-_{i})}_{(-_{j},-_{j})}\right) & \pitchfork (-)
	\end{align*}
	are isomorphic, since adjoints are unique. As $(A\pitchfork B)\pitchfork C\cong(A\times B)\pitchfork C$, this follows from \cref{shishido-identity-1}.

	A suitably dualised argument yields the result for higher arity coends.
\end{proof}
\begin{remark}[Fubini does not reduce arity]%
	Note that $p,q,r,s$ can't be broken further: given a functor $G$ of type $\typepq{p}{q}$, its $(p,q)$\hyp{}end isn't in general expressible in terms of $(p-r,q-s)$-ends for suitable $r,s\ge 1$. This confirms the fact that iterated ends \emph{are not} higher arity ends. Instead, higher arity ends are particular ends.

	That is, \cref{fubini-for-p-q-co-ends} does not allow us to reduce the arity of a higher arity co/end when $\CatFont{A}=\CatFont{B}$:
	\[
		\pqEnd{p}{q}{A}\pqEnd{r}{s}{B}D^{\bsA,\bsB}_{\bsA,\bsB}
		\cong
		\pqEnd{p+r}{q+s}{(A,B)\in\CatFont{A}\times\CatFont{A}}D^{\bsA,\bsB}_{\bsA,\bsB}
		\ncong
		\pqEnd{p+r}{q+s}{A\in\CatFont{C}}D^{\bsA}_{\bsA}.
	\]
	This is already apparent from the classical Fubini rule, where, given a functor $T\colon\CatFont{C}^{\op}\times\CatFont{C}\times\CatFont{E}^{\op}\times\CatFont{E}\to \CatFont{D}$ with $\CatFont{C}=\CatFont{E}$, we have once again
	\[
		\int_{(A,B)\in\CatFont{C}\times\CatFont{C}}T((A,B),(A,B))
		\ncong
		\int_{A\in\CatFont{C}}T(A,A,A,A).
	\]
	The main point in both cases is that we are \say{integrating} over a pair $(A,B)$, and not over a single variable $A$.

	From the point of view of adjoints, we have in (e.g.) the $(p,q)=(1,1)$ case
	\begin{align*}
		(-)\odot\left(\ph^{-_{1}}_{-_{3}}\times\ph^{-_{1}}_{-_{4}}\times\ph^{-_{2}}_{-_{3}}\times\ph^{-_{2}}_{-_{4}}\right) & \dashv \pqCoend{2}{2}{A\in\CatFont{C}}D^{A,A}_{A,A}                  \\
		(-)\odot\underbrace{\ph^{(-_{1},-_{2})}_{(-_{3},-_{4})}}_{\ph^{-_{1}}_{-_{3}}\times\ph^{-_{2}}_{-_{4}}}             & \dashv \int^{(A,B)\in\CatFont{C}\times\CatFont{C}}D^{(A,B)}_{(A,B)},
	\end{align*}
	and of course
	\[
		\ph^{-_{1}}_{-_{3}}\times\ph^{-_{1}}_{-_{4}}\times\ph^{-_{2}}_{-_{3}}\times\ph^{-_{2}}_{-_{4}}
		\neq
		\ph^{(-_{1},-_{2})}_{(-_{3},-_{4})}
		=
		\ph^{-_{1}}_{-_{3}}\times\ph^{-_{2}}_{-_{4}},
	\]
	so $\int^{A\in\CatFont{C}}D^{A,A}_{A,A}$ and $\int^{(A,B)\in\CatFont{C}\times\CatFont{C}}D^{(A,B)}_{(A,B)}$ are different as well.
\end{remark}

\section{Examples: a session of callisthenics}\label{sec:examples-of-higher-arity-ends}
\subsection{Examples arranged by dimension}
The first roundup of examples is a series of sanity checks:
\begin{example}[$(0,0)$-co/ends]
	For the case where $p=q=0$, we look at functors of the form $D\colon\catpt\to \clD$, where $\catpt$ is the terminal category. It is evident that such a functor corresponds precisely to an object of $\clD$, a $(0,0)$-wedge corresponds to the identity on that object, and the $(0,0)$-end of $D$ is precisely that object.
\end{example}
\begin{example}[$(1,0)$- and $(0,1)$-co/ends]
	When $(p,q)=(1,0)$, we consider functors of the form $D\colon\clC^{\op}\to \clD$, and we see from the universal property of $(p,q)$\hyp{}ends that the $(1,0)$-end of $D$ is the limit of $D$. Similarly, the $(0,1)$-end of a functor $D\colon\clC\to \clD$ is again the limit of $D$.

	In particular, starting with a functor $D\colon\clC\to \clD$ and passing to the opposite functor $D^{\op}\colon\clC^{\op}\to \clD^{\op}$, we get isomorphisms
	\begin{align*}
		\pqEnd{0}{1}{A\in\clC}D       & = \lim(D),   \\
		\pqEnd{1}{0}{A\in\clC}D^{\op} & = \colim(D).
	\end{align*}
\end{example}
\begin{example}[$(1,1)$-co/ends]
	Let $(p,q)=(1,1)$ and consider a diagram in $\clD$ of the form $D\colon\clC^{\op}\times\clC\to \clD$. Again from the universal property of $(p,q)$\hyp{}ends, we see that $(1,1)$-ends are nothing but ordinary ends. That is:
	\[\pqEnd{1}{1}{A\in\clC}D^{A}_{A}=\int_{A\in\clC}D^{A}_{A}.\]
\end{example}
Furthermore, $(n,0)$-co/ends and $(0,n)$-co/ends are just suitable co/limits:
\begin{example}[$(2,0)$-, $(0,2)$-co/ends; $(n,0)$- and $(0,n)$-co/ends]
	Given a diagram $D\colon\CatFont{C}^{2}\to \CatFont{D}$, we have
	\[
		\pqEnd{0}{2}{A\in\clC}D^{A,A}   = \lim\left(D\circ\Delta_{\clC}\right),
		\qquad
		\pqCoend{0}{2}{A\in\clC}D^{A,A} = \colim\left(D\circ\Delta_{\clC}\right),
	\]
	where $\Delta_{\clC}\colon\clC\to \clC\times\clC$ is the diagonal functor of $\clC$ in the Cartesian monoidal structure of $\Cats$.

	A similar argument yields, for a diagram $D\colon\CatFont{C}^{n}\to \CatFont{D}$:
	\[
		\pqEnd{0}{n}{A\in\clC}D^{\bsA}   = \lim\left(D\circ\Delta^{n}_{\clC}\right),
		\qquad
		\pqCoend{0}{n}{A\in\clC}D^{\bsA} = \colim\left(D\circ\Delta^{n}_{\clC}\right).
	\]%
\end{example}
We consider next the first nontrivial example:
\begin{example}[$(2,1)$- and $(1,2)$-co/ends]
	Given a functor $T : \clC^{-2}\times \clC\to \clD$ let's flesh out what a $(2,1)$-wedge is: it consists of an object $X$ endowed with maps
	\[
		\begin{tikzcd}
			\omega_A : X \ar[r] & T(AA;A)
		\end{tikzcd}
	\]
	with the property that, for every $f : A \to B$ in $\clC$, the square
	\[
		\begin{tikzcd}
			X \ar[r, "\omega"]\ar[d,"\omega"']& T(AA;A) \ar[d, "T(11f)"]\\
			T(BB;B)\ar[r, "T(ff1)"'] & T(AA;B)
		\end{tikzcd}
	\]
	The $(2,1)$-end of $T$ is the terminal object in the category $\catWd_{(2,1)}(T)$ of wedges for $T$.

	As a particular example, let $\clC$ be a Cartesian category. Let us consider the functor
	\[\begin{tikzcd}[row sep=0cm]
			T = \Hom(\__1 \times \__2, \__3) : \clC^\op\times \clC^\op \times \clC \ar[r] & \Set \\
			(A,B;C) \ar[r,mapsto]& \clC(A\times B,C)
		\end{tikzcd}\]
	What is a $(2,1)$-wedge for $T$? It consists of a set $X$, and a family of functions $\omega_A : X \to \clC(A\times A,A)$ with the property that for each $f : A \to B$, the square
	\[
		\begin{tikzcd}
			X \ar[r]\ar[d]& \clC(A\times A,A) \ar[d, "{f_*}"]\\
			\clC(B\times B,B) \ar[r, "{(f\times f)^*}"']& \clC(A\times A,B)
		\end{tikzcd}
	\]
	commutes. In other words, each $\omega_A(x)$ is a morphism $A\times A\to A$ in $\clC$ with the property that each $f : A \to B$ is a ``homomorphism'' with respect to $\omega_A(x), \omega_B(x)$:
	\[
		\begin{tikzcd}
			A \times A\ar[d, "f\times f"']\ar[r, "\omega_A(x)"] & A \ar[d, "f"]\\
			B \times B\ar[r, "\omega_B(x)"'] & B
		\end{tikzcd}
	\]
	This structure is easy to determine: let $b : 1 \to B$ be a point of $A$ (e.g., let $\clC=\Set$). Then the commutativity of
	\[
		\begin{tikzcd}
			1\ar[d, "f\times f"']\ar[r, "\omega_A(x)"] & 1 \ar[d, "f"]\\
			B \times B\ar[r, "\omega_B(x)"'] & B
		\end{tikzcd}
	\]
	tells that $\omega_B(x) : B \times B \to B$ is a section of the diagonal $\Delta_B$ (this means: $\omega_B(x)(b,b)=b$ for every $b\in B$). Moreover, the family $\omega_A : A \times A \to A$ is natural in $A$, i.e.\ it is a natural transformation
	\[ \tk{ \times \circ \Delta \ar[r, "\omega"] \& \id } \]
	that is a section of the natural transformation in the opposite direction, unit of the adjunction constant-product.

	There are few such transformations. First, observe that the functor $\times \circ \Delta$ coincides with the functor $X \mapsto X^2$, so corresponds to the cotensoring with $2=\{0,1\}$ in an abstract category, and it is just the corepresentable presheaf on $2$ in the category of sets. Similarly, the identity is the corepresentable over the point in the category of sets. All in all, in the category of sets
	\begin{align*}
		\pqEnd{2}{1}{A}\Hom(A\times A,A) & \cong [\Set,\Set](\Set(2,\_), \Set(1,\_)) \\
		                                 & \cong \Set(1,2) \cong 2
	\end{align*}
	by the Yoneda lemma.

	Similarly, in a category $\clC$ with $\Set$-cotensors,
	\[\pqEnd{2}{1}{A}\clC(A\times A,A)\cong [\clC,\Set]((2\pitchfork\_), (1\pitchfork\_))\]
	where the functor $(n\pitchfork \_)$ coincides with the $n$-fold iterated product $X \mapsto X^n =X\times\dots\times X$. A similar argument shows that $\pqEnd{n}{1}{A}\clC(A^n,A) \cong [\clC,\Set]((n\pitchfork\_), (1\pitchfork\_))$.
\end{example}
\begin{example}[Dinatural transformations via $(2,2)$-ends]\label{dinatural-transformations-as-a-2-2-end}
	This example was first discovered by Street and Dubuc in \cite[Theorem 1]{dinatural-transformations}. We give an account of it in our language.

	Let $F,G\colon\clC^{\op}\times\clC\to \clD$ be functors. Then
	\begin{align*}
		\DiNat(F,G) & \cong \int_{A\in\clC}\hom_{\clD}\left(F^{A}_{A},G^{A}_{A}\right),        \\
		            & \cong \pqEnd{2}{2}{A\in\clC}\hom_{\clD}\left(F^{A}_{A},G^{A}_{A}\right).
	\end{align*}
\end{example}
\begin{proof}%
	The proof of the first isomorphism is divided in two steps:
	\begin{enumtag}{dep}
		\item First, consider the functor
		\[\hom_{\clD}\left(F^{-_{2}}_{-_{1}},G^{-_{1}}_{-_{2}}\right)\colon\clC^{\op}\times\clC\to \Sets\]
		sending
		\begin{enumerate}
            \item An object $\left(A,X\right)$ of $\clC^{\op}\times\clC$ to the set $\Hom_{\clD}\left(F^{X}_{A},G^{A}_{X}\right)$;
			\item A morphism $\Big(A\xrightarrow{f}B,X\xrightarrow{g}Y\Big)$ of $\clC^{\op}\times\clC$ to the map
			      \[\clD\left(F^{g}_{f},G^{f}_{g}\right)\colon\clD\left(F^{X}_{A},G^{A}_{X}\right)\to \clD\left(F^{Y}_{B},G^{B}_{Y}\right)\]
			      defined as the composition
			      \[
				      \begin{tikzcd}[row sep={9.0em,between origins}, column sep={9.0em,between origins},  ampersand replacement=\&]
					      {\clD\big(F^{X}_{A},G^{A}_{X}\big)}
					      \arrow[r, "\scriptscriptstyle{\clD\big(F^{X}_{A},G^{A}_{f}\big)}"']
					      \arrow[rrrr, to path={%
								      |- ([yshift=2.0ex]\tikztotarget.north)node[near end,above]{$\scriptstyle\clD\big(F^{g}_{f},G^{f}_{g}\big)$}
								      -- (\tikztotarget)}]
					      \&
					      {\clD\big(F^{X}_{A},G^{A}_{Y}\big)}
					      \arrow[r, "\scriptscriptstyle{\clD\big(F^{f}_{A},G^{A}_{Y}\big)}"']
					      \&
					      {\clD\big(F^{Y}_{A},G^{A}_{Y}\big)}
					      \arrow[r, "\scriptscriptstyle{\clD\big(F^{Y}_{A},G^{f}_{Y}\big)}"']
					      \&
					      {\clD\big(F^{Y}_{A},G^{B}_{Y}\big)}
					      \arrow[r, "\scriptscriptstyle{\clD\big(F^{Y}_{f},G^{B}_{Y}\big)}"']
					      \&
					      {\clD\big(F^{Y}_{B},G^{B}_{Y}\big)}\mathrlap{.}
				      \end{tikzcd}
			      \]
		\end{enumerate}
		By functoriality of $\Hom$s, the assignment $(A,X)\mapsto\clD\big(F^{X}_{A},G^{A}_{X}\big)$ preserves identities and composition, defining therefore a functor.
		\item Second, we compute the end $\int_{A\in\clC}\hom_{\clD}\left(F^{A}_{A},G^{A}_{A}\right)$; this is given by the equaliser of the pair of maps
		\[
			\begin{tikzcd}[row sep={13.5em,between origins}, column sep={13.5em,between origins},  ampersand replacement=\&]
				\prod_{A\in\clC_o}\clD\big(F^{A}_{A},G^{A}_{A}\big)
				\arrow[r,"\lambda",shift left=1.0]
				\arrow[r,"\rho"', shift right=1.0]
				\&
				\mspace{5.0mu}
				\prod_{f :A\to B}
				\Set\big(\clC(A,B),\clD\big(F^{B}_{A},G^{A}_{B}\big)\big)
			\end{tikzcd}%
		\]
		where $\lambda$ and $\rho$ are the morphisms induced by the universal property of the product by the morphisms
		\begin{align*}
			\lambda_{A,B} & \colon \textstyle \prod_A \clD\big(F^{A}_{A},G^{A}_{A}\big)\to \Set\big( \big(\clC\big(A,B\big),\clD\big(F^{B}_{A},G^{A}_{B}\big)\big), \\
			\rho_{A,B}    & \colon \textstyle \prod_A \clD\big(F^{A}_{A},G^{A}_{A}\big)\to \Set\big( \big(\clC\big(A,B\big),\clD\big(F^{B}_{A},G^{A}_{B}\big)\big)
		\end{align*}
		acting on elements as
		\begin{align*}
			\left(\alpha_{A}\colon F^{A}_{A}\to  G^{A}_{A}\right) & \mapsto \left(\var[f]{A}{B}\mapsto\left(G^{\id_{A}}_{f}\circ\alpha_{A}\circ F^{f}_{\id_{A}}\right)\right), \\
			\left(\alpha_{B}\colon F^{B}_{B}\to  G^{B}_{B}\right) & \mapsto \left(\var[f]{A}{B}\mapsto\left(G^{f}_{\id_{B}}\circ\alpha_{B}\circ F^{\id_{B}}_{f}\right)\right),
		\end{align*}
		and hence asking for $\lambda$ and $\rho$ to be equal is precisely the dinaturality condition for a family
		\[\big\{\alpha_{A}\colon F^{A}_{A}\to  G^{A}_{A}\big\}_{A\in\clC_o}\]
		of morphisms of $\clD$. As an element of the end $\int_{A\in\clC}\clD\left(F^{A}_{A},G^{A}_{A}\right)$ is precisely such a family equalising $\lambda$ and $\rho$, the result follows.
	\end{enumtag}
	As for the second isomorphism, we define a functor
	\[\clD(F_\uparrow,G_\downarrow)\colon\pq{\clC}[2][2]\to \Sets\]
	in a similar manner as we did above and then invoke \cref{p-q-ends-as-ordinary-ends} of \cref{prop:properties-of-p-q-ends}. The universal property of the $(2,2)$-end of $\clD(F_\uparrow,G_\downarrow)$ is the same as the universal property of the equaliser defining $\DiNat(F,G)$.
\end{proof}
Generalising \cref{dinatural-transformations-as-a-2-2-end}, we have the following.
\begin{example}[$(p,q)$\hyp{}Dinatural transformations via $(q,p)$-ends]\label{p-q-as-dinaturals}%
	Let $F$ and $G$ be functors of type $\typepq{p}{q}$ and $\typepq{q}{p}$, respectively. Then
	\[
		\pqDiNat{p}{q}(F,G) \cong \pqEnd{q}{p}{A\in\clC}\Hom_{\clD}\left(F^{\bsA}_{\bsA},G^{\bsA}_{\bsA}\right),
	\]
	where the \say{integrand} is the functor
	\begin{diagram*}
		\begin{tikzcd}[row sep=0.0em, column sep=2.7em,  ampersand replacement=\&]
			\pq{\clC}[q][p]
			\arrow[r]
			\&
			\clD
			\\
			(\uA,\uB)
			\arrow[r, mapsto]
			\&
			\Hom_{\clD}\left(F^{B_{1},\ldots,B_{p}}_{A_{1},\ldots,A_{q}},G^{A_{1},\ldots,A_{q}}_{B_{1},\ldots,B_{p}}\right)\mrp{.}
		\end{tikzcd}
	\end{diagram*}%
\end{example}
\begin{proof}%
	This is a combination of \cref{higher-arity-dinaturality-via-ordinary-dinaturality}, \cref{dinatural-transformations-as-a-2-2-end}, and \cref{p-q-ends-as-ordinary-ends} of \cref{prop:properties-of-p-q-ends}.
\end{proof}
\subsection{Classes of higher arity coends}
\subsubsection{A glance at weighted co/ends}\label{glance_at_weightends}
We now introduce a natural factory of examples for higher arity co/ends. In a nutshell, weighted co/ends are to co/ends as weighted co/limits are to co/limits.
\begin{definition}[Weighted co/ends]\label{def:weighted-co-ends}
    Let $\clC$ and $\clD$ be $\clV$-enriched categories and $D\colon\clC^{\op}\otimesV\clC\to\clD$ a $\clV$-functor, and $W : \CatFont{C}^{\op}\otimesV\CatFont{C} \to \clV$ a $\clV$-presheaf.
	\begin{enumtag}{we}
		\item The \emph{end of $D$ weighted by $W$} is, if it exists, the object $\smash{\WeightedEnd{A\in\clC}{W}D^{A}_{A}}$ of $\clD$ with the property that
		\[\hom_{\CatFont{D}}\Big(-,{\WeightedEnd{A\in\clC}{W}D^{A}_{A}}\Big)\cong\VDiNat{\clV}(W,\eHom_{\clC}(-,D)).\]
		\item The \emph{coend of $D$ weighted by $W$} is, if it exists, the object $\smash{\WeightedCoend{A\in\clC}{W}D^{A}_{A}}$ of $\clD$ with the property that
		\[\hom_{\CatFont{D}}\Big({\WeightedCoend{A\in\clC}{W}D^{A}_{A}},-\Big)\cong\VDiNat{\clV}(W,\eHom_{\clC}(D,-)).\]
	\end{enumtag}
\end{definition}
\begin{example}[Weighted co/ends are $(2,2)$-co/ends]
    For $\CatFont{V}$ a co/tensored monoidal category, there are $(2,2)$-co/end formulas for weighted co/ends:
	\begin{align*}
		\wEnd[W]{A\in\clC}D^{A}_{A}   &\cong \pqEnd{2}{2}{A\in\clC}W^{A}_{A}\pitchfork D^{A}_{A}, \\
		\wCoend[W]{A\in\clC}D^{A}_{A} &\cong \pqCoend{2}{2}{A\in\clC}W^{A}_{A}\odot D^{A}_{A}.
	\end{align*}
\end{example}
\begin{example}[Weights increase arity]
	Let $F,G\colon\clC\to \CatFont{D}$ and $W\colon\clC^{\op}\otimesV\clC\to \CatFont{V}$ be $\CatFont{V}$-functors. In analogy with
	\[
		\VNat{\CatFont{V}}(F,G)
		\defeq
		\int_{A\in\clC}\eHom_{\clD}(F_{A},G_{A}),
	\]
	we define the \emph{object $\wNat{W}(F,G)$ of natural transformations from $F$ to $G$ weighted by $W$} by
	\begin{equation}\label{object-wnat-w-f-g-of-natural-transformations-from-f-to-g-weighted-by-w}
		\wNat{W}(F,G)
		\defeq
		\wEnd[W]{A\in\clC}\eHom_{\clD}(F_{A},G_{A}).
	\end{equation}
	Taking $W$ to be mute in its contravariant variable, we can give a reformulation of the universal property of weighted limits:
	\[
		\ph\left(-,\wlim{W}(D)\right)
		\cong
		\wNat{W}\left(\Delta_{(-)},D\right).
	\]
	Defining $\wVDiNat{W}{\CatFont{V}}(F,G)$ by a similar formula, we also obtain the following isomorphism in the case of weighted ends:
	\[
		\ph\left(-,\wEnd[W]{A\in\clC}D^{A}_{A}\right)
		\cong
		\wVDiNat{W}{\CatFont{V}}\left(\Delta_{(-)},D\right).
	\]
	This naturally suggests a definition of \say{doubly-weighted ends}:
	\[
		\ph\left(-,\wEnd[W_{1},W_{2}]{A\in\clC}D^{A}_{A}\right)
		\cong
		\wVDiNat{W_{1}}{\CatFont{V}}(W_{2},D).
	\]
    Proceeding inductively leads to the notion of an end weighted by a collection of functors $\{W_{1},\ldots,W_{n}\}$. These ``$n$-weighted ends'' however, can actually be computed as $(n+1,n+1)$-ends:%
	\[
		\int_{A\in\clC}^{[W_{1},\ldots,W_{n}]}D^{A}_{A}
		\cong
		\pqEnd{n+1}{n+1}{A\in\clC}\left((W_{1})^{A}_{A}\times\cdots\times(W_{n})^{A}_{A}\right)\odot D^{A}_{A}.
	\]
	As such, we see that weighting an end increases its arity by $(1,1)$.
\end{example}
\subsubsection{Weighted Kan extensions}\label{weikan}
Another source of examples comes from ``weighing'' left and right Kan extensions. While the most general such weight is a profunctor, having type $\typepq{1}{1}$, weights of type $\typepq{1}{0}$ or $\typepq{0}{1}$ are specially interesting, as they give a more direct parallel to the classical theory of weighted co/limits (see \cref{ex:weighted-co-limits-weighted-kan-extensions}).

For \cref{left-weikan,right-weikan} below, recall from \cref{object-wnat-w-f-g-of-natural-transformations-from-f-to-g-weighted-by-w} the definition of the object $\wNat{W}(F,G)$ of weighted natural transformations.
\begin{definition}\label{left-weikan}
	The \emph{left Kan extension of $F$ along $K$ weighted by $W$} is, if it exists, the $\clV$-functor
	\[
		\left(\wLan{W}_{K}F\colon\clD\to \CatFont{E}\right)\colon
		\begin{tikzcd}[row sep={3.6em,between origins}, column sep={3.6em,between origins}, ampersand replacement=\&]
			{}
			\&
			\clD
			\arrow[d, "\wLan{W}_{K}F", dashed]
			\\
			\clC
			\arrow[ru, "K"]
			\arrow[l, "W",mid vert,loop left]
			\arrow[r, "F"'{name=F}]
			\&
			\CatFont{E}\mrp{,}
			\arrow[from=F,to=1-2,shorten=1.125em,Rightarrow,xshift=-0.125em,yshift=-0.25em]
		\end{tikzcd}
	\]
	for which we have a $\CatFont{V}$-natural isomorphism
	\begin{equation}\label{weighted-left-kan-extension-defining-equation}
		\VNat{\clV}\left(\wLan{W}_{K}F,G\right)\cong\wVNat{W}{\clV}\left(F,G\circ K\right),
	\end{equation}
	natural in $G$.
\end{definition}
One defines weighted right Kan extensions in a dual manner:
\begin{definition}\label{right-weikan}
	The \emph{right Kan extension of $F$ along $K$ weighted by $W$} is, if it exists, the $\clV$-functor
	\[
		\left(\wRan{W}_{K}F\colon\clD\to \CatFont{E}\right)\colon
		\begin{tikzcd}[row sep={3.6em,between origins}, column sep={3.6em,between origins}, ampersand replacement=\&]
			{}
			\&
			\clD
			\arrow[d, "\wRan{W}_{K}F", dashed]
			\\
			\clC
			\arrow[ru, "K"]
			\arrow[l, "W",mid vert,loop left]
			\arrow[r, "F"'{name=F}]
			\&
			\CatFont{E}\mrp{,}
			\arrow[from=F,to=1-2,shorten=1.125em,Leftarrow,xshift=-0.125em,yshift=-0.25em]
		\end{tikzcd}
	\]
	for which we have a $\CatFont{V}$-natural isomorphism
	\begin{equation}\label{weighted-right-kan-extension-defining-equation}
		\VNat{\clV}\left(G,\wRan{W}_{K}F\right)\cong\wVNat{W}{\clV}\left(G\circ K,F\right),
	\end{equation}
	natural in $G$.
\end{definition}
\begin{example}[Weighted co/limits as weighted Kan extensions]\label{ex:weighted-co-limits-weighted-kan-extensions}
	Let $D\colon\clC\to \CatFont{D}$ be a diagram on a category $\CatFont{D}$. Then we may canonically identify the left Kan extension of $D$ along the terminal functor with its colimit:
	\[
		\Lan_{!}D\cong\ceil{\colim(D)}
		\quad
		\begin{tikzcd}[row sep={3.6em,between origins}, column sep={3.6em,between origins}, ampersand replacement=\&]
			{}
			\&
			\catpt
			\arrow[d, "\ceil{\colim(D)}", dashed]
			\\
			\clC
			\arrow[ru, "!"]
			\arrow[r, "D"'{name=F}]
			\&
			\CatFont{D}\mrp{.}
			\arrow[from=F,to=1-2,shorten=1.0em,Rightarrow,xshift=-0.125em,yshift=-0.25em]
		\end{tikzcd}
	\]
	Similarly, given a weight $W\colon\clC^{\op}\to \Sets$, we have
	\[
		\wLan{W}_{!}D\cong\ceil{\wcolim{W}(D)}
		\quad
		\begin{tikzcd}[row sep={3.6em,between origins}, column sep={3.6em,between origins}, ampersand replacement=\&]
			{}
			\&
			\catpt
			\arrow[d, "\ceil{\wcolim{W}(D)}", dashed]
			\\
			\clC
			\arrow[ru, "!"]
			\arrow[l, "W",mid vert,loop left]
			\arrow[r, "D"'{name=F}]
			\&
			\CatFont{D}\mrp{.}
			\arrow[from=F,to=1-2,shorten=1.0em,Rightarrow,xshift=-0.125em,yshift=-0.25em]
		\end{tikzcd}
	\]
\end{example}
Weighted Kan extensions may also be expressed as $(p,q)$-co/ends:
\begin{align}
	\wLan{W}_{K}F & \cong \wCoend[W]{A\in\clC}\eHom_{\clC}\left(K_{A},-\right)     \odot     F_{A} \cong \pqCoend{2}{2}{A\in\clC}\left(W^{A}_{A}\times\eHom_{\clC}(K_{A},-)\right)\odot      F_{A}, \label{eq:weighted-kan-extension-formula-left} \\
	\wRan{W}_{K}F & \cong \wEnd[W]{A\in\clC}  \eHom_{\clC}\left(-,K_{A}\right)\pitchfork F_{A} \cong \pqEnd{2}{2}{A\in\clC}  \left(W^{A}_{A}\times\eHom_{\clC}(-,K_{A})\right)\pitchfork F_{A}.\label{eq:weighted-kan-extension-formula-right}
\end{align}
Equipped with this description, we now proceed to compute a few weighted Kan extensions.
\begin{example}
	Consider the functor $\inj^\op : \pt^\op \to \Delta^\op$. The left and right Kan extensions of a set $X_{\bullet}\colon\catpt\to \Sets$ along $\inj^{\op}$ are given by
	\begin{align*}
		\Lan_{\inj^{\op}}(X) & \cong \underline{X}_{\bullet} \\
		\Ran_{\inj^{\op}}(X) & \cong \Cech(X).
	\end{align*}
    Passing from ordinary Kan extensions to weighted ones and hence picking a weight $W\colon\catpt^{\op}\times\catpt\to\Sets$ (whose image in $\Sets$ we also denote by $W$), we obtain
	\[
        \begin{aligned}
            \wLan{W}_{\inj^{\op}}(X) &\cong \underline{W\times X}_{\bullet}, \\
            \wRan{W}_{\inj^{\op}}(X) &\cong \Cech(W\times X),
        \end{aligned}
        \qquad
		\begin{tikzcd}[row sep={4.5em,between origins}, column sep={4.5em,between origins}, ampersand replacement=\&]
			{}
			\&
			\SimplexCategory^\op
            \arrow[d, "\wLan{W}_{\inj^{\op}}(X)", dashed]
			\\
			\catpt^{\op}
			\arrow[ru, "\inj^{\op}"]
			\arrow[l, "W",mid vert,loop left]
			\arrow[r, "X"'{name=F}]
			\&
			\Sets\mrp{.}
			\arrow[from=F,to=1-2,shorten=1.125em,Rightarrow,xshift=-0.125em,yshift=-0.25em]
		\end{tikzcd}
	\]
\end{example}
\begin{example}
	The above example has a more interesting counterpart, in which we consider the left adjoint
	\[
		\begin{tikzcd}[row sep=0.0em, column sep=2.7em,  ampersand replacement=\&]
			\mathllap{\pr^{\op}\colon}\SimplexCategory^{\op}
			\arrow[r]
			\&
			\catpt^{\op}
			\\
			{[n]}
			\arrow[r, mapsto]
			\&
			{\star}
		\end{tikzcd}
	\]%
    of $\inj^{\op}$. The left and right Kan extensions of a simplicial set $X_{\bullet}\colon\SimplexCategory^\op\to\Sets$ along $\pr^{\op}$ are given by
	\begin{align*}
		\Lan_{\pr^{\op}}(X_{\bullet}) & \cong \pi_{0}(X_{\bullet})               \\
		\Ran_{\pr^{\op}}(X_{\bullet}) & \cong \ev_{0}(X_{\bullet}) \defeq X_{0}.
	\end{align*}
    We now have a much wider range of choices for the weight $W$: we may choose it to be any cosimplicial space $W^{\bullet}_{\bullet}\colon\SimplexCategory^{\op}\times\SimplexCategory\to\Sets$:
	\[
		\begin{tikzcd}[row sep={4.5em,between origins}, column sep={4.5em,between origins}, ampersand replacement=\&]
			{}
			\&
			\catpt^{\op}
			\arrow[d, "\wLan{W}_{\pr^{\op}}X_{\bullet}", dashed]
			\\
			\SimplexCategory^\op
			\arrow[ru, "\pr^{\op}"]
			\arrow[l, "W",mid vert,loop left]
			\arrow[r, "X_{\bullet}"'{name=F}]
			\&
			\Sets\mrp{.}
			\arrow[from=F,to=1-2,shorten=1.125em,Rightarrow,xshift=-0.125em,yshift=-0.25em]
		\end{tikzcd}
	\]
    For instance, taking $W=\Delta^{\bullet}$ almost gives the geometric realisation of $X_{\bullet}$:
    \[
        \wLan{\Delta^{\bullet}}_{\pr^{\op}}(X_{\bullet}) \cong \int^{[n]\in\SimplexCategory}\Delta^{n}\times X_{n},
    \]
    with the caveat that the geometric realisation involves $\lvert\Delta^{n}\rvert$, rather than $\Delta^{n}$ itself. Dually, taking again $W=\Delta^{\bullet}$ but now for a cosimplicial object $X^{\bullet}\colon\SimplexCategory\to \Sets$, we have
    \[
        \wRan{\Delta^{\bullet}}_{\pr}(X^{\bullet}) = \Tot(X_{\bullet}).
    \]
\end{example}
\begin{example}[Stalks of a sheaf {(\cite[Paragraph 6.8 and Section 7.1]{sgaiv})}]
	Let $i_{p}\colon\{p\}\longhookrightarrow X$ be the inclusion of a point into a topological space $X$. We get an induced functor
	\[
		\begin{tikzcd}[row sep=0.0em, column sep=2.7em,  ampersand replacement=\&]
			\mathllap{\Open(i_{p})\colon}\Open(X)
			\arrow[r]
			\&
			\Open(\{p\})
			\\
			U
			\arrow[r, mapsto]
			\&
			i_{p}^{-1}(U)\mrp{.}
		\end{tikzcd}
	\]%
	Considering now left Kan extensions along the opposite of $\Open(i_{p})$,
	\[
		\begin{tikzcd}[ampersand replacement=\&]
			{}
			\&
			\Open(\{p\})^{\op}
			\arrow[d, "\Lan_{\Open(i_{p})^{\op}}\SheafFont{F}", dashed]
			\\
			\Open(X)^{\op}
			\arrow[ru, "\Open(i_{p})^{\op}"]
			\arrow[r, "\SheafFont{F}"'{name=F}]
			\&
			\Sets\mrp{,}
			\arrow[from=F,to=1-2,shorten=1em,Rightarrow,xshift=-0.125em,yshift=-0.25em]
		\end{tikzcd}
	\]
	we obtain a functor $\Lan_{\Open(i_{p})^{\op}}\colon\PSh{X}\to \PSh{\{p\}}$, whose image at $\SheafFont{F}$ we write $\ceil{\SheafFont{F}_{p}}$ for simplicity. The restriction of this functor to $\Shv{X}$ can be identified with the stalk functor $(-)_{p}\colon\Shv{X}\to \Sets$: we have $\Open(\{p\})=\{\varnothing\longhookrightarrow\{p\}\}$ and computing the images of $\varnothing$ and $\{p\}$ under $\ceil{\SheafFont{F}_{p}}$ via the usual colimit formula for left Kan extensions gives
	\begin{align*}
		\ceil{\SheafFont{F}_{p}}(\{p\})       & \cong \colim\left(\left(\Open(\ceil{p})\downarrow\underline{\{p\}}\right)^{\op}\xlongertwoheadrightarrow{\projection^{\op}}\Open(X)^{\op}\xto {\SheafFont{F}}\Set\right),       \\
		                                      & \cong \colim_{U\ni p}(\SheafFont{F}(U)),                                                                                                                                        \\
		                                      & \cong \SheafFont{F}_{p}                                                                                                                                                         \\
		\ceil{\SheafFont{F}_{p}}(\varnothing) & \cong \colim\left(\left(\Open(\ceil{p})\downarrow\underline{\varnothing}\right)^{\op}\xlongertwoheadrightarrow{\projection^{\op}}\Open(X)^{\op}\xto {\SheafFont{F}}\Set\right), \\
		                                      & \cong \colim_{U\hookrightarrow\varnothing}(\SheafFont{F}(U)),                                                                                                                   \\
		                                      & \cong \SheafFont{F}(\varnothing).
	\end{align*}
	(in case $\SheafFont{F}$ is a sheaf, $\SheafFont{F}(\varnothing)$ is the singleton set.) Consider the same situation, but now with a weight $W\colon\Open(X)\times\Open(X)^{\op}\to \Sets$:
	\[
		\begin{tikzcd}[row sep={4.5em,between origins}, column sep={4.5em,between origins}, ampersand replacement=\&]
			{}
			\&
			\Open(\{p\})^{\op}
			\arrow[d, "\wLan{W}_{\Open(i_{p})^{\op}}\SheafFont{F}", dashed]
			\\
			\Open(X)^{\op}
			\arrow[ru, "\Open(i_{p})^{\op}"]
			\arrow[l, "W",mid vert,loop left]
			\arrow[r, "\SheafFont{F}"'{name=F}]
			\&
			\Sets\mrp{.}
			\arrow[from=F,to=1-2,shorten=1.125em,Rightarrow,xshift=-0.125em,yshift=-0.25em]
		\end{tikzcd}
	\]
	Using \cref{eq:weighted-kan-extension-formula-left}, we may compute $\wLan{W}_{\Open(i_{p})^{\op}}\SheafFont{F}\defeq\ceil{\SheafFont{F}_{p}^{[W]}}$ as the weighted coend
	\begin{align*}
		\ceil{\SheafFont{F}^{[W]}_{p}} & \defeq \wCoend[W]{U\in\Open(X)}\Hom_{\Open(X)^{\op}}\big(\Open\big(i_{p}^{\op}\big)(U),-\big)\odot\SheafFont{F}(U) \\
		                               & \cong  \wCoend{U\in\Open(X)}W^{U}_{U}\times\Hom_{\Open(X)}\big(\chi_{p}(U),-)\big)\times\SheafFont{F}(U),
	\end{align*}
	where
	\[
		\chi_{p}(U)
		=
		\begin{cases}
			\emptyset & \text{if $p\nin U$,} \\
			U         & \text{otherwise.}
		\end{cases}
	\]
	For instance, taking $W$ to be a sheaf $\SheafFont{G}$ on $X$ gives
	\[
		\SheafFont{F}^{[\SheafFont{G}]}_{p} \defeq \ceil{\SheafFont{F}^{[\SheafFont{G}]}_{p}}(\{p\}) \cong \big(\SheafFont{F}\times\SheafFont{G}\big)_{p}.
	\]
\end{example}
\subsubsection{A glance at diagonality}\label{glance_at_extradiag}
In a nutshell, ``diagonal'' category theory arises when, instead of considering a natural transformation filling a higher-dimensional cell, we consider a \emph{dinatural} one. Transformations that are more general than natural ones notoriously do not compose (see \cite{kelly1972abstract,kelly1972many} and mostly \cite{AlessioThesis} for a modern account); yet, the category theory arising from this generalisation is interesting.

For the purposes of our exposition here, left/right Kan extensions are the most interesting categorical gadget to ``diagonalise''; when this is done, they provide examples of higher arity co/ends.
\begin{definition}\label{def:diagonal-left-kan-extensions}%
    The \emph{diagonal left Kan extension} of a functor $F\colon\clC^{\op}\times\clC\to \CatFont{D}$ along a functor $K\colon\clC^{\op}\times\clC\to \CatFont{D}$ is, if it exists, the functor $\DiLan_{K}F\colon\CatFont{D}\to \CatFont{E}$ such that we have an isomorphism
	\[
		\Nat(\DiLan_{K}F,G)
		\cong
		\DiNat(F,G\circ K)\mrp{,}
		\qquad
		\begin{tikzcd}[row sep={5.4em,between origins}, column sep={5.4em,between origins}, ampersand replacement=\&]
			{}
			\&
			\CatFont{D}
			\arrow[d, "\DiLan_{K}F", dashed]
			\\
			\clC^{\op}\times\clC
			\arrow[ru, "K"]
			\arrow[r, "F"'{name=F}]
			\&
			\clE
			\arrow[from=F,to=1-2,din=0.70,shorten <=1.2em,shorten >=1.2em+0.3em]
		\end{tikzcd}
	\]
	natural in $G$.
\end{definition}
\begin{example}[Ends as diagonal left Kan extensions]\label{ex:coends-dilan}
    Classically, the left Kan extension of a functor $D\colon\clC\to\CatFont{D}$ along the terminal functor $\pr\colon\clC\longtwoheadsrightarrow\catpt$ may be canonically identified with its colimit. Passing to diagonal category theory, one obtains a similar result: the diagonal left Kan extension of a functor $D\colon\clC^{\op}\times\clC\to\CatFont{D}$ along $\pr\colon\clC^{\op}\times\clC\longtwoheadsrightarrow\catpt$ may similarly be canonically identified with its coend.
	\[
		\begin{tikzcd}[row sep={4.5em,between origins}, column sep={4.5em,between origins}, ampersand replacement=\&,cramped]
			{}
			\&
			\catpt
			\arrow[d, "\ceil{\colim(D)}", dashed]
			\\
			\clC
			\arrow[ru, "\pr"]
			\arrow[r, "D"'{name=F}]
			\&
			\clD\mrp{,}
			\arrow[from=F,to=1-2,Rightarrow,shorten=1.25em,xshift=0.1em]
		\end{tikzcd}
		\begin{tikzcd}[row sep={4.5em,between origins}, column sep={4.5em,between origins}, ampersand replacement=\&,cramped]
			{}
			\&
			\catpt
			\arrow[d, "\ceil{\int^{A}D^{A}_{A}}", dashed]
			\\
			\clC^{\op}\times\clC
			\arrow[ru, "\pr"]
			\arrow[r, "D"'{name=F}]
			\&
			\clD\mrp{.}
			\arrow[from=F,to=1-2,din=0.65,shorten <=1.0em,shorten >=1.0em+0.3em,xshift=-0.1em]
		\end{tikzcd}
	\]
\end{example}
\begin{remark}%
    While ordinary Kan extensions may be computed as co/ends, diagonal Kan extensions admit $(2,2)$-co/end formulas:%
    \footnote{%
        Upon inspection, one observes that there is a striking similarity between \cref{eq:dilankf,eq:dirankf} and \cref{eq:weighted-kan-extension-formula-left,eq:weighted-kan-extension-formula-right}. This is not a coincidence, as it is possible to show that diagonal Kan extensions are precisely $\Hom$-weighted Kan extensions):
        \begin{align*}
            \DiLan_{K}F & \cong \wCoend[\Hom_{\clC}(-,-)]{A,B\in\clC}\clD\left(K^{B}_{A},-\right)\odot      F^{A}_{B}, \\
            \DiRan_{K}F & \cong \wEnd[\Hom_{\clC}(-,-)]{A,B\in\clC}  \clD\left(-,K^{B}_{A}\right)\pitchfork F^{A}_{B}.
        \end{align*}
        This is both an analogy as well as a generalisation (\cref{ex:coends-dilan}) of the fact that co/ends are precisely $\Hom$-weighted co/limits.
    }%
    \begin{align}
        \DiLan_{K}F & \cong \pqCoend{2}{2}{A\in\clC}\clD\left(K^{A}_{A},-\right)\odot F^{A}_{A},   \label{eq:dilankf} \\
        \DiRan_{K}F & \cong \pqEnd{2}{2}{A\in\clC}\clD\left(-,K^{A}_{A}\right)\pitchfork F^{A}_{A},\label{eq:dirankf}
    \end{align}
    where the pairing in \cref{eq:dilankf} is such that $\DiLan_{K}F$ is the coend of
    \[
        (A,B)\mapsto\clD\left(K^{B}_{A},-\right)\odot F^{A}_{B}.
    \]
\end{remark}
\begin{example}[Application to $(p,p)$-dinaturality]%
    As diagonal Kan extensions along identity functors satisfy
    \begin{align*}
        \Nat\Big(\DiLan_{\id_{\CatFont{C}^{(p,p)}}}F,G\Big) &\cong \pqDiNat{p}{p}(F,G),\\
        \Nat\Big(F,\DiRan_{\id_{\CatFont{C}^{(p,p)}}}G\Big) &\cong \pqDiNat{p}{p}(F,G),
    \end{align*}
    they provide us with a tool to study $(p,p)$\hyp{}dinaturality in terms of (ordinary) naturality. A variant of this construction allowing also the case $p\neq q$ will be studied in \cref{sec:higher-arity-yoneda}.
\end{example}
\subsubsection{Weighted diagonal Kan extensions}
In the same spirit, we may combine the two perspectives found in \cref{def:weighted-co-ends,def:diagonal-left-kan-extensions}, thus obtaining notions of \emph{weighted diagonal Kan extensions}. While this topic is outside the scope of the present paper, we note that these constructions turn out to be examples of $(4,4)$-co/ends:
\[
    \begin{aligned}
        \wDiLan{W}_{K}F &\cong \pqCoend{4}{4}{A\in\clC}\left(W^{A,A}_{A,A}\times\eHom_{\clC}\left(K^{A}_{A},-\right)\right)\odot      F^{A}_{A}, \\
        \wDiRan{W}_{K}F &\cong \pqEnd{4}{4}{A\in\clC}\left(W^{A,A}_{A,A}\times  \eHom_{\clC}\left(-,K^{A}_{A}\right)\right)\pitchfork F^{A}_{A}.
    \end{aligned}
    \qquad
	\begin{tikzcd}[row sep={5.4em,between origins}, column sep={5.4em,between origins}, ampersand replacement=\&]
		{}
		\&
		\CatFont{D}
		\arrow[d, "\DiLan_{K}F", dashed]
		\\
		\clC^{\op}\times\clC
		\arrow[l, "W",mid vert,loop left]
		\arrow[ru, "K"]
		\arrow[r, "F"'{name=F}]
		\&
		\clE\mrp{.}
		\arrow[from=F,to=1-2,din=0.65,shorten <=1.3em,shorten >=1.3em+0.3em]
	\end{tikzcd}
\]
At this point, it is evident that the list of examples is virtually endless. We plan to dedicate separate works \cite{extradiag,weighend} to a torough investigation of the topic.
\subsubsection{Daydreaming about operads}
Day convolution was introduced by B.\ Day in \cite{day:thesis,day:report}, in order to classify monoidal structures on the category $\PSh{\clC}$ of presheaves on $\clC$. Day proved that $\PSh{\clC}$ can be turned into a monoidal category in as many ways as $\clC$ can be turned into a pseudomonoid in the bicategory $\Prof$ of profunctors.\footnote{More formally, let $S : \Cat \to \Cat$ be the 2-monad of pseudomonoids; let $\tilde S : \Prof \to \Prof$ be the lifting of $S$ to the bicategory of profunctors (i.e.\ to the Kleisli bicategory of the presheaf construction $\PShf$); then, given an object $\clC$ of $\Cat$, there is a bijection between pseudo-$S$-algebra structures on $\PSh{\clC}$ and pseudo-$\tilde S$-algebras on $\clC$, as an object of $\Prof$.}

We now propose an analogue of this framework based on higher arity coends: let $(\clC,\otimes,\Unit)$ be a monoidal category, and let $\clK\defeq \PSh{\clC}$. Higher arity Day convolution is defined as a certain family of functors $\otimesDayN{n} : \clK^n\to \clK$.
\begin{definition}\label{nn_day_convolution}
	The \textbf{Day $(n,n)$-convolution} of an $n$-tuple of presheaves $\tpl{\SheafFont{F}}$ is the presheaf
	\[\otimesDayN{n}(\tpl{\SheafFont{F}})\colon\clC^{\op}\to \Sets\]
	defined at $A\in\clC_o$ as the $(n,n)$-coend
	\[\otimesDayN{n}(\tpl{\SheafFont{F}}) \defeq A \mapsto \pqCoend{n}{n}{A\in\clC}\SheafFont{F}_{1}(A)\times\cdots\times\SheafFont{F}_{n}(A)\times\clC\left(-,A^{\otimes n}\right),\]
	where $A^{\otimes n}$ is shorthand for the $n$-fold tensor product of $A$ with itself.
\end{definition}
\begin{example}[Day convolution operad]\label{the-day-higher-arity-convolution-operad}%
	The \textbf{Day convolution operad associated to $(\clC,\otimes,\Unit)$} is the symmetric operad $\DayOperad$ whose set of generating operations (see \cite[Section 1.2.5]{fresse-operads}) is given by $\{\id,\otimesDayN{2},\otimesDayN{3},\ldots,\otimesDayN{n},\ldots\}$.
\end{example}
\begin{remark}[Unwinding \cref{the-day-higher-arity-convolution-operad}]%
	We spell out in detail the first four sets of $n$-ary operations of $\DayOperad$:
	\begin{align*}
		\DayOperad_{1} & = \{\id\}                                                                                                                                                            \\
		\DayOperad_{2} & = \{\otimesDayN{2}(-,-)\}                                                                                                                                            \\
		\DayOperad_{3} & = \{\otimesDayN{3}(-,-,-),\otimesDayN{2}(\otimesDayN{2}(-,-),-),\otimesDayN{2}(-,\otimesDayN{2}(-,-))\}                                                              \\
		\DayOperad_{4} & = \{\otimesDayN{4}(-,-,-,-),\otimesDayN{2}(\otimesDayN{3}(-,-,-),-),\otimesDayN{2}(-,\otimesDayN{3}(-,-,-)),\otimesDayN{2}(\otimesDayN{2}(-,-),\otimesDayN{2}(-,-)), \\
                       & \hspace{16.0pt}{\otimesDayN{3}(\otimesDayN{2}(-,-),-,-)},\otimesDayN{3}(-,\otimesDayN{2}(-,-),-),\otimesDayN{3}(-,-,\otimesDayN{2}(-,-))\}
	\end{align*}
	All in all, the set $\mathsf{Day}_n$ can be succinctly described by the formula
    \[\mathsf{Day}_n = \{\otimesDayN{n}\} \cup \sum_{p_1+\dots + p_k=n} \otimesDayN{k} \circ (\mathsf{Day}_{p_1} \times \dots \times \mathsf{Day}_{p_k}),\]
    while the operadic composition of $\DayOperad$ is defined via ``grafting'' in the usual way:
	\begin{diagram*}
		\begin{tikzcd}[row sep=0.0em, column sep=2.7em,  ampersand replacement=\&]
			\DayOperad_{n}\times\DayOperad_{k_{1}}\times\cdots\times\DayOperad_{k_{n}}
			\arrow[r]
			\&
			\DayOperad_{\sum k_i}
			\\
			{(\theta;\theta_{1},\ldots,\theta_{k})}
			\arrow[r, mapsto]
			\&
			{\theta(\theta_{1}(-_{1},\ldots,-_{k_{1}}),\ldots,\theta_{k}(-_{1},\ldots,-_{k_{n}}))\mrp{.}}
		\end{tikzcd}
	\end{diagram*}%
\end{remark}
We refrain from going further in the analysis of the properties of the family of $n$-ary Day convolutions; it seems to the authors this constitutes a mildly interesting object, especially because as already said, convolution monoidal structures on a presheaf category $[\clA^\op,\Set]$ correspond to \emph{promonoidal} structures on $\clA$; similarly here, the full assignment of the $n$-ary Day convolutions gives rise to a family of profunctors 
\[ \fkp_n : \clA^n \pto \clA. \]
What are the properties of this family of profunctors? We leave this as an open question. Conjecturally, this family is akin to an `unbiased' version of promonoidal structure on $\clA$, and should determine some sort of correspondence \emph{à la} Day; yet, the Day $(2,2)$-convolution evidently is not the usual Day convolution on $\PSh{\clA}$, due to the difference between a $(2,2)$-coend and an iterated coend.
\section{Kusarigamas and twisted arrow categories}\label{sec:higher-arity-yoneda}
The aim of this section is to introduce and study a fundamental computational tool that will endow higher arity coends with a fairly rich calculus, serving as the next best substitute to a ``higher arity Yoneda lemma'', whose most naïve formulation turns out to be false. Generalising a construction of Street--Dubuc introduced in \cite[Theorem 2]{dinatural-transformations}, we introduce in \cref{def:co-kusarigama} functors
\begin{align*}
	\ckgfpq{p}{q} & : \Cat\big(\pq{\clC},\clD\big)\to \Cat\big(\pq{\clC}[q][p],\clD\big),   \\
	\kgfpq{q}{p}  & : \Cat\big(\clC^{(q,p)},\clD\big)\to \Cat\big({\clC}^{(p,q)},\clD\big),
\end{align*}
which we dub \emph{co/kusarigama}. These constructions also generalise the product-hom functor of \cref{prod_hom}, in the sense that
\[ \hom_{\Pi,p,q}(\uA,\uB)\cong\ckgpq{p}{q}{\pt}^{\uA}_{\uB}, \]
where $\pt$ is the terminal functor. Co/kusarigama allow us to pass from dinaturality to naturality, underpinning a number of results in the theory of higher arity co/ends, such as the construction of higher arity twisted arrow categories.

Overall, the entire structure of this section concentrates on studying the properties of the functors $\kgfpq{q}{p}$ and $\ckgfpq{p}{q}$, which may be regarded as
\begin{enumerate}
	\item Universal objects among $(p,q)$\hyp{}dinatural transformations, through which all other $(p,q)$\hyp{}dinaturals factor (\cref{def:co-kusarigama,rem:unwinding-co-kusarigama}):
	      \[
            \Nat\left(\ckgpq{p}{q}{F},G\right)
            \cong
            \pqDiNat{p}{q}\left(F,G\right)
            \cong
            \Nat\left(F,\kgpq{p}{q}{G}\right);
          \]
	\item Functors that can be inductively defined through suitable Kan extensions (\cref{higher-arity-co-kusarigama-from-1-1-co-kusarigama}), starting from the case $\typepq{1}{1}$:
	      \[
		      \ckgpq{p}{q}{F} \cong \Lan_{\Delta_{q,p}}\left(\ckgb{\Delta_{p,q}^{*}(F)}\right),
		      \qquad
		      \kgpq{p}{q}{G} \cong \Ran_{\Delta_{p,q}}\left(\kgb{\Delta_{q,p}^{*}(G)}\right).
	      \]
	\item \say{Twisted versions} of $F$ and $G$, which may be computed by use of a similar formula as the one computing co/ends as co/limits via the twisted arrow category (\cref{kusarigama-twisted-categories}).
\end{enumerate}
Finally, the paramount property of the co/kusarigama functors is that given a category $\clC$, the category of elements of $\ckgpq{p}{q}{\pt}$, where $\pt : \pq{\clC} \to \Set$ is the terminal presheaf, is the universal fibration needed to build a higher-arity version of the \emph{twisted arrow category} (i.e., the category of elements of $\hom_\clC$): we study this construction in \cref{higher-arity-twisted-arrow-categories}. This makes it possible to express the $(p,q)$\hyp{}co/end of a diagram $G : \pq{\clC} \to \clD$ as a co/limit over the $(p,q)$\hyp{}twisted arrow category of $\clC$.
\subsection{Co/kusarigama: basic definitions}\label{rewriting}
Let $\clC$ and $\clD$ be categories.
\begin{definition}\label{def:co-kusarigama}
	Let $F$ and $G$ be functor from $\clC$ to $\clD$ of types $\typepq{p}{q}$ and $\typepq{q}{p}$.
	\begin{enumtag}{ck}
		\item\label{def:kusarigama-ck}The \emph{kusarigama}%
		\footnote{
			A \emph{kusarigama} (\jap{鎖鎌}) is a Japanese compound weapon made of a sickle (\emph{kama}) and a blunt weight (\emph{fundo}) attached to the opposite ends of a chain (\emph{kusari}).
			The weight was used to disarm the opponent by entangling their sword in the chain, or as a single weapon; disarmed or damaged the opponent, the sickle was then used to deliver the final, fatal strike. Kusarigamas were probably adapted from an old farming tool, and first adopted by Koga ninjas as a fast, compact weapon; its use then spread to tactic-oriented esoteric weaponry schools like Shinkage-ry\=u and Sui\=o-ry\=u. See \cite{vaga} for more information.

            The reason for this terminological choice is the following: as proved in \cref{con:constructing-cokusarigama}, $\ckgpq{p}{q}{F}$ is computed via the $(p,q)$-coend formula
            \[
                \ckgpq{p}{q}{F} \cong \pqCoend{p}{q}{A\in\clC}\left(\ph^{-}_{A_{1}}\times\cdots\times\ph^{-}_{A_{q}}\times\ph^{A_{1}}_{-}\times\cdots\times\ph^{A_{p}}_{-}\right)\odot F^{\bsA}_{\bsA}.
            \]
            In this equation, we have a sickle $\pqEnd{p}{q}{}$ connected to the weight $F$ by the chain $\ph^{-}_{A_{1}}\times\cdots\times\ph^{-}_{A_{q}}\times\ph^{A_{1}}_{-}\times\cdots\times\ph^{A_{p}}_{-}$ of hom-functors.
		} %
		of $G$ is, if it exists, the object
		\[\kgpq{q}{p}{G}\colon\pq{\clC}\to \clD\]
		of $\Cat(\pq{\clC},\clD)$ representing the functor
		\[\pqDiNat{p}{q}(-,G)\colon\Cat(\pq{\clC},\clD)\to \Sets.\]
		\item\label{def:cokusarigama-ck}The \emph{cokusarigama} of $F$ is, if it exists, the object
		\[\ckgpq{p}{q}{F}\colon\pq{\clC}[q][p]\to \clD\]
		of $\Cat(\pq{\clC},\clD)$ corepresenting the functor
		\[\pqDiNat{p}{q}(F,-)\colon\Cat(\pq{\clC}[q][p],\clD)\to \Sets.\]
	\end{enumtag}
\end{definition}
\begin{remark}\label{co-kusarigama-from-dinat-to-nat}
	Thus, co/kusarigama are defined by the following relations:
	\begin{align*}
		\Nat\left(\ckgpq{p}{q}{F},-\right) & \cong \pqDiNat{p}{q}(F,-), \\
		\Nat\left(-,\kgpq{p}{q}{G}\right)  & \cong \pqDiNat{p}{q}(-,G).
	\end{align*}
\end{remark}
It is crucial to focus on the exact way in which the types of $F,G$, and of $\ckgpq{p}{q}{F},\kgpq{p}{q}{G}$ interchange: asking that $F,G$ be of type of types $\typepq{p}{q}$ and $\typepq{q}{p}$ is the only possible choice for the three objects $\Nat(\ckgpq{p}{q}{F},G)$, $\Nat(F,\kgpq{p}{q}{G})$ and $\pqDiNat{p}{q}(F,G)$ to exist, according to our \cref{def:p-q-dinatural-transformation}.

This means that $\kgfpq{p}{q}, \ckgfpq{p}{q}$ are candidates to be functors
\[\ckgfpq{p}{q} : \Cat(\pq{\clC},\clD) \to \Cat(\pq{\clC}[q][p],\clD) \qquad \kgfpq{p}{q} : \Cat(\pq{\clC}[q][p],\clD) \to \Cat(\pq{\clC},\clD) \]
Among many other properties, we prove in \cref{properties-of-co-kusarigama} that these correspondences are indeed functors.
\begin{remark}[Unwinding \cref{def:co-kusarigama}]\label{rem:unwinding-co-kusarigama}
	The co/representability conditions defining co/kusarigama unwind as the following universal properties:
	\begin{enumtag}{uk}
		\item The \emph{cokusarigama} of a functor $F : \pq{\clC}\to \clD$ is, if it exists, the pair $(\ckgpq{p}{q}{F},\eta)$ with
		\[
			\ckgpq{q}{p}{F} : \pq{\clC}[q][p]\to \clD
		\]
		a functor of type $\typepq{p}{q}$, and
		\[
			\eta : F\din\ckgpq{p}{q}{F}
		\]
		a $(p,q)$\hyp{}dinatural transformation satisfying the following universal property:

		\begin{quote}
			$(\star)$ Given a $(p,q)$\hyp{}dinatural transformation $\theta : F\din G$, there exists a unique natural transformation $\ckgpq{p}{q}{F}\Longrightarrow G$ making the diagram
			\[
				\begin{tikzcd}[row sep={5.4em,between origins}, column sep={5.4em,between origins},  ampersand replacement=\&]
					\ckgpq{p}{q}{F}
					\arrow[rd,"\exists!",Rightarrow]
					\&
					\\
					F
					\arrow[r,"\theta"' ,din=0.975, shorten >=+0.1em]
					\arrow[u,"\eta",    din=0.975, shorten >=+0.1em]
					\&
					G
				\end{tikzcd}
			\]
			commute.
		\end{quote}
		\item The \emph{kusarigama} of a functor $G : \pq{\clC}[q][p]\to \clD$ is, if it exists, the pair $(\kgpq{q}{p}{G},\epsilon)$ with
		\[
			\kgpq{q}{p}{G} : \pq{\clC}\to \clD
		\]
		a functor of type $\typepq{p}{q}$, and
		\[
			\epsilon : \kgpq{q}{p}{G}\din G
		\]
		a $(p,q)$\hyp{}dinatural transformation satisfying the following universal property:

		\begin{quote}
			$(\star)$ Given a $(p,q)$\hyp{}dinatural transformation $\theta :  F\din G$, there exists a unique natural transformation $F\Longrightarrow\kgpq{q}{p}{G}$ making the diagram
			\[
				\begin{tikzcd}[row sep={5.4em,between origins}, column sep={5.4em,between origins},  ampersand replacement=\&]
					\&
					\kgpq{q}{p}{G}
					\arrow[d,"\epsilon",din=0.975, shorten >=+0.1em]
					\\
					F
					\arrow[r,"\theta"'  ,din=0.975, shorten >=+0.1em]
					\arrow[ru,"\exists!",Rightarrow]
					\&
					G
				\end{tikzcd}
			\]
			commute.
		\end{quote}
	\end{enumtag}
\end{remark}
\begin{notation}%
	Given tuples $\uA,\uC\in\clC^{-p}=(\clC^p)^\op, \uB,\uD\in\clC^q$ we make use of the notation
	\begin{align*}
		\hom_{\pq{\clC}}((\uA,\uB),(\uC,\uD)) & \defeq \ph^{(\uA,\uB)}_{(\uC,\uD)},
	\end{align*}
	as well as of the equalities
	\begin{align*}
		\ph^{(\uA,\uB)}_{(\uC,\uD)} & \defeq \ph^{\uC}_{\uA}\times\ph^{\uB}_{\uD} = \ph^{C_{1}}_{A_{1}}\times\cdots\times\ph^{C_{p}}_{A_{p}}\times\ph^{B_{1}}_{D_{1}}\times\cdots\times\ph^{B_{q}}_{D_{q}}.
	\end{align*}
\end{notation}
\begin{construction}[Constructing cokusarigama]\label{con:constructing-cokusarigama}%
	Suppose that $\clD$ is cocomplete. Then
    \begin{equation}\label{eq:constructing-cokusarigama}
		\pqCoend{p}{q}{A\in\clC}\left(\ph^{-}_{\bsA}\times\ph^{\bsA}_{-}\right)\odot F^{\bsA}_{\bsA}
    \end{equation}
	meaning the $(p,q)$\hyp{}coend of
	\[
		\begin{tikzcd}[row sep=0.0em, column sep=2.7em,  ampersand replacement=\&]
			\pq{\clC}
			\arrow[r]
			\&
			\Cat(\pq{\clC}[q][p],\clD)
			\\
			(\uA,\uB)
			\arrow[r, mapsto]
			\&
			\hom_{\pq{\clC}[q][p]}\left((\uB,\uA);(-,-)\right)\odot F^{\uA}_{\uB},
		\end{tikzcd}
	\]
	satisfies the universal property in \cref{def:cokusarigama-ck}.
\end{construction}
\begin{proof}%
	The proof is merely a formal manipulation:
	\begin{align*}
		\pqDiNat{p}{q}(F,G) & \cong  \pqEnd{q}{p}{X\in\clC}\Hom_{\clD}\left(F^{\bsX}_{\bsX},G^{\bsX}_{\bsX}\right)                                                                                  \\
		                    & \cong  \pqEnd{q}{p}{X\in\clC}\Hom_{\clD}\left(F^{\bsX}_{\bsX},\int_{\uA,\uB\in\clC}\left(\ph^{\uA}_{\bsX}\times\ph^{\bsX}_{\uB}\right)\pitchfork G^{\uA}_{\uB}\right) \\
		                    & \cong  \pqEnd{q}{p}{X\in\clC}\int_{\uA,\uB\in\clC}\Hom_{\clD}\left(F^{\bsX}_{\bsX},\left(\ph^{\uA}_{\bsX}\times\ph^{\bsX}_{\uB}\right)\pitchfork G^{\uA}_{\uB}\right) \\
		                    & \cong  \pqEnd{q}{p}{X\in\clC}\int_{\uA,\uB\in\clC}\Hom_{\clD}\left(\left(\ph^{\uA}_{\bsX}\times\ph^{\bsX}_{\uB}\right)\odot F^{\bsX}_{\bsX},G^{\uA}_{\uB}\right)      \\
		                    & \cong  \int_{\uA,\uB\in\clC}\pqEnd{q}{p}{X\in\clC}\Hom_{\clD}\left(\left(\ph^{\uA}_{\bsX}\times\ph^{\bsX}_{\uB}\right)\odot F^{\bsX}_{\bsX},G^{\uA}_{\uB}\right)      \\
		                    & \cong  \int_{\uA,\uB\in\clC}\Hom_{\clD}\left(\pqCoend{p}{q}{X\in\clC}\left(\ph^{\uA}_{\bsX}\times\ph^{\bsX}_{\uB}\right)\odot F^{\bsX}_{\bsX},G^{\uA}_{\uB}\right)    \\
		                    & \defeq \int_{\uA,\uB\in\clC}\Hom_{\clD}\left(\ckgpq{p}{q}{F}^{\uA}_{\uB},G^{\uA}_{\uB}\right)                                                                         \\
		                    & \cong \Nat\left(\ckgpq{p}{q}{F},G\right).\qedhere
	\end{align*}
\end{proof}
\begin{construction}[Constructing Kusarigamas]\label{con:constructing-kusarigama}%
	Suppose that $\clD$ is complete. Then
	\[
		\pqEnd{q}{p}{A\in\clC}\left(\ph^{\bsA}_{-}\times\ph^{-}_{\bsA}\right)\pitchfork G^{\bsA}_{\bsA},
	\]%
	meaning the $(q,p)$-end of
	\begin{diagram}\label{diagram:constructions-of-kusarigama}
		\begin{tikzcd}[row sep=0.0em, column sep=2.7em,  ampersand replacement=\&]
			\pq{\clC}[q][p]
			\arrow[r]
			\&
			\Cat(\pq{\clC},\clD)
			\\
			(\uA,\uB)
			\arrow[r, mapsto]
			\&
			\hom_{\pq{\clC}[q][p]}\left((\uA,\uB);(-,-)\right)\pitchfork G^{\uA}_{\uB},
		\end{tikzcd}
	\end{diagram}
	satisfies the universal property in \cref{def:cokusarigama-ck}.
\end{construction}
\begin{proof}%
	While this is dual to \cref{con:constructing-cokusarigama}, we register its derivation below for the sake of completeness.
	\begin{align*}
		\pqDiNat{p}{q}(F,G) & \cong  \pqEnd{q}{p}{X\in\clC}\Hom_{\clD}\left(F^{\bsY}_{\bsX},G^{\bsX}_{\bsY}\right)                                                                                  \\
		                    & \cong  \pqEnd{q}{p}{X\in\clC}\Hom_{\clD}\left(\int^{\uA,\uB\in\clC}\left(\ph^{\bsY}_{\uA}\times\ph^{\uB}_{\bsX}\right)\odot F^{\uA}_{\uB},G^{\bsX}_{\bsY}\right)      \\
		                    & \cong  \pqEnd{q}{p}{X\in\clC}\int_{\uA,\uB\in\clC}\Hom_{\clD}\left(\left(\ph^{\bsY}_{\uA}\times\ph^{\uB}_{\bsX}\right)\odot F^{\uA}_{\uB},G^{\bsX}_{\bsY}\right)      \\
		                    & \cong  \pqEnd{q}{p}{X\in\clC}\int_{\uA,\uB\in\clC}\Hom_{\clD}\left(F^{\uA}_{\uB},\left(\ph^{\bsY}_{\uA}\times\ph^{\uB}_{\bsX}\right)\pitchfork G^{\bsX}_{\bsY}\right) \\
		                    & \cong  \int_{\uA,\uB\in\clC}\pqEnd{q}{p}{X\in\clC}\Hom_{\clD}\left(F^{\uA}_{\uB},\left(\ph^{\bsY}_{\uA}\times\ph^{\uB}_{\bsX}\right)\pitchfork G^{\bsX}_{\bsY}\right) \\
		                    & \cong  \int_{\uA,\uB\in\clC}\Hom_{\clD}\left(F^{\uA}_{\uB},\pqEnd{q}{p}{X\in\clC}\left(\ph^{\bsY}_{\uA}\times\ph^{\uB}_{\bsX}\right)\pitchfork G^{\bsX}_{\bsY}\right) \\
		                    & \defeq \int_{\uA,\uB\in\clC}\Hom_{\clD}\left(F^{\uA}_{\uB},\kgpq{p}{q}{G}^{\uA}_{\uB}\right)                                                                          \\
		                    & \cong \Nat(F,\kgpq{p}{q}{G}).\qedhere
	\end{align*}
\end{proof}%
\begin{notation}[$(1,1)$-kusarigama]
	For the sake of brevity, we often write $\kg{D}$ and $\ckg{D}$ for $\kgpq{1}{1}{D}$ and $\ckgpq{1}{1}{D}$, respectively.
\end{notation}
\begin{proposition}[Properties of co/kusarigama]\label{properties-of-co-kusarigama}
	Let $D,F,G :\pq{\clC}\rightrightarrows\clD$ be functors, where $\clD$ is a bicomplete category.
	\begin{enumtag}{pk}
		\item\label{functoriality-of-co-kusarigama}\SloganFont{Functoriality}The assignments $D\mapsto\kg{D},\ckg{D}$ define functors
		\begin{align*}
			\ckgfpq{p}{q} & : \Cat\big(\pq{\clC},\clD\big)\to \Cat\big(\pq{\clC}[q][p],\clD\big), \\
			\kgfpq{p}{q}  & : \Cat\big(\pq{\clC},\clD\big)\to \Cat\big(\pq{\clC}[q][p],\clD\big).
		\end{align*}
		\item\label{adjointness-of-co-kusarigama}\SloganFont{Adjointness}We have an adjunction
		\begin{equation*}
			\begin{tikzcd}
				\Cat\big(\pq{\clC},\clD\big)
				\arrow[r, "\ckgfpq{p}{q}"{name=F}, shift left=1.75] &
				\Cat\left(\pq{\clC}[q][p],\clD\right).
				\arrow[l, "\kgfpq{q}{p}"{name=G}, shift left=1.75]
				\arrow[phantom, from=F, to=G, "\dashv" rotate=-90]
			\end{tikzcd}
		\end{equation*}
		\item\label{commutativity-of-co-kusarigama-with-homs}\SloganFont{Commutativity with homs} Let $F : \pq{\clC} \to \clD$ be a functor, and let us consider the functors
		\begin{align*}
			\clD(F,1) : \clD \to \Cat(\pq{\clC}[q][p],\Set), D & \mapsto \big((\uA,\uB)\mapsto\clD\left(F^{\uA}_{\uB}, D\right)\big), \\
			\clD(1,F) : \clD^\op \to \Cat(\pq{\clC},\Set), D   & \mapsto \big((\uA,\uB)\mapsto\clD\left(D,F^{\uA}_{\uB}\right)\big),
		\end{align*}
		then the diagrams
		\begin{center}
			\begin{adjustbox}{max height=0.5\textheight, max width=0.875\textwidth}
				\parbox{\linewidth}{
					\[
						\begin{tikzcd}[row sep={4.5em,between origins}, column sep={5.0em,between origins},  ampersand replacement=\&]
							\&\clD\ar[dr, "{\clD(F,1)}"]\ar[dl, "{\clD(\ckgpq{p}{q}F,1)}"']\& \\
							\Cat(\pq{\clC}[q][p],\Set) \&\&\ar[ll,"\kgfpq{q}{p}"] \Cat(\pq{\clC},\Set)
						\end{tikzcd}\qquad
						\begin{tikzcd}[row sep={4.5em,between origins}, column sep={5.0em,between origins},  ampersand replacement=\&]
							\&\clD \ar[dr, "{\clD(1,F)}"]\ar[dl, "{\clD(1,\kgpq{p}{q}F)}"'] \& \\
							\Cat(\pq{\clC},\Set) \&\&\ar[ll,"\kgfpq{p}{q}"] \Cat(\pq{\clC}[q][p],\Set)
						\end{tikzcd}
					\]
				}
			\end{adjustbox}
		\end{center}
		commute:
		\[ 	\clD(\ckgpq{p}{q}{F},1) \cong \kgpq{q}{p}{\clD(F,1)} \qquad \qquad
			\clD(1,\kgpq{p}{q}{F}) \cong \kgpq{p}{q}{\clD(1,F)}.\]
		\item\label{co-limits-of-kusarigama}\SloganFont{Limits of kusarigama} Let $F : \pq{\clC} \to \clD$ be a functor; we have functorial isomorphisms
		\[
			\pqEnd{p}{q}{A\in\clC}F^{\bsA}_{\bsA}   \cong \lim\left(\kgpq{p}{q}{F}\right),   \qquad
			\pqCoend{p}{q}{A\in\clC}F^{\bsA}_{\bsA} \cong \colim\left(\ckgpq{q}{p}{F}\right).
		\]
		\item\label{higher-arity-co-kusarigama-from-1-1-co-kusarigama}\SloganFont{Higher arity co/kusarigama from $(1,1)$-co/kusarigama}The cokusarigama
		\[\ckgpq{p}{q}{F}\colon\pq{\clC}[q][p]\to \clD\]
		of a functor $F\colon\pq{\clC}\to \clD$ is the left Kan extension of the $(1,1)$-cokusarigama of $\pqdiag^{*}(F)$ along $\Delta_{q,p}$:
		\[
			\ckgpq{p}{q}{F}
			=
			\Lan_{\Delta_{q,p}}\left(\ckgb{\Delta_{p,q}^{*}(F)}\right)
			\quad
			\begin{tikzcd}[row sep={6.3em,between origins}, column sep={8.1em,between origins},  ampersand replacement=\&]
				{}
				\&
				\pq{\clC}[q][p]
				\arrow[d, "\ckgpq{p}{q}{F}", dashed]
				\\
				\clC^{\op}\times\clC
				\arrow[ru, "\Delta_{q,p}"]
				\arrow[r, "\ckgb{\pqdiag^{*}(F)}"'{name=F}]
				\&
				\clD.
				\arrow[from=F,to=1-2,shorten=2.0em,Rightarrow,xshift=-0.0em,yshift=-0.25em]
			\end{tikzcd}
		\]
		Moreover, if $\clC$ has finite products and finite coproducts, then $\ckgpq{p}{q}{-}$ factors as
		\[
			\begin{tikzcd}[row sep={10.4em,between origins}, column sep={10.4em,between origins},  ampersand replacement=\&]
				\Cat\left(\pq{\clC},\clD\right)
				\arrow[r, "\pqdiag^{*}"]
				\&
				\Cat\left(\clC^{\op}\times\clC,\clD\right)
				\arrow[r, "\ckgf"]
				\&
				\Cat\left(\clC^{\op}\times\clC,\clD\right)
				\arrow[r, "\left(\prodbi^{q,p}\right)^{*}"]
				\&
				\Cat\left(\pq{\clC},\clD\right).
			\end{tikzcd}
		\]
		Dually, the kusarigama
		\[\kgpq{q}{p}{G}\colon\pq{\clC}\to \clD\]
		of $G\colon\pq{\clC}[q][p]\to \clD$ is the right Kan extension of the $(1,1)$-kusarigama of $\Delta_{q,p}^{*}(G)$ along $\pqdiag$:
		\[
			\kgpq{q}{p}{G}
			=
			\Ran_{\Delta_{p,q}}\left(\kgb{\Delta_{q,p}^{*}(G)}\right)
			\quad
			\begin{tikzcd}[row sep={6.3em,between origins}, column sep={8.1em,between origins},  ampersand replacement=\&]
				{}
				\&
				\pq{\clC}
				\arrow[d, "\kgpq{q}{p}{G}", dashed]
				\\
				\clC^{\op}\times\clC
				\arrow[ru, "\pqdiag"]
				\arrow[r, "\kgb{\Delta_{q,p}^{*}(G)}"'{name=F}]
				\&
				\clD.
				\arrow[from=F,to=1-2,shorten=2.0em,Leftarrow,xshift=-0.0em,yshift=-0.25em]
			\end{tikzcd}
		\]
		Moreover, if $\clC$ has finite products and finite coproducts, then $\kgpq{q}{p}{-}$ factors as
		\[
			\begin{tikzcd}[row sep={10.4em,between origins}, column sep={10.4em,between origins},  ampersand replacement=\&]
				\Cat\left(\pq{\clC},\clD\right)
				\arrow[r, "\Delta_{q,p}^{*}"]
				\&
				\Cat\left(\clC^{\op}\times\clC,\clD\right)
				\arrow[r, "\kgf"]
				\&
				\Cat\left(\clC^{\op}\times\clC,\clD\right)
				\arrow[r, "\left(\M^{p,q}\right)^{*}"]
				\&
				\Cat\left(\pq{\clC},\clD\right).
			\end{tikzcd}
		\]
		In fact, the adjunction yielding $\textcolor{OIblue}{\ckgfpq{p}{q}}\dashv\textcolor{OIvermillion}{\kgfpq{p}{q}}$ can be extended as in the following diagram of adjunctions:
		\[
			\begin{tikzcd}[row sep={11.7em,between origins}, column sep={11.7em,between origins},  ampersand replacement=\&]
				\left[\pq{\clC},\clD\right]
				\arrow[r,"\Lan_{\W_{p,q}}"{name=1},                        shift left=14]
				\arrow[r,"\W_{p,q}^{*}"{name=2,description},  shift left=7,leftarrow]
				\arrow[r,"\pqdiag^{*}"{name=3,description},OIblue]
				\arrow[r,"\M_{p,q}^{*}"{name=4,description},  shift right=7,leftarrow,OIvermillion]
				\arrow[r,"\Ran_{\M_{p,q}}"'{name=5},                       shift right=14]
				\&
				\left[\clC^{\op}\times\clC,\clD\right]
				\arrow[r, "\ckgf"{name=ckg},shift left=2,OIblue]
				\arrow[r, "\kgf"'{name=kg},shift right=2,leftarrow,OIvermillion]
				\&
				\left[\clC^{\op}\times\clC,\clD\right]
				\arrow[r,"\Lan_{\W_{q,p}}"{name=1'}, shift left=14,leftarrow]
				\arrow[r,"\W_{q,p}^{*}"{name=2',description},  shift left=7,OIblue]
				\arrow[r,"\Delta_{q,p}^{*}"{name=3',description},leftarrow,OIvermillion]
				\arrow[r,"\M_{q,p}^{*}"{name=4',description},  shift right=7]
				\arrow[r,"\Ran_{\M_{q,p}}"'{name=5'}, shift right=14,leftarrow]
				\&
				\left[\pq{\clC}[q][p],\clD\right];
				\arrow[from=1,to=2,"\dashv"{rotate=-90},phantom]
				\arrow[from=2,to=3,"\dashv"{rotate=-90},phantom]
				\arrow[from=3,to=4,"\dashv"{rotate=-90},phantom]
				\arrow[from=4,to=5,"\dashv"{rotate=-90},phantom]
				\arrow[from=kg,to=ckg,"\dashv"{rotate=-90},phantom]
				\arrow[from=1',to=2',"\dashv"{rotate=-90},phantom]
				\arrow[from=2',to=3',"\dashv"{rotate=-90},phantom]
				\arrow[from=3',to=4',"\dashv"{rotate=-90},phantom]
				\arrow[from=4',to=5',"\dashv"{rotate=-90},phantom]
			\end{tikzcd}
		\]
	\end{enumtag}
	see \cref{quintuple-biprod-delta-pq-adjunction}.
\end{proposition}
\begin{proof}
	We often prove the statements for cokusarigama only, as the ones for kusarigama follow by an easy dualisation.

	\cref{functoriality-of-co-kusarigama}: This follows from \cite[Theorems IX.7.2 and IX.7.3]{working-categories}.

	\cref{adjointness-of-co-kusarigama}: This follows straight from the definition of co/kusarigama.

	\cref{commutativity-of-co-kusarigama-with-homs}: For the first statement, we have
	\begin{align*}
		\clD\left(1,\pqEnd{q}{p}{A\in\clC}\ph^{(A,\ldots,A)}\pitchfork F^{\bsA}_{\bsA}\right) & \cong  \pqEnd{q}{p}{A\in\clC}\clD\left(-,\ph^{(A,\ldots,A)}\pitchfork F^{\bsA}_{\bsA}\right) \\
		                                                                                    & \cong  \pqEnd{q}{p}{A\in\clC}\ph^{(A,\ldots,A)}\pitchfork\clD\left(-,F^{\bsA}_{\bsA}\right)  \\
		                                                                                    & \defeq \kgpq{p}{q}{\clD(1,F)}.
	\end{align*}

	\cref{co-limits-of-kusarigama}: We just prove the first statement, the other being a straightforward dualisation. We have
	\begin{alignat*}{2}
		\clD\left(-,\pqEnd{p}{q}{A}D^{\bsA}_{\bsA}\right) & \defeq \pqDiNat{p}{q}\left(\Delta_{\pt},h_{D}\right)
		                                                & \qquad
		                                                & \text{}                                                     \\
		                                                & \cong  \Nat\left(\Delta_{\pt},\kgpq{p}{q}{h_{D}}\right)
		                                                & \qquad
		                                                & \text{by \cref{co-kusarigama-from-dinat-to-nat},}          \\
		                                                & \cong  \Nat\left(\Delta_{\pt},h_{\kgpq{p}{q}{D}}\right)
		                                                & \qquad
		                                                & \text{by \cref{commutativity-of-co-kusarigama-with-homs},} \\
		                                                & \defeq h_{\lim(\kgpq{p}{q}{D})}.
		                                                & \qquad
		                                                & \text{}
	\end{alignat*}
	The result then follows from the Yoneda lemma.

	\cref{higher-arity-co-kusarigama-from-1-1-co-kusarigama}: We have
	\begin{alignat*}{2}
		\Nat\left(\Lan_{\Delta_{q,p}}\ckgb{\pqdiag^{*}(F)},G\right) & \defeq \Nat\left(\ckgb{\pqdiag^{*}(F)},\Delta_{q,p}^{*}(G)\right)
		                                                            & \qquad
		                                                            & \text{}                                                               \\
		                                                            & \cong  \DiNat\left(\pqdiag^{*}(F),\Delta_{q,p}^{*}(G)\right)
		                                                            & \qquad
		                                                            & \text{by \cref{co-kusarigama-from-dinat-to-nat},}                    \\
		                                                            & \cong  \pqDiNat{p}{q}\left(F,G\right)
		                                                            & \qquad
		                                                            & \text{by \cref{higher-arity-dinaturality-via-ordinary-dinaturality},} \\
		                                                            & \cong  \Nat\left(\ckgpq{p}{q}{F},G\right)
		                                                            & \qquad
		                                                            & \text{by \cref{co-kusarigama-from-dinat-to-nat} again.}
	\end{alignat*}
	The stated factorisation follows from the isomorphism $\Lan_{\Delta_{q,p}}\cong\smash{\left(\prodbi^{q,p}\right)^{*}}$ of \cref{quintuple-biprod-delta-pq-adjunction}.
\end{proof}
\subsection{Examples of co/kusarigama}\label{sec:examples-of-co-kusarigama}
\begin{example}[Cokusarigama of $\hom$ functors]%
    The computation given in the proof of \cref{ponk} generalises to show that, given $(\uA,\uB)\in \clC^{(p,q)}_o$, the cokusarigama of the functor $\ph_{(\uA,\uB)}\colon\clC^{(q,p)}\to \Sets$, which may be written as
    \begin{align*}
        \Hom_{\clC^{(p,q)}}((-,-);(\uA,\uB)) & \defeq \Hom_{\clC^{p}}(\uA,-)\times\Hom_{\clC^{q}}(-,\uB) \\
                                             & \defeq
        \ph^{A_{1}}_{-_{q+1}}
        \times
        \cdots
        \times
        \ph^{A_{p}}_{-_{q+p}}
        \times
        \ph^{-_{1}}_{B_{1}}
        \times
        \cdots
        \times
        \ph^{-_{q}}_{B_{q}},
    \end{align*}
    is given by
    \begin{align*}
        \ckgpq{q}{p}{\ph_{(\uA,\uB)}}
         & \defeq
        \int^{X\in\clC}
        \ph^{X}_{-_{p+1}}\times\cdots\times\ph^{X}_{-_{p+q}}\times\ph^{-_{1}}_{X}\times\cdots\times\ph^{-_{p}}_{X}
        \times
        \ph^{A_{1}}_{X}\times\cdots\times\ph^{A_{p}}_{X}\times\ph^{X}_{B_{1}}\times\cdots\times\ph^{X}_{B_{q}} \\
         & \cong
        \left(
        \ph^{A_{1}}_{B_{1}}\times\cdots\times\ph^{A_{1}}_{B_{q}}
        \times
        \cdots
        \times
        \ph^{A_{p}}_{B_{1}}\times\cdots\times\ph^{A_{p}}_{B_{q}}
        \right)
        \\
         & \times
        \left(
        \ph^{A_{1}}_{-_{p+1}}\times\cdots\times\ph^{A_{1}}_{-_{p+q}}
        \times
        \cdots
        \times
        \ph^{A_{p}}_{-_{p+1}}\times\cdots\times\ph^{A_{p}}_{-_{p+q}}
        \right)
        \\
         & \times
        \left(
        \ph^{-_{1}}_{B_{1}}\times\cdots\times\ph^{-_{1}}_{B_{q}}
        \times
        \cdots
        \times
        \ph^{-_{p}}_{B_{1}}\times\cdots\times\ph^{-_{p}}_{B_{q}}
        \right)
        \\
         & \times
        \left(
        \ph^{-_{1}}_{-_{p+1}}\times\cdots\times\ph^{-_{1}}_{-_{p+q}}
        \times
        \cdots
        \times
        \ph^{-_{p}}_{-_{p+1}}\times\cdots\times\ph^{-_{p}}_{-_{p+q}}
        \right)
    \end{align*}
\end{example}
\begin{example}[Co/kusarigama of constant functors]
	Let $E$ be a set and let's equally denote $E : \clC^{(p,q)} \to \Set$ the constant functor on $E$; assume $\clC$ has finite products and coproducts; then, we can compute the kusarigama of $E$ as the integral
	\begin{align*}
		\kg{\underline E} & \cong \int_A \big(\clC[\bsY|A] \times \clC[A|\bsX] \big) \pitchfork E                                  \\
		                  & \cong \int_A \textstyle\Set\Big(\clC(A,\prod X_i),\Set\big(\clC(\coprod Y_j,A),E\big)\Big)             \\
		                  & \cong \textstyle\Cat(\clC^{(p,q)},\Set)\Big(\clC(-,\prod X_i),\Set\big(\clC(\coprod Y_j,-),E\big)\Big) \\
		                  & \cong \textstyle \Set\big(\clC(Y,X), E\big),
	\end{align*}
	where $X \defeq \prod X_i, Y \defeq \coprod Y_j$.

	In particular, when $\clD=\Sets$:
	\[\kg{\one}=\int_{A\in\clC}\left[\ph_{A}\times\ph^{A},\one\right]\cong{\one}.\]
	This is in accordance with the fact that dinatural transformations to $\Delta_{\one}$ coincide with natural transformations to $\Delta_{\one}$.

    Dually,
	\begin{align}
		\ckg{E}_{\bsX}^{\bsY} & = \int^A(\ph^A)^p \times (\ph_A)^q \times E \notag                \\
		                      & \cong \int^A \clC[\bsY|A] \times \clC[A|\bsX] \times E \notag     \\
		                      & \cong \Big(\int^A \clC(Y,A) \times \clC(A,X)\Big) \times E \notag \\
		                      & \cong \clC\big(Y,X\big) \times E,
	\end{align}
	where $X \defeq \prod X_i$ and $Y \defeq \coprod Y_j$.
\end{example}
\begin{example}[The co/kusarigama of the identity functor]\label{coku_of_id}
	Let $\clC$ be a complete and cocomplete category (so that the co/ends in \cref{con:constructing-cokusarigama} and \cref{con:constructing-kusarigama} exist). 
	
	We want to compute the co/kusarigama of the identity functor $\id_{\pq{\clC}} : \pq{\clC} \to \pq{\clC}$. By virtue of the universal property of the product category $\pq{\clC}$, it is then enough to determine the functor 
	\[
	\begin{tikzcd}
		\pq{\clC} \ar[r, "\ckgpq{p}{q}{\id}"] & \pq{\clC}[q][p] \ar[r, "\pi_j"] & \clC^{\pm}
	\end{tikzcd}
	\]
	where the functor $\pi_j$ projects to the factor $\clC$ for $1\le j\le q$, and to $\clC^\op$ for $p+1\le j\le q+p$.

	In case $(p,q)=(2,1)$ one sees that for objects $(X_1,X_2,Y)$ the diagram 
	\[
	\begin{tikzcd}
        \displaystyle\coprod_{\mathclap{f\colon B\to A}}\left(\ph^{X_{1}}_{B}\times\ph^{X_{2}}_{B}\times\ph^{A}_{Y}\right)\odot B
		\ar[d, shift left=.5em, "\alpha"]
		\ar[d, shift right=.5em, "\beta"'] \\
        \displaystyle\coprod_{\mathclap{A\in\clC}}\left(\ph^{X_{1}}_{A}\times\ph^{X_{2}}_{A}\times\ph^{A}_{Y}\right)\odot A 
		\ar[d, "c"] & \left(\var[u]{X_1}{A},\var[v]{X_2}{A},\var[w]{A}{Y},a\right)\ar[d, mapsto]\\ 
        (\ph^{X_{1}}_{Y}\times\ph^{X_{2}}_{Y})\odot Y & (w\circ u,w\circ v,w\circ a)
	\end{tikzcd}	
	\]
	commutes and in fact that it is a coequaliser: every other $\zeta : \coprod_{A\in\clC} \clC(X_1,A)\times \clC(X_2,A)\times \clC(A,Y) \odot A \to E$ coequalising the pair $(\alpha,\beta)$ must factor through $\clC(X_1,Y)\times \clC(X_2,Y)\odot Y$ with a uniquely determined map.

	A standard argument, carried over the general case, to find the coequaliser defining the end and coend in \cref{con:constructing-cokusarigama} and \cref{con:constructing-kusarigama} now yields
	\[
		\ckgpq{p}{q}{\id}(\uX,\uY) = \hom_{\Pi,p,q}(\uY,\uX)\odot (\uY,\uX),\qquad\quad 
	\kgpq{p}{q}{\id}(\uX,\uY) = \hom_{\Pi,p,q}(\uX,\uY)\pitchfork (\uY,\uX).
	\]
\end{example}
\begin{remark}	
	The previous argument hides a technical point. It holds by virtue of the following fact: if two categories $\clA,\clB$ are co/tensored over $\Set$, then so is their product $\clA\times \clB$, with the component-wise action of a functor $\odot : \Set \times \clA\times\clB \to \clA\times \clB$.

	A similar result does \emph{not} hold for a generic base of enrichment.
\end{remark} 
\subsection{Higher arity twisted arrow categories}\label{higher-arity-twisted-arrow-categories}
Classically, it is possible to compute the co/end of a diagram $D\colon\CatFont{C}^{\op}\times\CatFont{C}\longrightarrow\CatFont{D}$ as the co/limit of $D$ over the \emph{twisted arrow category} $\Tw{\clC}$ of $\clC$, i.e.\ over the category of elements of the $\hom$ functor of $\CatFont{C}$. The purpose of this section is to formulate and prove an analogous description for higher arity co/ends.

In this section, we abbreviate $\ckgpq{p}{q}{\Delta_{\pt}}$ as $\ckgpq{p}{q}{\pt}$.
\begin{definition}\label{p-q-twisted-category}%
    The \emph{$(p,q)$-twisted arrow category of $\CatFont{C}$} is the category $\pqTw{p}{q}{\clC}$ defined as the category of elements of $\ckgpq{p}{q}{\pt}$:
    \begin{center}
        \begin{tikzcd}[row sep={3.6em,between origins}, column sep={7.2em,between origins}, background color=backgroundColor, ampersand replacement=\&]
            \pqTw{p}{q}{\clC}
            \ar[d,two heads]
            \ar[r, "\Forgetful_{(p,q)}",two heads]
            \&
            \pq{\clC}
            \ar[d,"\ckgpq{p}{q}{\pt}"]
            \\
            \catpt
            \ar[ur,Rightarrow, shorten <=2em, shorten >=2em]
            \ar[r,"\ceil{\pt}"',hook]
            \&
            \Set.
        \end{tikzcd}
    \end{center}%
\end{definition}
\begin{remark}[Unwinding \cref{p-q-twisted-category}]
    By the calculation in the proof of \cref{ponk}, we have $\ckgpq{p}{q}{\pt}\cong\Hom_{\Pi,p,q}$. As a result, we see that $\pqTw{p}{q}{\CatFont{C}}$ may be described as the category whose
    \begin{enumtag}{kcc}
        \item Objects are collections $\left\{f_{ij}\colon A_{i}\longrightarrow B_{j}\right\}$ of morphisms of $\clD$ with $0\leq i\leq p$ and $0\leq j\leq q$;
        \item Morphisms are collections of factorisations of the codomain through the domain, of the form
              \[
                  \begin{tikzcd}[row sep={4.5em,between origins}, column sep={4.5em,between origins}, background color=backgroundColor, ampersand replacement=\&]
                      A_{i}
                      \arrow[r,"f"]
                      \&
                      B_{j}
                      \arrow[d,"\psi_{j}"]
                      \\
                      A_{i}'
                      \arrow[u,"\phi_{i}"]
                      \arrow[r,"g"']
                      \&
                      B_{j}',
                  \end{tikzcd}
              \]
              one for each $0\leq i\leq p$ and each $0\leq j\leq q$.
    \end{enumtag}
\end{remark}
\begin{lemma}\label{pq-co-ends-as-weighted-co-limits}
    Let $D\colon\CatFont{C}^{(p,q)}\longrightarrow\CatFont{D}$ be a diagram. We have natural isomorphisms
	\[
        \pqCoend{p}{q}{A\in\CatFont{C}} D^{\bsA}_{\!\bsA} \cong \wcolim{\ckgpq{p}{q}{\pt}}(D)
        \qquad\qquad
        \pqEnd{p}{q}{A\in\CatFont{C}}	D^{\bsA}_{\!\bsA} \cong \wlim{\ckgpq{p}{q}{\pt}}(D),
	\]
	generalising the well-known isomorphisms
	\[
        \int^{A\in\CatFont{C}} D^{A}_{A}    \cong \wcolim{\hom_{\CatFont{C}}}(D)
        \qquad\qquad
        \int_{A\in\CatFont{C}} D^{A}_{A}    \cong \wlim{\hom_{\CatFont{C}}}(D),
	\]
    valid for $(p,q)=(1,1)$.
\end{lemma}
\begin{proof}%
    We have
    \begin{align*}
        \ph\left(-,\pqEnd{p}{q}{A\in\CatFont{C}}D^{\bsA}_{\!\bsA}\right) &\cong \DiNat(\Delta_{\pt},\ph_{D})\\
                                                                         &\cong \Nat(\ckgpq{p}{q}{\pt},\ph_{D})\\
                                                                         &\cong h_{\wlim{\ckgpq{p}{q}{\pt}}(D)}.
    \end{align*}
    The proof of the remaining isomorphism is formally dual to the above one.
\end{proof}
\begin{proposition}[$(p,q)$-Ends as limits, yet again]\label{p-q-ends-from-p-q-tw-c}%
	Let $D : \pq{\clC}\longrightarrow\clD$ be a functor.
    We have isomorphisms
		\begin{align*}
			\pqEnd{p}{q}{A\in\clC}D^{\bsA}_{\bsA}   &\cong \lim\Big(\pqTw{p}{q}{\clC}\xloongrightarrow{\Forgetful_{(p,q)}}\pq{\clC}\xloongrightarrow{D}\clD\Big),               \\
			\pqCoend{p}{q}{A\in\clC}D^{\bsA}_{\bsA} &\cong \colim\Big(\pqTw{p}{q}{\clC^{\op}}^{\op}\xloongrightarrow{\Forgetful_{(p,q)}}\pq{\clC}\xloongrightarrow{D}\clD\Big).
		\end{align*}
\end{proposition}
\begin{proof}%
    This result follows from \cref{pq-co-ends-as-weighted-co-limits} and the classical description of weighted colimits as conical ones \cite[Section 3.4, Equation 3.33]{kelly}.
\end{proof}
\subsection{Twisted arrow categories associated to cokusarigama}\label{kusarigama-twisted-categories}
In this short section, we give a co/comma category formula for computing co/kusarigama. These generalise the construction in \cref{higher-arity-twisted-arrow-categories} and work for arbitrary $(p,q)$. However, these turn out to be too complicated for $p,q\geq2$ as to be practically useful%
\footnote{%
	Similarly to how a morphism of $\pqTw{p}{q}{\clC}$ turned out to involve $pq$ arrows of $\clC$, unravelling the construction given in this section for arbitrary $(p,q)$ gives a category $\mathsf{Tw}^{(p,q)}(\clC)$ whose morphisms now consist of $4pq$ morphisms of $\clC$. Additionally, \emph{each of these} points now in a different directions (i.e.\ they cannot anymore be arranged as morphisms in product categories). Together, these two points make $\mathsf{Tw}^{(p,q)}(\clC)$ unusable in practice when $p$ and $q$ are too large. As a compromise, we work out the case $(p,q)=(1,1)$, which is both the simplest case as well as the most useful one.
}, %
so we restrict our attention to the case $(p,q)=(1,1)$ below. Let $F\colon\clC^{\op}\times\clC\to \clD$ be a functor and fix $A,B\in\clC_o$.
\begin{definition}\label{twckgc}%
	The \emph{twisted arrow category of $\clC$ for $(1,1)$-cokusarigama at $(A,B)$} is the category $\Twckg{A}{B}{\clC}$ defined as the category of elements of $\ph^{A}_{B}\times\ph^{A}_{-_{2}}\times\ph^{-_{1}}_{B}\times\ph^{-_{1}}_{-_{2}}$.
\end{definition}
\begin{remark}%
	Concretely, $\Twckg{A}{B}{\clC}$ may be described as the category whose
	\begin{enumtag}{kcc}
		\item Objects are squares of the form%
		\[
			\begin{tikzcd}[row sep={4.5em,between origins}, column sep={4.5em,between origins},  ampersand replacement=\&]
				X
				\arrow[r,"\phi"]
				\arrow[d,"f"']
				\&
				B
				\\
				Y
				\&
				A
				\arrow[l,"\psi"]
				\arrow[u,"g"']
			\end{tikzcd}
		\]
		with $X,Y\in\clC_o$ and $f,g,\phi,\psi\in\Mor(\clC)$;
		\item Morphisms are twisted commutative cubes
		\[
			\begin{tikzcd}[row sep={3.6em,between origins}, column sep={3.6em,between origins},  ampersand replacement=\&]
				X
				\arrow[rr,"\phi"]
				\arrow[dd,"f"']
				\&
				\&
				B
				\arrow[from=dd,"g"'description,pos=0.2]
				\&
				\\
				\&
				X'
				\arrow[rr,"\phi'"description,crossing over]
				\arrow[lu,"\xi_{1}"description]
				\&
				\&
				B
				\arrow[lu,"\xi_{3}"description]
				\\
				Y
				\arrow[rd,"\xi_{2}"description]
				\&
				\&
				A
				\arrow[ll,"\psi"',pos=0.2]
				\arrow[rd,"\xi_{4}"description]
				\&
				\\
				\&
				Y'
				\arrow[from=uu,"f'"'description,crossing over]
				\&
				\&
				A
				\arrow[ll,"\psi'"]
				\arrow[uu,"g'"']
			\end{tikzcd}
		\]%
	\end{enumtag}
\end{remark}
\begin{remark}[$\Twckg{A}{B}{\clC}$ as a generalisation of the twisted arrow category]
	The twisted arrow category of $\clC$ naturally fits inside $\Twckg{A}{B}{\clC}$:
	\[
		\begin{tikzcd}[row sep={4.5em,between origins}, column sep={4.5em,between origins},  ampersand replacement=\&]
			\textcolor{OIvermillion}{W}
			\arrow[r, "f",OIvermillion]
			\arrow[d, "p"',leftarrow,OIvermillion]
			\&
			\textcolor{OIvermillion}{X}
			\\
			\textcolor{OIvermillion}{Y}
			\arrow[r, "g"',OIvermillion]
			\&
			\textcolor{OIvermillion}{Z}
			\arrow[u, "q"',leftarrow,OIvermillion]
		\end{tikzcd}
		\begin{tikzcd}[row sep={3.6em,between origins}, column sep={3.6em,between origins},  ampersand replacement=\&]
			{}
			\arrow[r,squiggly]
			\&
			{}
		\end{tikzcd}
		\begin{tikzcd}[row sep={2.7em,between origins}, column sep={2.7em,between origins},  ampersand replacement=\&]
			\textcolor{OIvermillion}{W}
			\arrow[rr,gray!40]
            \arrow[dd,"f"',gray!40,OIvermillion]
			\&
			\&
			\textcolor{gray!40}{\bullet}
            \arrow[from=dd,gray!40]
			\&
			\\
			\&
			\textcolor{OIvermillion}{Y}
			\arrow[rr,gray!40,crossing over]
			\arrow[lu,"p"',OIvermillion]
			\&
			\&
			\textcolor{gray!40}{\bullet}
			\arrow[lu,gray!40]
			\\
			\textcolor{OIvermillion}{X}
			\arrow[rd,"q"',OIvermillion]
			\&
			\&
			\textcolor{gray!40}{\bullet}
			\arrow[ll,gray!40]
			\arrow[rd,gray!40]
			\&
			\\
			\&
			\textcolor{OIvermillion}{Z}
			\arrow[from=uu,gray!40,"g"description,OIvermillion]
			\&
			\&
			\textcolor{gray!40}{\bullet}\mrp{.}
			\arrow[ll,gray!40]
            \arrow[uu,gray!40]
		\end{tikzcd}
	\]%
	This comes from the identity $\ckg{\pt}\cong\Hom$.
\end{remark}
\begin{proposition}[Co/kusarigama as limits]%
	Given a functor $D\colon\clC^{\op}\times\clC\to \clD$, we have isomorphisms
	\begin{align*}
		\ckg{D}^{A}_{B} & \cong \colim\left(\Twckg{A}{B}{\clC}\xlongertwoheadsrightarrow{\mathrm{pr}}\clC^{\op}\times\clC\xto {D}\clD\right), \\
		\kg{D}^{A}_{B}  & \cong \lim\left(\Twckg{A}{B}{\clC}\xlongertwoheadsrightarrow{\mathrm{pr}}\clC^{\op}\times\clC\xto {D}\clD\right).
	\end{align*}
\end{proposition}
\begin{proof}%
	Firstly, observe that we may compute $\ckg{D}$ as the following weighted coend:
	\begin{align*}
		\wCoend[\ph^{A}_{-_{2}}\times\ph^{-_{1}}_{B}]{X\in\CatFont{C}}D^{X}_{X} & \cong \int^{X\in\CatFont{C}}\left(\ph^{A}_{X}\times\ph^{X}_{B}\right)\odot D^{X}_{X} \\
		                                                                        & \cong \ckg{D}^{A}_{B}.
	\end{align*}
	Now, weighted coends corepresent functors of the form $\DiNat(W,\ph^{D})$, but since
	\[ \DiNat(W,\ph^{D})\cong\Nat(\ckg{W},\ph^{D}), \]
	we see that the above weighted coend is the weighted colimit of $D$ by $\ckg{\ph^{A}_{-_{2}}\times\ph^{-_{1}}_{B}}$. From the computation in the proof of \cref{ponk}, we have $\ckg{\ph^{A}_{-_{2}}\times\ph^{-_{1}}_{B}}\cong\ph^{A}_{B}\times\ph^{A}_{-_{2}}\times\ph^{-_{1}}_{B}\times\ph^{-_{1}}_{-_{2}}$. The result then follows from the classical description of weighted colimits as conical ones \cite[Section 3.4, Equation 3.33]{kelly}.

	The second formula is proved in a dual fashion.
\end{proof}

\printbibliography
\end{document}